\documentclass[reqno]{amsart}
\usepackage{relsize}
\usepackage{scalerel}
\usepackage{stackengine,wasysym}
\usepackage{todonotes}
\usepackage{xcolor}

\usepackage{amsmath,amssymb} 
\usepackage[toc,page]{appendix}
\usepackage{verbatim}
\allowdisplaybreaks
\theoremstyle{plain}
\newtheorem{lemma}{Lemma}[section]
\newtheorem{theorem}[lemma]{Theorem}
\newtheorem{corollary}[lemma]{Corollary}

\newtheorem{proposition}[lemma]{Proposition}

\theoremstyle{remark}
\newtheorem{remark}{Remark}

\makeatletter
\newcommand*{\rom}[1]{\expandafter\@slowromancap\romannumeral #1@}
\makeatother

\def\lc{\left\lceil}   
\def\rc{\right\rceil}
\numberwithin{equation}{section}
\begin{document}

\vskip 0.125in

\title[The Effect of Rotation on the Inviscid Primitive Equations]
{On the effect of rotation on the life-span of analytic solutions to the $3D$ inviscid primitive equations}

\date{\today}

\author[T. E. Ghoul]{Tej Eddine Ghoul}
\address[T. E. Ghoul]
{Department of Mathematics  \\
New York University in Abu Dhabi  \\
Saadyat Island, Abu Dhabi  \\
 UAE.} \email{teg6@nyu.edu}

\author[S. Ibrahim]{Slim Ibrahim}
\address[S. Ibrahim]
{Department of Mathematics and Statistics  \\
University of Victoria  \\
3800 Finnerty Road, Victoria  \\
 B.C., Canada V8P 5C2.} \email{ibrahims@uvic.ca}

\author[Q. Lin]{Quyuan Lin*}\thanks{*Corresponding author. Department of Mathematics, Texas A\&M University,	College Station, TX 77840, USA. E-mail address: lyn0131@tamu.edu}
\address[Q. Lin]
{Department of Mathematics  \\
	Texas A\&M University  \\
	College Station  \\
	Texas, TX 77840, USA.} \email{lyn0131@tamu.edu}

\author[E.S. Titi]{Edriss S. Titi}
\address[E.S. Titi]
{Department of Mathematics  \\
	Texas A\&M University  \\
	College Station  \\
	Texas, TX 77840, USA.  Department of Applied Mathematics and Theoretical Physics\\ University of Cambridge\\
Wilberforce Road, Cambridge CB3 0WA, UK.
 Department of Computer Science and Applied Mathematics \\
Weizmann Institute of Science  \\
Rehovot 76100, Israel.} \email{titi@math.tamu.edu} \email{Edriss.Titi@damtp.cam.ac.uk}
\email{edriss.titi@weizmann.ac.il}

\begin{abstract}
We study the effect of the rotation on the life-span of solutions to the $3D$ hydrostatic Euler equations with rotation and the inviscid Primitive equations (PEs) on the torus. The space of analytic functions appears to be the natural space to study the initial value problem for the inviscid PEs with general initial data, as they have been recently shown to exhibit Kelvin-Helmholtz type instability. First, for a short interval of time that is independent of the rate of rotation $|\Omega|$, we establish the local well-posedness of the inviscid PEs in the space of analytic functions. In addition, thanks to a fine analysis of the barotropic and baroclinic modes decomposition, we establish two results about the long time existence of solutions. (i) Independently of $|\Omega|$, we show that the life-span of the solution tends to infinity as the analytic norm of the initial baroclinic mode goes to zero. Moreover, we show in this case that the solution of the $3D$ inviscid PEs converges to the solution of the limit system, which is governed by the $2D$ Euler equations. (ii) We show that the life-span of the solution can be prolonged unboundedly with $|\Omega|\rightarrow \infty$, which is the main result of this paper. This is established for ``well-prepared" initial data, namely, when only the Sobolev norm  (but not the analytic norm) of the baroclinic mode is small enough, depending on $|\Omega|$. Furthermore, for large $|\Omega|$ and ``well-prepared" initial data, we show that the solution to the $3D$ inviscid PEs is approximated by the solution to a simple limit resonant system with the same initial data. 
\end{abstract}

\maketitle

MSC Subject Classifications: 35Q35, 35Q86, 86A10, 76E07.\\

Keywords: inviscid primitive equations; hydrostatic Euler equations; fast rotation; limit resonant system

\section{Introduction}  
For large-scale oceanic and atmospheric dynamics, the vertical scale (a few kilometers for the ocean, 10-20 kilometers for the atmosphere) is much smaller than the horizontal scales (several thousands of kilometers). The following $3D$ viscous primitive equations (PEs) has been a standard framework for studying geostrophic adjustment of frontal anomalies in a rotating continuously stratified fluid of strictly rectilinear fronts and jets (see, e.g., \cite{BL72,GI76,GI82,HO93,HO04,KP97,PZ05,RO38} and references therein):
\begin{eqnarray}
	&&\hskip-.8in 
	\partial_t v + v\cdot \nabla v + w\partial_z v - \nu_h \Delta v - \nu_z \partial_{zz} v +\Omega v^\perp + \nabla p = 0 , \label{1}  \\
	&&\hskip-.8in  
	\partial_z p + T =0 , \label{2}  \\
	&&\hskip-.8in
	\partial_t T + v\cdot \nabla T + w\partial_z T - \kappa_h \Delta T - \kappa_z \partial_{zz} T = 0, \label{3} \\
	&&\hskip-.8in
	\nabla \cdot v + \partial_z w =0,  \label{4} 
\end{eqnarray}
setting in the horizontal channel $\big\{(x_1,x_2,z): 0\leq z\leq H, (x_1,x_2)\in \mathbb{T}^2\big\}$, subject to the following initial and boundary conditions:
\begin{eqnarray}
&&\hskip-.8in
(v,T)|_{t=0} =(v_0, T_0),  \label{OIC}\\
&&\hskip-.8in
(v_z, w, T_z)|_{z=0,H}=0,  \label{OBC-1}\\
&&\hskip-.8in
v, w, T \;\text{are periodic in} \; (x_1, x_2) \; \text{with period} \; 1. \label{OBC-2}
\end{eqnarray}
Here the horizontal velocity field $v=(v_1, v_2)$, the vertical velocity $w$, the temperature $T$, and the pressure $p$ are the unknown quantities which are functions of the independent variables $(\boldsymbol{x}',z,t) = (x_1, x_2, z,t)$. The $2D$ horizontal gradient and Laplacian are denoted by $\nabla = (\partial_{x_1}, \partial_{x_2})$ and $\Delta = \partial_{x_1x_1} + \partial_{x_2x_2}$, respectively. The non-negative constants $\nu_h, \nu_z, \kappa_h$ and $\kappa_z$ are the horizontal viscosity, the vertical viscosity, the horizontal diffusivity and the vertical diffusivity coefficients, respectively. The parameter $\Omega \in \mathbb{R}$ stands for the speed of rotation in the Coriolis force, and $v^\perp = (-v_2, v_1)$. The $3D$ viscous PEs is derived by performing a formal asymptotic limit of the small aspect ratio (the ratio of the depth or the height to the horizontal length scale) from the Rayleigh-B\'enard (Boussinesq) system, and this limit is justified rigorously first by Az\'erad and Guill\'en \cite{AG01} in a weak sense then later by Li and Titi \cite{LT18} in a strong sense with error estimates. 

The global existence of strong solutions to the $3D$ PEs with full viscosity and full diffusion was first established by Cao and Titi in \cite{CT07}, and later by Kobelkov in \cite{K06}, see also the subsequent articles of Kukavica and Ziane \cite{KZ07,KZ072} for different boundary conditions,  as well as Hieber and Kashiwabara \cite{Hieber-Kashiwabara} for some progress towards relaxing the smoothness on the initial data by using the semigroup method. This result has been improved later by Cao, Li and Titi \cite{CLT16,CLT17,CLT17b}, where the authors proved global well-posedness for $3D$ PEs with only horizontal viscosity, i.e., with $\nu_h >0$ and $\nu_z = 0$. On the other hand, with only vertical viscosity, i.e., $\nu_h =0$ and $\nu_z>0$, Cao, Lin and Titi established recently \cite{CLT19} the local well-posedness of the PEs in Sobolev spaces by considering an additional weak dissipation, which is the linear (Rayleigh-like friction) damping. This linear damping helps the system overcome the ill-posedness in Sobolev spaces established in \cite{RE09}. See also \cite{CT10} for a similar idea on the effect of this linear damping.

When $\nu_h = \nu_z =0$, the inviscid PEs without coupling with the temperature is also called the hydrostatic Euler equations. In the absence of rotation ($\Omega = 0$), the linear ill-posedness of the inviscid PEs, near certain shear-flows, has been established by Renardy in \cite{RE09}. Later on, the nonlinear ill-posedness of the inviscid PEs without rotation was established by Han-Kwan and Nguyen in \cite{HN16}, where they built an abstract framework to show that the inviscid PEs are ill-posed in any Sobolev space. Moreover, it was proven that smooth solutions to the inviscid PEs, in the absence of rotation, can develop singularities in finite time (cf. Cao, Ibrahim, Nakanishi and Titi \cite{CINT15}, and Wong \cite{W12}). It is shown in \cite{ILT20} that these results on the finite-time blowup and the ill-posedness can also be extended to the $3D$ inviscid PEs with rotation, i.e., $\Omega \neq 0$. By virtue of the finite-time blowup results, one can conclude that there is no hope to show the global well-posedness of the $3D$ inviscid PEs, even with fast rotation. The optimal result one can expect is that fast rotation prolongs the life-span of solutions to the $3D$ inviscid PEs. 

The linear ill-posedness results mentioned above show that the linearized $2D$ inviscid PEs (as well as the $3D$ case \cite{ILT20}), around a special steady state background flow, has unstable solutions of the form $u(t,x,z) = e^{2\pi ikx} e^{\sigma_k t} u_k(z)$, where $\Re \sigma_k = \lambda k$ for some $\lambda \in \mathbb{R}$ and $\lambda \neq 0$. Such Kelvin-Helmholtz type instability, which is similar to the one appearing in the context of vortex sheets (see, e.g., \cite{CO89}, the survey paper \cite{BT07} and reference therein), precludes the construction of solutions in Sobolev spaces for general initial data. To overcome this strong instability, one should consider initial data $u_0$ that are strongly localized in Fourier, typically for which $|\hat{u}_0(k,z)|\lesssim e^{-\delta |k|^{1/s}}$ with $\delta >0$ and $s\geq 1$. Such localization condition corresponds to Gevrey class of order $s$ in the $x$ variable. Kelvin-Helmholtz type instability forces us to choose $s=1$ for the well-posedness result, which is the space of analytic functions. This is consistent with positive results reported in \cite{KTVZ11} and in this paper. Notably, for the Prandtl equations, which have some similarities in its structure with the PEs, is shown in \cite{GD10} that its linearization around a special background flow has unstable solutions of similar form, but with $\Re \sigma_k \sim \lambda \sqrt{k}$ for $k\gg 1$ arbitrarily large and some positive $\lambda \in \mathbb{R}_+$. This implies that the optimal Gevrey class order $s$ for Prandtl equation is $s=2$, which is consistent with the positive results reported in \cite{DG19,LMY20}. This shows that the linear instability of the inviscid PEs is ``worse" than that of the Prandtl equations.

Due to the ill-posedness discussed above, in order to show the well-posedness of the inviscid PEs, one needs to assume either some special structures (local Rayleigh condition) on the initial data or real analyticity for general initial data \cite{BR99,BR03,GR99,KMVW14,KTVZ11,MW12}. Indeed, the authors of \cite{KTVZ11} establish the local well-posedness of 
the $3D$ inviscid PEs in the space of analytic functions for various boundary conditions including the periodic boundary condition. Their approach utilizes explicit estimates for the pressure, regardless of the underlying boundary conditions. These estimates depend explicitly on $\Omega$, from which one concludes that the time of existence shrinks to zero as $|\Omega|$ increases toward infinity. As we will describe below, this conclusion is in some sense counter intuitive at least in the absence of boundary for the $3D$ Euler or Navier-Stokes equations, i.e., in the case of periodic boundary condition. Indeed, Babin, Mahalov and Nicolaenko \cite{BMN97, BMN99a, BMN99b, BMN00} have shown that in $\mathbb T^3$, fast rotation displays a strong averaging mechanism that weakens the nonlinear effects. This mechanism gives the global regularity in the $3D$ Navier-Stokes case, and the prolongation of the life-span of the solutions in the case of $3D$ Euler equations (see also \cite{CDGG06,D05,EM96,IY,KLT14} and references therein for the case of $\mathbb R^3$). In addition, we refer to \cite{BIT11,GST15,KTZ18,LT04} for simple examples demonstrating the above averaging/dispersion mechanism. Our purpose here is to show that in $\mathbb T^3$, the fast rotation delays the singularity formation, and thus prolongs the life-span of the solution of the $3D$ inviscid PEs.

For mathematical simplicity, we consider system (\ref{1})--(\ref{4}) with $T_0 = 0$, which implies $T \equiv 0$ for smooth solutions. By considering the inviscid case, i.e., $\nu_h = \nu_z = 0$, in this paper, we are interested in the effect of rotation on the $3D$ inviscid PEs (hydrostatic Euler equations)
\begin{eqnarray}
	&&\hskip-.8in \partial_t v + v\cdot \nabla v + w\partial_z v  +\Omega v^\perp + \nabla p = 0 , \label{EQ1-1}  \\
	&&\hskip-.8in  \partial_z p =0 , \label{EQ1-2}  \\
	&&\hskip-.8in
	\nabla \cdot v + \partial_z w =0,  \label{EQ1-3} 
\end{eqnarray}
in three-dimensional unit torus $\mathbb{T}^3$, 
subject to the following initial and boundary conditions:
\begin{eqnarray}
&&\hskip-.8in
v|_{t=0} =v_0,  \label{EQ-ic} \\
&&\hskip-.8in
v, w\;\text{are periodic in} \; (\boldsymbol{x}',z) \; \text{with period} \; 1,  \label{EQ-bc1}\\
&&\hskip-.8in
v \; \text{is even in} \; z \; \text{and} \; w \; \text{is odd in} \; z.  \label{EQ-bc2}
\end{eqnarray}

Observe that the space of periodic functions with respect to $z$ with the symmetry condition (\ref{EQ-bc2}) is invariant under the dynamics of system (\ref{EQ1-1})--(\ref{EQ1-3}). If $H=\frac{1}{2}$, the solution to system (\ref{EQ1-1})--(\ref{EQ1-3}) in $\mathbb{T}^3$ subject to (\ref{EQ-ic})--(\ref{EQ-bc2}) restricted on the horizontal channel $\big\{(\boldsymbol{x}',z): 0\leq z\leq \frac{1}{2}, \boldsymbol{x}'\in \mathbb{T}^2\big\}$ is the solution to system (\ref{EQ1-1})--(\ref{EQ1-3}) subject to the physical boundary conditions, i.e., $w|_{z= 0,\frac{1}{2}}=0$ and $v, w$ are periodic in $\boldsymbol{x}'$ with period $1$, and initial condition $v_0$ being even extendable in $z$ variable. Working in $\mathbb{T}^3$ allows us to use Fourier analysis, and makes the mathematical presentation simpler and more elegant.

The paper is organized as follows. In section 2, we introduce the notation and collect some preliminary results. In section 3, we establish the local well-posedness of the $3D$ inviscid PEs (\ref{EQ1-1})--(\ref{EQ1-3}) subject to (\ref{EQ-ic})--(\ref{EQ-bc2}) in the space of analytic functions in a short time interval uniform in $\Omega$.  In section 4, independently of $\Omega$, we show that the life-span of the solution tends to infinity as the analytic norm of the initial baroclinic mode goes to zero. Moreover, we show in this case that the solution of the $3D$ inviscid PEs converges to the solution of the limit system, which is governed by the $2D$ Euler equations. The intuition stems from the observation that the $3D$ inviscid PEs is reduced to the $2D$ Euler equations when the baroclinic mode is zero initially. In section 5, we explore further the structure of the inviscid PEs with rotation and derive its formal limit resonant system when $|\Omega|\rightarrow \infty$. Let us emphasize that this limit resonant system is not solely the $2D$ Euler equations when the initial baroclinic mode is not zero. Moreover, we investigate this limit resonant system and establish its global regularity in both Sobolev and the analytic functions spaces. In section 6, we establish the main result of this paper, namely, the life-span of the solution to the $3D$ inviscid PEs goes toward infinity, with $|\Omega|\rightarrow \infty$. This is established for well-prepared initial data, namely, when only the Sobolev norm  (but not the analytic norm) of the baroclinic mode is small enough, depending on $|\Omega|$. Furthermore, for large $|\Omega|$ and ``well-prepared" initial data, we show that the solution to the $3D$ inviscid PEs is indeed approximated by the solution to the limit resonant system that is the main feature of section 5. We also discuss in this section the rational behind the need for the smallness condition in the well-prepared initial data. The last section is an appendix, which is devoted to stating and proving technical lemmas concerning key nonlinear estimates.

\section{Preliminaries}
In this section, we introduce the notation and collect some preliminary results that will be used in this paper. The universal constant $C$ appears in this paper may change from step to step. When we use subscript for $C$, e.g., $C_r$, it means that the constant depends only on $r$. 
\subsection{Functional Settings}
We use the notation $\boldsymbol{x}:= (\boldsymbol{x}',z) = (x_1, x_2, z)\in \mathbb{T}^3$, where $\boldsymbol{x}'$ and $z$ represent the horizontal and vertical variables, respectively. $\mathbb{T}^3$ is the three-dimensional torus with unit length. Denote by
$
    \|f \|:=\|f\|_{L^2(\mathbb{T}^3)} = (\int_{\mathbb{T}^3} |f(\boldsymbol{x})|^2 d\boldsymbol{x})^{\frac{1}{2}},
$
associated with the inner product
$  \langle f,g\rangle = \int_{\mathbb{T}^3} f(\boldsymbol{x})g(\boldsymbol{x}) d\boldsymbol{x}
$
for $f,g \in L^2(\mathbb{T}^3)$. For a function $f \in L^2(\mathbb{T}^3)$, $\hat{f}_{\boldsymbol{k}}$ denotes its Fourier coefficient, so that
$
f(\boldsymbol{x}) = \sum\limits_{\boldsymbol{k}\in \mathbb{Z}^3} \hat{f}_{\boldsymbol{k}} e^{2\pi i\boldsymbol{k}\cdot \boldsymbol{x}}$ and $ \hat{f}_{\boldsymbol{k}} = \int_{\mathbb{T}^3} e^{-2\pi i\boldsymbol{k}\cdot \boldsymbol{x}} f(\boldsymbol{x}) d\boldsymbol{x}.$
For $r\geq 0$, define the following Sobolev $H^r$ norm and $\dot{H}^r$ semi-norm
\begin{eqnarray*}
&&\hskip-.1in
\|f \|_{H^r}:= \Big(\sum\limits_{\boldsymbol{k}\in \mathbb{Z}^3} (1+|\boldsymbol{k}|^{2r}) |\hat{f}_{\boldsymbol{k}}|^2  \Big)^{\frac{1}{2}}, \quad \|f \|_{\dot{H}^r}:= \Big(\sum\limits_{\boldsymbol{k}\in  \mathbb{Z}^3} |\boldsymbol{k}|^{2r} |\hat{f}_{\boldsymbol{k}}|^2  \Big)^{\frac{1}{2}}.
\end{eqnarray*}
For more details about Sobolev spaces, see \cite{AR75}. 

For $s>0$, a function $f\in C^\infty(\mathbb{T}^3)$ is said to be in Gevrey class of order $s$, denoted by $f\in G^s(\mathbb{T}^3)$, if there exist constants $\rho>0$ and $M>0$ such that for every $\boldsymbol{x}\in \mathbb{T}^3$ and $\alpha \in \mathbb{N}^3$, one has 
$
|\partial^\alpha f(\boldsymbol{x})| \leq M\Big( \frac{\alpha!}{\rho^{|\alpha|}} \Big)^s.
$
Denote by $A = \sqrt{-(\Delta + \partial_{zz})}$, subject to periodic boundary condition. For each $s>0$ and $r\geq 0$, we define a family, parameterized by $\tau\geq 0$, of normed spaces
\begin{eqnarray*}
\mathcal{D}(e^{\tau A^{1/s}}: H^r(\mathbb{T}^3)) := \{f\in H^r(\mathbb{T}^3): \|e^{\tau A^{1/s}} f\|_{H^r} <\infty  \},
\end{eqnarray*}
where the norm is defined by
\begin{eqnarray*}
\|e^{\tau A^{1/s}} f\|_{H^r} := \Big(\sum\limits_{\boldsymbol{k}\in  \mathbb{Z}^3} (1+|\boldsymbol{k}|^{2r} e^{2\tau |\boldsymbol{k}|^{1/s}} )|\hat{f}_{\boldsymbol{k}}|^2 \Big)^{\frac{1}{2}}. \label{analytic-norm}
\end{eqnarray*}
Let us denote the semi-norm by 
\begin{eqnarray*}
\|A^r e^{\tau A^{1/s}} f\| := \Big(\sum\limits_{\boldsymbol{k}\in  \mathbb{Z}^3} |\boldsymbol{k}|^{2r} e^{2\tau |\boldsymbol{k}|^{1/s}} |\hat{f}_{\boldsymbol{k}}|^2\Big)^{\frac{1}{2}}, \label{anlytic-semi-norm}
\end{eqnarray*}
then it is easy to see that
\begin{eqnarray*}
\|e^{\tau A^{1/s}} f\|_{H^r}^2 = \|A^r e^{\tau A^{1/s}} f\|^2 + \|f\|^2.
\end{eqnarray*}
For more details about Gevrey class, we refer the readers to \cite{FT98, FT89, LO97}. Observe that
\begin{eqnarray}\label{gevrey-relation}
G^s(\mathbb{T}^3) = \bigcup\limits_{\tau >0} \mathcal{D}(e^{\tau A^{1/s}} : H^r(\mathbb{T}^3)).
\end{eqnarray}
For the proof of \eqref{gevrey-relation}, see \cite{LO97}. The next lemma comes from \cite{LO97} (see also \cite{FT98}), addressing an important property of the space $\mathcal{D}(e^{\tau A^{1/s}} : H^r(\mathbb{T}^3))$.
\begin{lemma} \label{lemma-banach-algebra}
If $s\geq 1$, $\tau \geq 0$, and $r>\frac{3}{2}$, then $\mathcal{D}(e^{\tau A^{1/s}} : H^r(\mathbb{T}^3))$ is a Banach algebra, and for any $f, g\in \mathcal{D}(e^{\tau A^{1/s}} : H^r(\mathbb{T}^3))$, we have 
\begin{eqnarray*}
\| e^{\tau A^{1/s}} (fg)\|_{H^r} \leq C_{r,s}\| e^{\tau A^{1/s}} f\|_{H^r} \| e^{\tau A^{1/s}} g\|_{H^r}.
\end{eqnarray*}
For the semi-norm, we also have a similar estimate
\begin{eqnarray*}
\| A^r e^{\tau A^{1/s}} (fg)\| \leq C_{r,s}  \Big(|\hat{f}_0|+\| A^r e^{\tau A^{1/s}} f\|\Big) \Big(|\hat{g}_0|+ \|  A^r e^{\tau A^{1/s}} g\|\Big).
\end{eqnarray*}
\end{lemma}
For the proof, we refer the readers to \cite{FT98} for the case when $s=1$, and to \cite{OL96} for the case when $s>1$.
\begin{remark}
Since the inviscid PEs is linearly ill-posed in Sobolev spaces and Gevrey class of order $s>1$ \cite{ILT20,RE09}, we focus on Gevrey class of order $s=1$, which is equivalent to the space of analytic function.
\end{remark}
\subsection{Projections and reformulation of the problem}
In this paper, we assume that $\int_{\mathbb{T}^3} v_0(\boldsymbol{x}) d\boldsymbol{x} = 0$. This assumption is made to simplify the mathematical presentation. See Remark \ref{remark-zero-average} for detailed explanation. Integrating (\ref{EQ1-1}) in $\mathbb{T}^3$, by integration by parts, thanks to (\ref{EQ1-3}) and (\ref{EQ-bc1}), we obtain 
\begin{eqnarray*}\label{mean-zero}
\partial_t \int_{\mathbb{T}^3} v d\boldsymbol{x}  + \Omega \int_{\mathbb{T}^3} v^\perp d\boldsymbol{x} =0.
\end{eqnarray*}
Therefore, for any time $t\geq 0$, $v$ has zero mean in $\mathbb{T}^3$ :
\begin{equation}\label{mean-zero-2}
    \int_{\mathbb{T}^3} v d\boldsymbol{x} = \hat{v}_0 = 0.
\end{equation}

Denote by 
\begin{equation*}
     \dot{L}^2 : = \Big\{\varphi\in L^2(\mathbb{T}^3,\mathbb{R}^2)  : \int_{\mathbb{T}^3} \varphi(\boldsymbol{x})d\boldsymbol{x} = 0  \Big\}.
\end{equation*}
The barotropic mode $\overline{v}$ and baroclinic mode $\widetilde{v}$ are defined by
\begin{equation*}\label{barotropic-and-baroclinic}
    \overline{v}(\boldsymbol{x}'):= \int_0^1 v(\boldsymbol{x}',z)dz = \sum\limits_{\boldsymbol{k}\in \mathbb{Z}^3, k_3 = 0} \hat{v}_{\boldsymbol{k}} e^{2\pi i\boldsymbol{k}\cdot \boldsymbol{x}}, \;\;\;\;\;\; \widetilde{v}(\boldsymbol{x}):= v-\overline{v} = \sum\limits_{\boldsymbol{k}\in \mathbb{Z}^3, k_3 \neq 0} \hat{v}_{\boldsymbol{k}} e^{2\pi i\boldsymbol{k}\cdot \boldsymbol{x}}.
    \end{equation*}
From boundary condition (\ref{EQ-bc2}) and incompressible condition (\ref{EQ1-3}), one observes that 
\begin{equation}\label{incompressible-2d}
    \nabla\cdot \overline{v} = \int_0^1 \nabla\cdot v(\boldsymbol{x}',z)dz = -\int_0^1 \partial_z w(\boldsymbol{x}',z)dz =0.
\end{equation}
Since $\nabla\cdot \overline{v} =0$ and $\overline{v}$ has zero mean over $\mathbb{T}^2$ due to (\ref{mean-zero-2}), there exists a stream function $\psi(\boldsymbol{x}')\in H^1 (\mathbb T^2)$, defined uniquely up to a constant, such that $\overline{v} = \nabla^{\perp}\psi = (-\partial_{x_2} \psi, \partial_{x_1}\psi).$ Therefore, one has $v\in \mathcal{S}$
where
\begin{equation*}
    \mathcal{S}:= \Big\{\varphi\in \dot{L}^2: \nabla\cdot \overline{\varphi} = 0  \Big\} = \Big\{\varphi\in \dot{L}^2: \varphi = \nabla^\perp \psi(\boldsymbol{x}') + \widetilde{\varphi}(\boldsymbol{x})  \text{ with } \psi \in H^1(\mathbb T^2)  \Big\}.
\end{equation*}
For $\varphi\in\dot{L}^2$, the rotating matrix is
\begin{equation*}
    \mathcal{J}\varphi := \begin{pmatrix}
0 & -1 \\
1 & 0
\end{pmatrix} 
\begin{pmatrix}
\varphi_1 \\
\varphi_2
\end{pmatrix} = (-\varphi_2 , \varphi_1) = \varphi^\perp  .
\end{equation*}
Denote the $2D$ Leray projection by
$
    \mathbb{P}_h \overline{\varphi} := \overline{\varphi} - \nabla \Delta^{-1} \nabla\cdot \overline{\varphi},
$
where $ \Delta^{-1} $ represents the inverse of Laplacian operator in $ \mathbb T^2 $ with zero mean value.  Inspired by the $2D$ Leray projection, we define the projection $P_S : \dot{L}^2\rightarrow\mathcal{S}$ as
$
    P_S \varphi := \widetilde{\varphi} + \mathbb{P}_h \overline{\varphi}.
$
Moreover, define an operator $P: \mathcal{S}\rightarrow\mathcal{S}$ as
$
    P \varphi:= P_S (\mathcal{J}\varphi) .
$
A direct computation using $\nabla\cdot \overline{\varphi} = 0$ yields
$
    P \varphi = \widetilde{\varphi}^\perp.
$
It is easy to see that the kernel of $P$ is 
\begin{equation*}
    \ker P =  \Big\{\varphi\in \mathcal{S}:  \widetilde{\varphi}^\perp = 0  \Big\} =  \Big\{\varphi\in \mathcal{S}:  \varphi = \overline{\varphi} \Big\}.
\end{equation*}
Therefore, we define the projection $P_0: \mathcal{S} \rightarrow \ker P$ as
\begin{equation*}\label{P0}
    P_0 \varphi := \overline{\varphi} = \int_0^1 \varphi(\boldsymbol{x}',z) dz,
\end{equation*}
which actually projects any vector $\varphi\in \mathcal{S}$ to its barotropic mode. Now applying $P_S$ to equation (\ref{EQ1-1}), thanks to (\ref{EQ1-2}), and since $v\in \mathcal{S}$, we get
\begin{equation}\label{equation-P_S-projection}
    \partial_t v + P_S(v\cdot \nabla v + w\partial_z v) + \Omega \widetilde{v}^\perp = 0.
\end{equation}
Next, applying $P_0$ and $I-P_0$ to equation (\ref{equation-P_S-projection}), by integration by parts, thanks to (\ref{EQ-bc2}) and (\ref{incompressible-2d}), we derive the evolution equations for the barotropic mode $\overline{v}$ and the baroclinic mode $\widetilde{v}$:
\begin{eqnarray}
&&\hskip-.8in  \partial_t \overline{v} + \mathbb{P}_h \Big(\overline{v}\cdot \nabla \overline{v}\Big) + \mathbb{P}_h P_0 \Big((\nabla\cdot \widetilde{v}) \widetilde{v} + \widetilde{v}\cdot \nabla \widetilde{v} \Big)  = 0, \label{local-system-1} \\
&&\hskip-.8in \partial_t \widetilde{v} + \widetilde{v} \cdot \nabla \widetilde{v} + \widetilde{v} \cdot \nabla \overline{v} + \overline{v} \cdot \nabla \widetilde{v} - P_0\Big(\widetilde{v} \cdot \nabla \widetilde{v} + (\nabla \cdot \widetilde{v}) \widetilde{v} \Big) - \Big(\int_0^z \nabla\cdot \widetilde{v}(\boldsymbol{x}',s)ds \Big) \partial_z \widetilde{v}  + \Omega \widetilde{v}^{\perp} = 0 . \label{local-system-2} 
\end{eqnarray}
In summary, we have the following lemma.
\begin{lemma}\label{1st-equivalence}
 For $v\in\mathcal{S}$, system (\ref{EQ1-1})--(\ref{EQ1-3}) is equivalent to system (\ref{local-system-1})--(\ref{local-system-2}).
\end{lemma}
Notice that if we consider $v_0\in\ker P$, i.e., consider $\widetilde{v}_0 = 0$, then from (\ref{local-system-2}) we can see $\widetilde{v}$ remains zero. Therefore, system (\ref{local-system-1})--(\ref{local-system-2}) reduces to the $2D$ Euler equations, which is globally well-posed.
Based on this observation, we establish the first long time existence result in section 4 by assuming the analytic norm of $\widetilde{v}_0$ is small. In order to investigate the effect of rotation, we further study the evolution of the baroclinic mode. This can be done by further decomposing the baroclinic mode in order to identify the resonant and non-resonant parts due to the rotation. Since the rotation matrix $\mathcal{J}$ has eigenvalues $\pm i$, with corresponding eigenvectors $\frac{1}{\sqrt{2}}
\begin{pmatrix}
1 \\
\mp i
\end{pmatrix} $,
we can define
\begin{equation*}
    P_\pm \varphi := \Big\langle(I-P_0)\varphi,\overline{\frac{1}{\sqrt{2}}
\begin{pmatrix}
1 \\
\pm i
\end{pmatrix}} \Big\rangle_E \frac{1}{\sqrt{2}}
\begin{pmatrix}
1 \\
\pm i
\end{pmatrix} = \frac{1}{2}\Big\langle\widetilde{\varphi}, \overline{\begin{pmatrix}
1 \\
\pm i
\end{pmatrix}} \Big\rangle_E\begin{pmatrix}
1 \\
\pm i
\end{pmatrix} 
= \frac{1}{2}(\widetilde{\varphi} \pm i\widetilde{\varphi}^\perp).
\end{equation*}
Here $\langle \cdot, \cdot \rangle_E$ denotes the usual Euclidean inner product. Similar ideas and projections for $3D$ rotating Euler equations can be found in \cite{D05, KLT14}. Observe that the operator $P$ has three eigenvalues, $0$ and $\pm i$. Therefore, the projections $P_0$ and $P_\pm$ project $v$ into the eigenspaces corresponding to $0$ and $\mp i$, respectively. Consequently, we have the following:
\begin{lemma}\label{lemma-orthogonal-decomposition}
For any $\varphi \in L^2(\mathbb{T}^3)$, we have 
\begin{equation*}
    \varphi = P_0 \varphi + P_+ \varphi + P_- \varphi \quad\text{ and }
\end{equation*}
\begin{equation*}
 P_\pm P_\pm \varphi = P_\pm \varphi, \;\;\;\;\;\; P_0 P_0 \varphi = P_0 \varphi, \;\;\;\;\;\; P_\pm P_\mp \varphi = P_0 P_\pm \varphi = P_\pm P_0 \varphi= 0.
\end{equation*}
\end{lemma}
\begin{proof}
The proof is straightforward from the definition of $P_0$ and $P_\pm$, and the fact that $\overline{\widetilde{\varphi}}=\widetilde{\overline{\varphi}} = 0$. 
\end{proof}

For projections $P_0, P_\pm$, we have the following properties. The proof is straightforward and we omit it.
\begin{lemma} \label{lemma-projection}
	For $f,g \in L^2(\mathbb{T}^3)$, we have
	$
	\langle P_0 f, g \rangle = \langle f, P_0 g\rangle = \langle P_0 f, P_0 g\rangle$ and $\langle P_\pm f, g\rangle = \langle f, P_\mp g\rangle.$
	If $f \in H^r(\mathbb{T}^3)$ with $r\geq 0$, then for $|\alpha|\leq r$, we have
	$
	\partial^\alpha P_0 f = P_0 \partial^\alpha f$  and $\partial^\alpha P_\pm f = P_\pm \partial^\alpha f. 
	$
Furthermore, if $f\in \mathcal{D}(e^{\tau A^{1/s}} : H^r(\mathbb{T}^3))$ with $s>0$ and $r\geq 0$, one has
$
    A^{r}e^{\tau A^{1/s}} P_0 f = P_0 A^{r}e^{\tau A^{1/s}} f.
$
\end{lemma}

The Leray projection $\mathbb{P}_h$ enjoys the following properties. For the proof, see, for example, \cite{CF88}.
\begin{lemma}\label{lemma-leray}
For $f, g\in L^2(\mathbb{T}^3)$, we have
$
	\langle \mathbb{P}_h f, g \rangle = \langle f, \mathbb{P}_h g\rangle 
$
and
$
\mathbb{P}_h P_0 f = P_0 \mathbb{P}_h f. $
If $f\in H^r(\mathbb{T}^3)$ with $r\geq 0$, then for $|\alpha|\leq r$, one has
$
\partial^\alpha \mathbb{P}_h f = \mathbb{P}_h \partial^\alpha f. 
$
Moreover, if $f\in \mathcal{D}(e^{\tau A^{1/s}} : H^r(\mathbb{T}^3))$ with $s>0$ and $r\geq 0$, one gets $
    A^{r}e^{\tau A^{1/s}} \mathbb{P}_h f = \mathbb{P}_h A^{r}e^{\tau A^{1/s}} f.
$
\end{lemma}

For the relation between the norm of $v$ and the norms of $\overline{v}, \widetilde{v}$ in $L^2(\mathbb{T}^3)$ and $\mathcal{D}(e^{\tau A^{1/s}}: H^r(\mathbb{T}^3))$, we have the following Lemma. The proof is straightforward and we omit it.
\begin{lemma}\label{lemma-decomposition}
Let $v = P_0 v + (I-P_0)v = \overline{v} + \widetilde{v}$. Suppose that $r\geq 0$, $s>0$, and $\tau \geq 0$, we have 
\begin{equation*}
    \|v\|^2 = \|\overline{v}\|^2 + \|\widetilde{v}\|^2 \text{ and }
  \|e^{\tau A^{1/s}} v\|_{H^r}^2 = \|e^{\tau A^{1/s}} \overline{v}\|_{H^r}^2 + \|e^{\tau A^{1/s}} \widetilde{v}\|_{H^r} ^2  .
\end{equation*}
\end{lemma}

Observe that $\widetilde{v}^\perp$ can be written as $\widetilde{v}^{\perp} = -i(P_+ v - P_- v).$ Hence applying $P_\pm$ to (\ref{local-system-2}), one has
\begin{equation}\label{baroclinic-evolution-2} 
\begin{split}
    \partial_t P_\pm v + P_\pm\Big(\widetilde{v} \cdot \nabla \widetilde{v} + \widetilde{v} \cdot \nabla \overline{v} + \overline{v} \cdot \nabla \widetilde{v} &- P_0(\widetilde{v} \cdot \nabla \widetilde{v} + (\nabla \cdot \widetilde{v}) \widetilde{v} )\\
    &- (\int_0^z \nabla\cdot \widetilde{v}(\boldsymbol{x}',s)ds ) \partial_z \widetilde{v} \Big) \mp i\Omega P_\pm v = 0 . 
\end{split}
\end{equation}
By setting $u_\pm = e^{\mp i\Omega t}P_\pm v $, (\ref{baroclinic-evolution-2}) can be rewritten as
\begin{equation}\label{baroclinic-evolution-3}
\begin{split}
    \partial_t u_\pm + e^{\mp i\Omega t}P_\pm\Big(\widetilde{v} \cdot \nabla \widetilde{v} + \widetilde{v} \cdot \nabla \overline{v} + \overline{v} \cdot \nabla \widetilde{v} &- P_0(\widetilde{v} \cdot \nabla \widetilde{v} + (\nabla \cdot \widetilde{v}) \widetilde{v} )\\
    &- (\int_0^z \nabla\cdot \widetilde{v}(\boldsymbol{x}',s)ds ) \partial_z \widetilde{v} \Big) = 0.
\end{split}
\end{equation}
For the $u_+$ part, thanks to Lemma \ref{lemma-orthogonal-decomposition}, we have
\begin{equation*}
\begin{split}
    P_+(\widetilde{v} \cdot \nabla \widetilde{v})
    &=\frac{1}{2} (\widetilde{v} \cdot \nabla \widetilde{v} + i \widetilde{v} \cdot \nabla \widetilde{v}^\perp) - \frac{1}{2} P_0\Big( \widetilde{v} \cdot \nabla \widetilde{v} + i \widetilde{v} \cdot \nabla \widetilde{v}^\perp \Big)\\
    &= \frac{1}{2} \widetilde{v} \cdot \nabla (\widetilde{v} + i\widetilde{v}^\perp) - \frac{1}{2} P_0\Big( \widetilde{v} \cdot \nabla (\widetilde{v} + i\widetilde{v}^\perp) \Big) = e^{i\Omega t}\Big(\widetilde{v} \cdot \nabla u_+  - P_0(\widetilde{v} \cdot \nabla u_+) \Big)  ,
\end{split}
\end{equation*}
\begin{equation*}
    P_+(\widetilde{v} \cdot \nabla \overline{v}) = \frac{1}{2} (\widetilde{v} \cdot \nabla \overline{v} + i \widetilde{v} \cdot \nabla \overline{v}^\perp) = \frac{1}{2} \widetilde{v} \cdot \nabla (\overline{v} + i\overline{v}^\perp) ,
\end{equation*}
\begin{equation*}
    P_+(\overline{v} \cdot \nabla \widetilde{v}) = \frac{1}{2} (\overline{v} \cdot \nabla \widetilde{v} + i \overline{v} \cdot \nabla \widetilde{v}^\perp)  = e^{i\Omega t}(\overline{v} \cdot \nabla u_+ ),
\end{equation*}
\begin{equation*}
    P_+ P_0\Big(\widetilde{v} \cdot \nabla \widetilde{v} +(\nabla \cdot \widetilde{v}) \widetilde{v}\Big) = 0.
\end{equation*}
Observe that by integration by parts one has
\begin{equation*}
\begin{split}
    P_+\Big((\int_0^z \nabla\cdot \widetilde{v}(\boldsymbol{x}',s)ds) \partial_z \widetilde{v}\Big)  = &\frac{1}{2} \Big((\int_0^z \nabla\cdot \widetilde{v}(\boldsymbol{x}',s)ds ) \partial_z \widetilde{v} + i (\int_0^z \nabla\cdot \widetilde{v}(\boldsymbol{x}',s)ds ) \partial_z \widetilde{v}^\perp \Big) \\
    &- \frac{1}{2} P_0\Big((\int_0^z \nabla\cdot \widetilde{v}(\boldsymbol{x}',s)ds ) \partial_z \widetilde{v} + i (\int_0^z \nabla\cdot \widetilde{v}(\boldsymbol{x}',s)ds ) \partial_z \widetilde{v}^\perp \Big)\\
     = &e^{i\Omega t} (\int_0^z \nabla\cdot \widetilde{v}(\boldsymbol{x}',s)ds )  \partial_z u_+ + e^{i\Omega t} P_0 \Big( (\nabla\cdot\widetilde{v}) u_+  \Big) .
\end{split}
\end{equation*}
Therefore, $u_+$ part in (\ref{baroclinic-evolution-3}) becomes
\begin{eqnarray}
&&\hskip-.8in \partial_t u_+ = - \Big(\widetilde{v} \cdot \nabla u_+ + \overline{v} \cdot \nabla u_+ - P_0(\widetilde{v} \cdot \nabla u_+ + (\nabla \cdot \widetilde{v}) u_+) -(\int_0^z \nabla\cdot \widetilde{v}(\boldsymbol{x}',s)ds ) \partial_z u_+  \Big) \nonumber \\
&&\hskip-.28in - \frac{1}{2} e^{-i\Omega t} (\widetilde{v}\cdot \nabla)(\overline{v}+i\overline{v}^\perp).
\label{up2} 
\end{eqnarray}
Using $\widetilde{v} = u_+ e^{i\Omega t} + u_-e^{-i\Omega t}$, we can furthermore rewrite (\ref{up2}) as
\begin{eqnarray}
&&\hskip-.8in \partial_t u_+ = -e^{i\Omega t} \Big(u_+ \cdot \nabla u_+ - P_0( u_+ \cdot \nabla u_+ + (\nabla \cdot u_+) u_+) - (\int_0^z \nabla\cdot u_+(\boldsymbol{x}',s)ds ) \partial_z u_+ \Big) \nonumber \\
&&\hskip-.28in -  \Big(\overline{v} \cdot \nabla u_+  + \frac{1}{2}(u_+ \cdot \nabla)(\overline{v}+i\overline{v}^\perp) \Big) - e^{-2i\Omega t} \frac{1}{2} (u_- \cdot \nabla)(\overline{v}+i\overline{v}^\perp) \nonumber \\
&&\hskip-.28in - e^{-i\Omega t} \Big(u_- \cdot \nabla u_+ - P_0( u_- \cdot \nabla u_+ + (\nabla \cdot u_-) u_+) - (\int_0^z \nabla\cdot u_-(\boldsymbol{x}',s)ds ) \partial_z u_+ \Big). 
\label{up3} 
\end{eqnarray}
From \eqref{up3}, one can identify the resonant and non-resonant parts due to the rotation. Notice that $u_-$ is the complex conjugate of $u_+$, therefore,
\begin{eqnarray}
&&\hskip-.8in \partial_t u_- = -e^{-i\Omega t} \Big(u_- \cdot \nabla u_- - P_0( u_- \cdot \nabla u_- + (\nabla \cdot u_-) u_-) - (\int_0^z \nabla\cdot u_-(\boldsymbol{x}',s)ds ) \partial_z u_- \Big) \nonumber \\
&&\hskip-.28in -  \Big(\overline{v} \cdot \nabla u_-  + \frac{1}{2}(u_- \cdot \nabla)(\overline{v}-i\overline{v}^\perp) \Big)  - e^{2i\Omega t} \frac{1}{2} (u_+ \cdot \nabla)(\overline{v}-i\overline{v}^\perp) \nonumber \\
&&\hskip-.28in - e^{i\Omega t} \Big(u_+ \cdot \nabla u_- - P_0( u_+ \cdot \nabla u_- + (\nabla \cdot u_+) u_-) - (\int_0^z \nabla\cdot u_+(\boldsymbol{x}',s)ds ) \partial_z u_- \Big).
\label{um1} 
\end{eqnarray}
For $\overline{v}$, using $\widetilde{v} = u_+ e^{i\Omega t} + u_-e^{-i\Omega t}$, we can rewrite (\ref{local-system-1}) as:
\begin{eqnarray*}
&&\hskip-.8in \partial_t \overline{v} + \mathbb{P}_h(\overline{v}\cdot \nabla \overline{v}) + e^{2i\Omega t}\mathbb{P}_h P_0\Big(u_+ \cdot \nabla u_+  + (\nabla \cdot u_+) u_+\Big)  + e^{-2i\Omega t}\mathbb{P}_h P_0\Big(u_- \cdot \nabla u_-  + (\nabla \cdot u_-) u_-\Big) \nonumber \\
&&\hskip-.58in  + \mathbb{P}_h P_0 \Big(u_+ \cdot \nabla u_-  + u_- \cdot \nabla u_+ + (\nabla \cdot u_+) u_- + (\nabla \cdot u_-) u_+  \Big) = 0.
\label{barotropic-evolution-3} 
\end{eqnarray*}
Since $u_\pm = e^{\mp i\Omega t} P_\pm v = \frac{1}{2} e^{\mp i\Omega t} (\widetilde{v} \pm i \widetilde{v}^\perp)$, thanks to Lemma \ref{lemma-leray}, the last term becomes
\begin{eqnarray*}
&&\hskip-.8in
\mathbb{P}_h P_0 \Big(u_+ \cdot \nabla u_-  + u_- \cdot \nabla u_+ + (\nabla \cdot u_+) u_- + (\nabla \cdot u_-) u_+ \Big)  \\
&&\hskip-.9in 
= P_0 \mathbb{P}_h \Big(u_+ \cdot \nabla u_-  + u_- \cdot \nabla u_+ + (\nabla \cdot u_+) u_- + (\nabla \cdot u_-) u_+ \Big) \\
&&\hskip-.9in 
= \frac{1}{2} P_0 \mathbb{P}_h \Big(\widetilde{v} \cdot \nabla \widetilde{v} + \widetilde{v}^\perp \cdot \nabla \widetilde{v}^\perp + (\nabla \cdot \widetilde{v}) \widetilde{v} + (\nabla \cdot \widetilde{v}^\perp) \widetilde{v}^\perp\Big) = \frac{1}{2}  P_0 \mathbb{P}_h (\nabla |\widetilde{v}|^2) = 0.
\end{eqnarray*}
Therefore, one obtains
\begin{eqnarray}
&&\hskip-.8in \partial_t \overline{v} + \mathbb{P}_h(\overline{v}\cdot \nabla \overline{v}) + e^{2i\Omega t}\mathbb{P}_h P_0 \Big(u_+ \cdot \nabla u_+  + (\nabla \cdot u_+) u_+ \Big)  \nonumber\\
&&\hskip.22in
 +e^{-2i\Omega t}\mathbb{P}_h P_0\Big(u_- \cdot \nabla u_-  + (\nabla \cdot u_-) u_-\Big) = 0.
\label{barotropic-evolution-4} 
\end{eqnarray}

In summary, we hae the following lemma.
\begin{lemma}\label{2nd-equivalence}
For $v\in\mathcal{S}$, system (\ref{EQ1-1})--(\ref{EQ1-3}) is equivalent to system (\ref{up3})--(\ref{barotropic-evolution-4}).
\end{lemma}
For the relation between the norm of $\widetilde{v}$ and the norms of $u_\pm$ in $L^2(\mathbb{T}^3)$ and $\mathcal{D}(e^{\tau A^{1/s}}: H^r(\mathbb{T}^3))$, we have the following Lemma. The proof is straightforward and we omit it.
\begin{lemma}\label{lemma-vtilde-upm}
Let $u_\pm = \frac{1}{2} e^{\mp i\Omega t}(\widetilde{v}\pm i\widetilde{v}^\perp)$. Suppose that $r\geq 0$, $s>0$, and $\tau \geq 0$, we have
\begin{equation*}
    \|u_+\|^2 = \|u_-\|^2 = \frac{1}{2}\|\widetilde{v}\|^2 \text{ and } \|e^{\tau A^{1/s}} u_+\|_{H^r}^2 = \|e^{\tau A^{1/s}} u_-\|_{H^r}^2 = \frac{1}{2}\|e^{\tau A^{1/s}} \widetilde{v}\|_{H^r} ^2  .
\end{equation*}
\end{lemma}

In sections 3 and 4, we work with system (\ref{local-system-1})--(\ref{local-system-2}) since the results are independent of the rate of rotation. On the other hand, in section 5 and 6, we work with system (\ref{up3})--(\ref{barotropic-evolution-4}) since our focus is on the effect of rotation.

\section{Local in time Well-posedness}
In this section, we study the local in time well-posedness in the space of analytic functions system \eqref{local-system-1}--\eqref{local-system-2}
in $\mathbb{T}^3$, subject to the following symmetry boundary conditions and initial conditions:
\begin{eqnarray}
&&\hskip-.8in
\overline{v}, \widetilde{v}\text{ are periodic in } \mathbb{T}^3 \text{ and are even in } z; \label{local-bc2}\\
&&\hskip-.8in
\overline{v}|_{t=0} = \overline{v}_0 = P_0 v_0,  \hspace{0.1in} \widetilde{v}|_{t=0} = \widetilde{v}_0 = (I-P_0)v_0, \hspace{0.1in} \nabla\cdot \overline{v}_0 =0. \label{local-ic}
\end{eqnarray}
Observe that whenever $v\in\mathcal{S}$
then $\overline{v},\widetilde{v}\in\mathcal{S}$. We have the following result:

\begin{theorem}\label{theorem-local}
Assume $\overline{v}_0, \widetilde{v}_0 \in \mathcal{S}\cap\mathcal{D}(e^{\tau_0 A}: H^r(\mathbb{T}^3))$ with $r>\frac{5}{2}$ and $\tau_0 >0$. Let $\Omega \in \mathbb{R}$ be arbitrary and fixed. Then there exist a time 
\begin{equation}\label{T-1}
    \mathcal{T} = \frac{\tau_0}{1+2C_r(1+\|e^{\tau_0 A} \overline{v}_0 \|_{H^r}^2 + \|e^{\tau_0 A} \widetilde{v}_0 \|_{H^r}^2)} > 0,
\end{equation}
and a function
\begin{equation}\label{tau-1}
    \tau(t) = \tau_0 - 2t C_r(1+\|e^{\tau_0 A} \overline{v}_0 \|_{H^r}^2 + \|e^{\tau_0 A} \widetilde{v}_0 \|_{H^r}^2),  
\end{equation}
both independent of $\Omega$, such that there exists a unique solution
\begin{equation}\label{regularity-solution-local}
    (\overline{v}, \widetilde{v}) \in   L^\infty\big(0,\mathcal{T}; \mathcal{S}\cap\mathcal{D}(e^{\tau(t) A} : H^r(\mathbb{T}^3))\big) \cap L^2\big(0,\mathcal{T}; \mathcal{D}(e^{\tau(t) A} : H^{r+\frac{1}{2}}(\mathbb{T}^3))\big)
\end{equation}
to system (\ref{local-system-1})--(\ref{local-system-2}) on $[0,\mathcal{T}]$. Moreover, the unique solution $(\overline{v}, \widetilde{v})$ depends continuously on the initial data, in the sense of (\ref{continuous-dependence}).
\end{theorem}

Observe that in the theorem above, the local time of existence is independent of $\Omega$, unlike the situation in \cite{KTVZ11} under the periodic boundary condition. Thanks to Lemma \ref{1st-equivalence} and \ref{lemma-decomposition} we have the following corollary for the original system (\ref{EQ1-1})--(\ref{EQ-bc2}).
\begin{corollary}\label{corollary-local}
Assume $v_0 \in \mathcal{S}\cap\mathcal{D}(e^{\tau_0 A}: H^r(\mathbb{T}^3))$ with $r>\frac{5}{2}$ and $\tau_0 >0$. Let $\Omega \in \mathbb{R}$ be arbitrary and fixed. Then there exist a time $\mathcal{T}$ defined in (\ref{T-1}) and a function $\tau(t)$ defined in (\ref{tau-1}), both independent of $\Omega$, such that there exists a unique solution
\begin{equation*}
    v\in  L^\infty\big(0,\mathcal{T}; \mathcal{S}\cap\mathcal{D}(e^{\tau(t) A} : H^r(\mathbb{T}^3))\big) \cap L^2\big(0,\mathcal{T}; \mathcal{D}(e^{\tau(t) A} : H^{r+\frac{1}{2}}(\mathbb{T}^3))\big)
\end{equation*}
to system (\ref{EQ1-1})--(\ref{EQ-bc2}) on $[0,T]$. Moreover, the unique solution $v$ depends continuously on the initial data.
\end{corollary}
To show the existence of solutions, one can work on the Galerkin approximation of system (\ref{local-system-1})--(\ref{local-system-2}) to establish an uniform energy estimate, then by using, in a nontraditional way (cf. \cite{LT10}), the Aubin-Lions compactness theorem to pass to the limit and show the existence of solutions. For simplicity, we only do the formal energy estimates (for details, see \cite{GILT20}). Finally, we establish the uniqueness of solutions and its continuous dependence on the initial data.

\subsection{Energy Estimates}
In this section, we establish the formal energy estimates for system (\ref{local-system-1})--(\ref{local-system-2}). By virtue of Lemma \ref{lemma-projection} and Lemma \ref{lemma-leray}, and since $\nabla \cdot \overline{v} = 0$, we have the conservation of the $L^2$ energy
\begin{eqnarray}
&&\hskip-.78in
\|\overline{v} (t) \|^2 + \|\widetilde{v} (t) \|^2 
= \|\overline{v}_0\|^2 + \|\widetilde{v}_0 \|^2. \label{energy-estimate-L2-Galerkin}
\end{eqnarray}
Next, employing Lemma \ref{lemma-projection} and Lemma \ref{lemma-leray}, we derive the following estimate for the analytic norm
\begin{eqnarray*}\label{galerkin-barotropic-1}
&&\hskip-.58in
\frac{1}{2} \frac{d}{dt} \|A^r e^{\tau A} \overline{v}  \|^2 = \dot{\tau} \|A^{r+\frac{1}{2}} e^{\tau A} \overline{v}  \|^2 - \Big\langle A^r e^{\tau A}\Big(\overline{v}  \cdot \nabla \overline{v} \Big) , A^r e^{\tau A} \overline{v}  \Big\rangle \nonumber \\
&&\hskip.58in - \Big\langle A^r e^{\tau A}\Big((\nabla\cdot \widetilde{v} )\widetilde{v} \Big) , A^r e^{\tau A} \overline{v}  \Big\rangle - \Big\langle A^r e^{\tau A}\Big(\widetilde{v}  \cdot \nabla \widetilde{v} \Big) , A^r e^{\tau A} \overline{v}  \Big\rangle, \quad \text{and}
\end{eqnarray*}
\begin{eqnarray*}\label{galerkin-baroclinic-1}
&&\hskip-.58in
\frac{1}{2} \frac{d}{dt} \|A^r e^{\tau A} \widetilde{v}  \|^2 = \dot{\tau} \|A^{r+\frac{1}{2}} e^{\tau A} \widetilde{v}  \|^2 - \Big\langle A^r e^{\tau A}\Big(\widetilde{v}  \cdot \nabla \widetilde{v} \Big) , A^r e^{\tau A} \widetilde{v}  \Big\rangle - \Big\langle A^r e^{\tau A}\Big( \widetilde{v}  \cdot \nabla \overline{v}  \Big) , A^r e^{\tau A} \widetilde{v}  \Big\rangle  \nonumber \\
&&\hskip.58in - \Big\langle A^r e^{\tau A}\Big(\overline{v}  \cdot \nabla \widetilde{v} \Big) , A^r e^{\tau A} \widetilde{v}  \Big\rangle + \Big\langle A^r e^{\tau A}\Big( (\int_0^z \nabla\cdot \widetilde{v} (\boldsymbol{x}', s)ds) \partial_z \widetilde{v}  \Big) , A^r e^{\tau A} \widetilde{v}  \Big\rangle. 
\end{eqnarray*}
By Lemma \ref{lemma-type1}--\ref{lemma-type3}, since $\overline{v} $ and $\widetilde{v} $ having zero mean and thanks to Young's inequality, we obtain
\begin{eqnarray}
&&\hskip-.58in 
\frac{1}{2} \frac{d}{dt} \Big( \|A^r e^{\tau A} \overline{v}  \|^2 + \|A^r e^{\tau A} \widetilde{v}  \|^2  \Big) +  \Big( \|A^{r+\frac{1}{2}} e^{\tau A} \overline{v}  \|^2 +  \|A^{r+\frac{1}{2}} e^{\tau A} \widetilde{v}  \|^2 \Big)\nonumber \\
&&\hskip-.28in 
\leq  \Big( \dot{\tau} + C_r( \|A^r e^{\tau A} \overline{v}  \| + \|A^r e^{\tau A} \widetilde{v}  \|) + 1\Big) \Big( \|A^{r+\frac{1}{2}} e^{\tau A} \overline{v}  \|^2 +  \|A^{r+\frac{1}{2}} e^{\tau A} \widetilde{v}  \|^2 \Big)\nonumber\\
&&\hskip-.28in 
\leq\Big( \dot{\tau} + C_r(1+ \|e^{\tau A} \overline{v}  \|_{H^r}^2 + \|e^{\tau A} \widetilde{v}  \|_{H^r}^2) \Big) \Big( \|A^{r+\frac{1}{2}} e^{\tau A} \overline{v}  \|^2 +  \|A^{r+\frac{1}{2}} e^{\tau A} \widetilde{v}  \|^2 \Big)
. \label{energy-estimate-Galerkin-1}
\end{eqnarray}
\begin{remark}
Here we add the term $\|A^{r+\frac{1}{2}} e^{\tau A} \overline{v}  \|^2 +  \|A^{r+\frac{1}{2}} e^{\tau A} \widetilde{v}  \|^2 $ to both sides so that one can obtain the regularity in $L^2\big(0,\mathcal{T}; \mathcal{D}(e^{\tau(t) A} : H^{r+\frac{1}{2}}(\mathbb{T}^3))\big)$.
\end{remark}
Let $\tau$ satisfy 
\begin{eqnarray}\label{taudot}
\dot{\tau} + 2C_r(1+ \|e^{\tau_0 A} \overline{v}_0 \|_{H^r}^2 + \|e^{\tau_0 A} \widetilde{v}_0 \|_{H^r}^2)  = 0,
\end{eqnarray}
hence
\begin{equation}
  \tau(t) = \tau_0 - 2t C_r(1+\|e^{\tau_0 A} \overline{v}_0 \|_{H^r}^2 + \|e^{\tau_0 A} \widetilde{v}_0 \|_{H^r}^2).   \label{tau}
\end{equation}
Denote by 
\begin{eqnarray}
\mathcal{T} = \frac{\tau_0}{1+2C_r(1+\|e^{\tau_0 A} \overline{v}_0 \|_{H^r}^2 + \|e^{\tau_0 A} \widetilde{v}_0 \|_{H^r}^2)} > 0, \label{T}
\end{eqnarray}
therefore, 
$
    \tau(t)\geq \tau(\mathcal{T})=\frac{\tau_0}{1+2C_r(1+\|e^{\tau_0 A} \overline{v}_0 \|_{H^r}^2 + \|e^{\tau_0 A} \widetilde{v}_0 \|_{H^r}^2)}>0
$
for $t\in[0,\mathcal{T}]$. Here we require $C_r$ to be large enough such that 
\begin{equation}\label{Cr}
    C_r \geq 2 (\widetilde{C}_r+ C_{r-\frac{1}{2}}),
\end{equation}
where $\widetilde{C}_r$ appears in (\ref{uniqueness-L2}) and $C_{r-\frac{1}{2}}$ appears in (\ref{uniqueness-analytic}). 
Thanks to \eqref{energy-estimate-L2-Galerkin}, \eqref{energy-estimate-Galerkin-1}, and \eqref{taudot}, one obtains that for $t\in [0,\mathcal{T}]$,
\begin{equation}\label{energy-estimate-local}
\begin{split}
    \|e^{\tau(t) A} \overline{v} (t) \|_{H^r}^2 + \|e^{\tau(t) A} \widetilde{v} (t) \|_{H^r}^2 &+ 2 \int_0^t \|A^{r+\frac{1}{2}} e^{\tau(s) A} \overline{v}(s)  \|^2 +  \|A^{r+\frac{1}{2}} e^{\tau(s) A} \widetilde{v}(s)  \|^2 ds
    \\
    & \leq  \|e^{\tau_0 A} \overline{v}_0\|_{H^r}^2 + \|e^{\tau_0 A} \widetilde{v}_0\|_{H^r}^2. 
\end{split}
\end{equation}
Moreover, it is easy to see that $(\overline{v}, \widetilde{v}) \in \mathcal{S}$. Therefore, the solution $(\overline{v}, \widetilde{v})$ satisfies \eqref{regularity-solution-local}.

For the estimates on $\partial_t \overline{v} $ and $\partial_t \widetilde{v} $, by directly applying $L^2$ estimate on (\ref{local-system-1}) and (\ref{local-system-2}), thanks to Lemma \ref{lemma-banach-algebra}, \ref{lemma-projection} and \ref{lemma-leray}, by the H\"older inequality and the Sobolev inequality, since $r>\frac{5}{2}$, one has
\begin{eqnarray}
&&\hskip-.8in 
\| \partial_t \overline{v} \| \leq C_r (\|\overline{v} \|_{H^r}^2 + \|\widetilde{v} \|_{H^r}^2) , \quad \| \partial_t \widetilde{v} \| \leq C_r (\|\overline{v} \|_{H^r}^2 + \|\widetilde{v} \|_{H^r}^2 + |\Omega| \|\widetilde{v} \| ), \label{L2-time-derivative}
\end{eqnarray}
\begin{equation}\label{analytic-time-derivative-bar}
    \begin{split}
        \| A^{r-\frac{1}{2}}&e^{\tau A} \partial_t \overline{v} \| \leq C_r \Big( \|e^{\tau A} \overline{v}  \|_{H^{r}} \|e^{\tau A} \overline{v}  \|_{H^{r+\frac{1}{2}}} +  \|e^{\tau A} \widetilde{v}  \|_{H^{r}} \|e^{\tau A} \widetilde{v}  \|_{H^{r+\frac{1}{2}}} \Big),
    \end{split}
\end{equation}
\begin{equation}\label{analytic-time-derivative-tilde}
\begin{split}
    \| A^{r-\frac{1}{2}}&e^{\tau  A} \partial_t \widetilde{v}  \| \leq C_r \Big(\|e^{\tau  A} \overline{v}  \|_{H^{r+\frac{1}{2}}}^2 + \|e^{\tau  A} \widetilde{v}  \|_{H^{r+\frac{1}{2}}}^2 + |\Omega|\|A^{r}e^{\tau  A} \widetilde{v}  \| \Big) .
\end{split}
\end{equation}
By virtue of the bound \eqref{energy-estimate-local}, from (\ref{L2-time-derivative})--(\ref{analytic-time-derivative-tilde}), we have
\begin{eqnarray*}
&&\hskip-.8in 
\partial_t\overline{v} \in L^2\big(0,\mathcal{T}; \mathcal{D}(e^{\tau(t) A} : H^{r-\frac{1}{2}})\big) \cap L^\infty\big(0,\mathcal{T}; L^2), \nonumber \\
&&\hskip-.8in 
\partial_t\widetilde{v} \in  L^1\big(0,\mathcal{T}; \mathcal{D}(e^{\tau(t) A} : H^{r-\frac{1}{2}})\big) \cap L^\infty\big(0,\mathcal{T}; L^2).\label{uniform-bound-Galerkin-time-derivative} 
\end{eqnarray*}

\subsection{Uniqueness of Solutions and Continuous Dependence on the Initial Data} In this section, we show the uniqueness of solutions and the continuous dependence on the initial data. Let $(\overline{v}_1, \widetilde{v}_1)$ and
$(\overline{v}_2, \widetilde{v}_2)$ be two strong solutions to system
(\ref{local-system-1})--(\ref{local-system-2}) with initial data $((\overline{v}_0)_1, (\widetilde{v}_0)_1)$ and
$((\overline{v}_0)_2, (\widetilde{v}_0)_2)$, respectively. Assume the radius of analyticity for initial data $((\overline{v}_0)_1, (\widetilde{v}_0)_1)$ is $\tau_{10}$, and for $((\overline{v}_0)_2, (\widetilde{v}_0)_2)$ is $\tau_{20}$. Let $\tau_0 = \min\{\tau_{10}, \tau_{20}\}$, and 
\begin{equation}\label{M}
  M = \max\Big\{\|e^{\tau_{10} A} (\overline{v}_0)_1\|_{H^r}^2 + \|e^{\tau_{10} A} (\widetilde{v}_0)_1\|_{H^r}^2, \|e^{\tau_{20} A} (\overline{v}_0)_2\|_{H^r}^2 + \|e^{\tau_{20} A} (\widetilde{v}_0)_2\|_{H^r}^2\Big\}.  
\end{equation}
Denote by $\overline{v}=\overline{v}_1-\overline{v}_2$ and $\widetilde{v}=\widetilde{v}_1-\widetilde{v}_2$. By virtue of (\ref{tau}) and (\ref{T}), we define
\begin{equation}\label{Ttilde}
    \widetilde{\tau}(t) = \tau_0 - 2tC_r(1+M), \;\;\; \widetilde{\mathcal{T}} = \frac{\tau_0}{1+2C_r(1+M)}.
\end{equation}
Here $C_r$ satisfies (\ref{Cr}). 
From previous sections, and by the definition of $\tau_0$ and $M$, we know
$   \|e^{\widetilde{\tau}(t) A} \overline{v}_i(t)\|_{H^r}^2 + \| e^{\widetilde{\tau}(t) A} \widetilde{v}_i(t)\|_{H^r}^2 \leq M
$
for $i=1,2$ and $t\in [0,\widetilde{\mathcal T}]$. From (\ref{local-system-1})--(\ref{local-system-2}), it is clear that
\begin{eqnarray}
&&\hskip-.8in
\partial_t \overline{v} + \mathbb{P}_h\Big( \overline{v}\cdot \nabla \overline{v}_1 + \overline{v}_2 \cdot \nabla \overline{v} \Big) + \mathbb{P}_h P_0\Big( (\nabla\cdot \widetilde{v})\widetilde{v}_1 + (\nabla\cdot \widetilde{v}_2)\widetilde{v} + \widetilde{v}\cdot \nabla \widetilde{v}_1 + \widetilde{v}_2 \cdot \nabla \widetilde{v}  \Big) = 0,\label{uniqueness-diffenrece-bar} \\
&&\hskip-.8in
\partial_t \widetilde{v} + \widetilde{v}\cdot \nabla \widetilde{v}_1 + \widetilde{v}_2\cdot \nabla \widetilde{v} + \widetilde{v}\cdot \nabla \overline{v}_1 + \widetilde{v}_2\cdot \nabla \overline{v} + \overline{v}\cdot \nabla \widetilde{v}_1 + \overline{v}_2\cdot \nabla \widetilde{v} - P_0\Big( (\nabla\cdot \widetilde{v})\widetilde{v}_1 + (\nabla\cdot \widetilde{v}_2)\widetilde{v}\nonumber\\
&&\hskip-.38in  + \widetilde{v}\cdot \nabla \widetilde{v}_1 + \widetilde{v}_2 \cdot \nabla \widetilde{v}  \Big) - \Big( \int_0^z\nabla\cdot \widetilde{v}(\boldsymbol{x}',s)ds \Big)\partial_z \widetilde{v}_1 - \Big( \int_0^z\nabla\cdot \widetilde{v}_2(\boldsymbol{x}',s)ds \Big)\partial_z \widetilde{v} + \Omega \widetilde{v}^\perp = 0. \label{uniqueness-diffenrece-tilde}
\end{eqnarray}
By the energy estimate, thanks to Lemma \ref{lemma-projection} and Lemma \ref{lemma-leray}, we have
\begin{eqnarray}
&&\hskip-.8in 
\frac{1}{2} \frac{d}{dt} \Big(\|e^{\widetilde{\tau}(t) A} \overline{v}(t)\|_{H^{r-\frac{1}{2}}}^2 + \|e^{\widetilde{\tau}(t) A} \widetilde{v}(t)\|_{H^{r-\frac{1}{2}}}^2 \Big) - \dot{\widetilde{\tau}} \Big( \|A^{r} e^{\widetilde{\tau} A}\overline{v}\|^2 +  \|A^{r} e^{\widetilde{\tau} A}\widetilde{v}\|^2 \Big)\nonumber\\
&&\hskip-.8in 
+ \Big\langle  \overline{v}\cdot \nabla \overline{v}_1 + \overline{v}_2 \cdot \nabla \overline{v} + (\nabla\cdot \widetilde{v})\widetilde{v}_1 + (\nabla\cdot \widetilde{v}_2)\widetilde{v} + \widetilde{v}\cdot \nabla \widetilde{v}_1 + \widetilde{v}_2 \cdot \nabla \widetilde{v}, \overline{v} \Big\rangle \nonumber\\
&&\hskip-.8in
+ \Big\langle A^{r-\frac{1}{2}} e^{\widetilde{\tau} A}\Big( \overline{v}\cdot \nabla \overline{v}_1 + \overline{v}_2 \cdot \nabla \overline{v} + (\nabla\cdot \widetilde{v})\widetilde{v}_1 + (\nabla\cdot \widetilde{v}_2)\widetilde{v} + \widetilde{v}\cdot \nabla \widetilde{v}_1 + \widetilde{v}_2 \cdot \nabla \widetilde{v}\Big), A^{r-\frac{1}{2}} e^{\widetilde{\tau} A} \overline{v} \Big\rangle \nonumber\\
&&\hskip-.8in
+ \Big\langle \widetilde{v}\cdot \nabla \widetilde{v}_1 + \widetilde{v}_2\cdot \nabla \widetilde{v} + \widetilde{v}\cdot \nabla \overline{v}_1 + \widetilde{v}_2\cdot \nabla \overline{v} + \overline{v}\cdot \nabla \widetilde{v}_1 + \overline{v}_2\cdot \nabla \widetilde{v}  \nonumber\\
&&\hskip-.58in
- \Big( \int_0^z\nabla\cdot \widetilde{v}(\boldsymbol{x}',s)ds \Big)\partial_z \widetilde{v}_1 - \Big( \int_0^z\nabla\cdot \widetilde{v}_2(\boldsymbol{x}',s)ds \Big)\partial_z \widetilde{v}  , \widetilde{v} \Big\rangle \nonumber\\
&&\hskip-.8in
+ \Big\langle A^{r-\frac{1}{2}} e^{\widetilde{\tau} A}\Big[ \widetilde{v}\cdot \nabla \widetilde{v}_1 + \widetilde{v}_2\cdot \nabla \widetilde{v} + \widetilde{v}\cdot \nabla \overline{v}_1 + \widetilde{v}_2\cdot \nabla \overline{v} + \overline{v}\cdot \nabla \widetilde{v}_1 + \overline{v}_2\cdot \nabla \widetilde{v}  \nonumber\\
&&\hskip-.58in
- \Big( \int_0^z\nabla\cdot \widetilde{v}(\boldsymbol{x}',s)ds \Big)\partial_z \widetilde{v}_1 - \Big( \int_0^z\nabla\cdot \widetilde{v}_2(\boldsymbol{x}',s)ds \Big)\partial_z \widetilde{v}\Big]  , A^{r-\frac{1}{2}} e^{\widetilde{\tau} A}\widetilde{v} \Big\rangle = 0. \label{uniqueness-1}
\end{eqnarray}
Thanks to the H\"older inequality, Young's inequality and the Sobolev inequality, since $r>\frac{5}{2}$, and noticing that $\overline{v}$ and $\widetilde{v}$ have zero mean over $\mathbb{T}^3$, one has
\begin{eqnarray}
&&\hskip-.8in
\Big|\Big\langle \overline{v}\cdot \nabla \overline{v}_1 + \overline{v}_2 \cdot \nabla \overline{v} + (\nabla\cdot \widetilde{v})\widetilde{v}_1 + (\nabla\cdot \widetilde{v}_2)\widetilde{v} + \widetilde{v}\cdot \nabla \widetilde{v}_1 + \widetilde{v}_2 \cdot \nabla \widetilde{v} , \overline{v}\Big\rangle \nonumber\\
&&\hskip-.68in
+ \Big\langle \widetilde{v}\cdot \nabla \widetilde{v}_1 + \widetilde{v}_2\cdot \nabla \widetilde{v} + \widetilde{v}\cdot \nabla \overline{v}_1 + \widetilde{v}_2\cdot \nabla \overline{v} + \overline{v}\cdot \nabla \widetilde{v}_1 + \overline{v}_2\cdot \nabla \widetilde{v}  \nonumber\\
&&\hskip-.58in
- \Big( \int_0^z\nabla\cdot \widetilde{v}(\boldsymbol{x}',s)ds \Big)\partial_z \widetilde{v}_1 - \Big( \int_0^z\nabla\cdot \widetilde{v}_2(\boldsymbol{x}',s)ds \Big)\partial_z \widetilde{v}  , \widetilde{v} \Big\rangle\Big| \nonumber\\
&&\hskip-.8in
\leq \widetilde{C}_r\Big( \|\overline{v}_1\|_{H^{r}} + \|\overline{v}_2\|_{H^{r}} + \|\widetilde{v}_1\|_{H^{r}} + \|\widetilde{v}_2\|_{H^{r}}\Big)\Big(\|\overline{v}\|_{H^{r-\frac{1}{2}}}^2 + \|\widetilde{v}\|_{H^{r-\frac{1}{2}}}^2 \Big)\nonumber\\
&&\hskip-.8in
\leq \widetilde{C}_r\Big( \|\overline{v}_1\|_{H^{r}} + \|\overline{v}_2\|_{H^{r}} + \|\widetilde{v}_1\|_{H^{r}} + \|\widetilde{v}_2\|_{H^{r}}\Big)\Big(\|A^{r} e^{\widetilde{\tau} A}\overline{v}\|^2 +  \|A^{r} e^{\widetilde{\tau} A}\widetilde{v}\|^2\Big), \label{uniqueness-L2}
\end{eqnarray}
where in the last step we apply the Poincar\'e inequality. Next, thanks to Lemma \ref{lemma-type1}--\ref{lemma-type3} and by Young's inequality, we have
\begin{eqnarray}
&&\hskip-.8in
\Big|\Big\langle A^{r-\frac{1}{2}} e^{\widetilde{\tau} A}\Big( \overline{v}\cdot \nabla \overline{v}_1 + \overline{v}_2 \cdot \nabla \overline{v} + (\nabla\cdot \widetilde{v})\widetilde{v}_1 + (\nabla\cdot \widetilde{v}_2)\widetilde{v} + \widetilde{v}\cdot \nabla \widetilde{v}_1 + \widetilde{v}_2 \cdot \nabla \widetilde{v}\Big), A^{r-\frac{1}{2}} e^{\widetilde{\tau} A} \overline{v}\Big\rangle \nonumber\\
&&\hskip-.8in
+ \Big\langle A^{r-\frac{1}{2}} e^{\widetilde{\tau} A}\Big[ \widetilde{v}\cdot \nabla \widetilde{v}_1 + \widetilde{v}_2\cdot \nabla \widetilde{v} + \widetilde{v}\cdot \nabla \overline{v}_1 + \widetilde{v}_2\cdot \nabla \overline{v} + \overline{v}\cdot \nabla \widetilde{v}_1 + \overline{v}_2\cdot \nabla \widetilde{v}  \nonumber\\
&&\hskip-.58in
- \Big( \int_0^z\nabla\cdot \widetilde{v}(\boldsymbol{x}',s)ds \Big)\partial_z \widetilde{v}_1 - \Big( \int_0^z\nabla\cdot \widetilde{v}_2(\boldsymbol{x}',s)ds \Big)\partial_z \widetilde{v}\Big]  , A^{r-\frac{1}{2}} e^{\widetilde{\tau} A}\widetilde{v} \Big\rangle\Big|\nonumber\\
&&\hskip-.8in
\leq C_{r-\frac{1}{2}}\Big( \|e^{\widetilde{\tau} A}\overline{v}_1\|_{H^r} +  \|e^{\widetilde{\tau} A}\widetilde{v}_1\|_{H^r} + \|e^{\widetilde{\tau} A}\overline{v}_2\|_{H^r} +  \| e^{\widetilde{\tau} A}\widetilde{v}_2\|_{H^r}\Big)  \Big(\|A^{r} e^{\widetilde{\tau} A}\overline{v}\|^2 +  \|A^{r} e^{\widetilde{\tau} A}\widetilde{v}\|^2\Big).
\label{uniqueness-analytic}
\end{eqnarray}
Combining (\ref{uniqueness-1})--(\ref{uniqueness-analytic}) and thanks to (\ref{Cr}), we have
\begin{eqnarray*}
&&\hskip-.8in 
\frac{1}{2} \frac{d}{dt} \Big(\|e^{\widetilde{\tau}(t) A} \overline{v}(t)\|_{H^{r-\frac{1}{2}}}^2 + \|e^{\widetilde{\tau}(t) A} \widetilde{v}(t)\|_{H^{r-\frac{1}{2}}}^2 \Big)  \nonumber\\
&&\hskip-.9in 
\leq \Big[\dot{\widetilde{\tau}} + \frac{1}{2}C_r\Big( \|e^{\widetilde{\tau} A}\overline{v}_1\|_{H^r} +  \|e^{\widetilde{\tau} A}\widetilde{v}_1\|_{H^r} + \|e^{\widetilde{\tau} A}\overline{v}_2\|_{H^r} +  \| e^{\widetilde{\tau} A}\widetilde{v}_2\|_{H^r}\Big)  \Big]\nonumber\\
&&\hskip-.58in
\times \Big( \|A^{r} e^{\widetilde{\tau} A}\overline{v}\|^2 +  \|A^{r} e^{\widetilde{\tau} A}\widetilde{v}\|^2 \Big). \label{estimate-uniqueness}
\end{eqnarray*}
Since $\|e^{\widetilde{\tau} A} \overline{v}_i\|_{H^r}^2 + \| e^{\widetilde{\tau} A} \widetilde{v}_i\|_{H^r}^2 \leq M$ for $i=1,2$, by the Cauchy–Schwarz inequality, we know that
\begin{eqnarray*}
&&\hskip-.58in
\dot{\widetilde{\tau}} + \frac{1}{2}C_r\Big( \|e^{\widetilde{\tau} A}\overline{v}_1\|_{H^r} +  \|e^{\widetilde{\tau} A}\widetilde{v}_1\|_{H^r} + \|e^{\widetilde{\tau} A}\overline{v}_2\|_{H^r} +  \| e^{\widetilde{\tau} A}\widetilde{v}_2\|_{H^r}\Big) \nonumber\\
&&\hskip-.68in
\leq -2C_r(1+M) + \sqrt{2} C_r\sqrt{M} \leq (\frac{\sqrt{2}}{2}-2)C_r(1+M) < 0
\end{eqnarray*}
for $t\in[0,\widetilde{\mathcal{T}}]$. Therefore, for $t\in[0,\widetilde{\mathcal{T}}]$, we have
\begin{eqnarray}
&&\hskip-.58in
\|e^{\widetilde{\tau}(t) A} \overline{v}(t)\|_{H^{r-\frac{1}{2}}}^2 + \|e^{\widetilde{\tau}(t) A} \widetilde{v}(t)\|_{H^{r-\frac{1}{2}}}^2 \leq \|e^{\widetilde{\tau}_0 A} \overline{v}_0\|_{H^{r-\frac{1}{2}}}^2 +   \|e^{\widetilde{\tau}_0 A} \widetilde{v}_0\|_{H^{r-\frac{1}{2}}}^2. \label{continuous-dependence}
\end{eqnarray}

The above inequality proves the continuous dependence of the
solutions on the initial data, and in particular, when
$\overline{v}_0=\widetilde{v}_0=0$ and $\tau_{10}=\tau_{20}$, we have $\overline{v}=\widetilde{v}=0$ for all $t\in [0,\widetilde{\mathcal{T}}]$. Moreover, from (\ref{T}), (\ref{Ttilde}), and the definition of $M$ in (\ref{M}), we know $\widetilde{\mathcal{T}} = \mathcal{T}$. Therefore, the solution is unique, and this completes the proof of Theorem \ref{theorem-local}.
\begin{remark}\label{remark-zero-average}
  In case that $\int_{\mathbb{T}^3} v(\boldsymbol{x}) d\boldsymbol{x}= \int_{\mathbb{T}^2} \overline{v}(\boldsymbol{x}') d\boldsymbol{x}' \neq 0$, the only change in system (\ref{local-system-1})--(\ref{local-system-2}) is in (\ref{local-system-1}) which will become
  \begin{equation}\label{local-system-1-nonzero-average}
      \partial_t \overline{v} + \mathbb{P}_h \Big(\overline{v}\cdot \nabla \overline{v}\Big) + \mathbb{P}_h P_0 \Big((\nabla\cdot \widetilde{v}) \widetilde{v} + \widetilde{v}\cdot \nabla \widetilde{v} \Big) + \Omega \int_{\mathbb{T}^2} \overline{v}^\perp (\boldsymbol{x}') d\boldsymbol{x}'  = 0. 
  \end{equation}
  The additional term $\Omega \int_{\mathbb{T}^2} \overline{v}^\perp (\boldsymbol{x}') d\boldsymbol{x}' $ appearing in (\ref{local-system-1-nonzero-average}) does not change the energy estimates. Since 
 $
      \Big(\int_{\mathbb{T}^2} \overline{v}^\perp (\boldsymbol{x}') d\boldsymbol{x}'  \Big) \cdot  \Big( \int_{\mathbb{T}^2} \overline{v} (\boldsymbol{x}') d\boldsymbol{x}' \Big) =0,
  $
  the conservation of $L^2$ norm does not change. Since $\Omega \int_{\mathbb{T}^2} \overline{v}^\perp (\boldsymbol{x}') d\boldsymbol{x}' $ is independent of the spatial variables, it will disappear when we apply the operator $A^r e^{\tau A}$. Therefore, this additional term does not affect the higher order energy estimates. Thus, when $\int_{\mathbb{T}^2} \overline{v}(\boldsymbol{x}') d\boldsymbol{x}' \neq 0$, we still have the same results.
\end{remark}

\section{Long time existence of solutions}
In this section, we establish the long time existence of solutions to system (\ref{local-system-1})--(\ref{local-system-2}) provided that the analytic norm of $\widetilde{v}_0$ is small. Notice that we do not assume any smallness in $\overline{v}_0$, and therefore, we do not have smallness in $v_0$. The motivation is that, when $\widetilde{v} = 0$, system (\ref{local-system-1})--(\ref{local-system-2}) reduces to $2D$ Euler equations, for which we have global solution in the space of analytic functions (see \cite{LO97}). Therefore, if $\widetilde{v}_0$ is small in the analytic norm, one can expect that the solution to system (\ref{local-system-1})--(\ref{local-system-2}) exists for a long time. In section 6, however, we will demonstrate that system (\ref{local-system-1})--(\ref{local-system-2}) exhibits a different behavior for large value of $|\Omega|$ when we assume $\widetilde{v}_0$ is small only in Sobolev norm, but not in the analytic norm.

\subsection{$2D$ Euler equations}
Consider the $2D$ Euler equations in $\mathbb{T}^3$:
\begin{eqnarray}
&&\hskip-.68in \partial_t \overline{V} + \overline{V}\cdot \nabla \overline{V} + \nabla P =0, \label{Euler-1}\\
&&\hskip-.68in \nabla\cdot \overline{V}=0, \label{Euler-2}\\
&&\hskip-.68in  \overline{V}(0) = \overline{V}_0. \label{Euler-3}
\end{eqnarray}
Here $\overline{V}$ depends only on the horizontal variables $\boldsymbol{x}'$. The global existence of solutions to system (\ref{Euler-1})--(\ref{Euler-3}) in Sobolev spaces $H^r$ with $r\geq 3$ is a classical result, see, e.g., \cite{BM02}. Moreover, from equation (3.84) in \cite{BM02}, for $r\geq 3$, we have
\begin{equation}\label{Euler-Sobolev-estimate-1}
    \frac{d}{dt}\|\overline{V}\|_{H^r} \leq C_r \|\overline{V}\|_{H^r}(1+\ln^+\|\overline{V}\|_{H^r}).
\end{equation}
Let $\|\overline{V}_0\|_{H^r} \leq  M$ for some $M\geq 0$. Since $\ln^+ x + 1 \leq 2\ln(x+e)$, by setting $W(t)=\|\overline{V}(t)\|_{H^r} + e$, from (\ref{Euler-Sobolev-estimate-1}), we have
$
    \frac{d}{dt} W \leq C_r W \ln W.
$
Therefore, we get the following bound:
\begin{equation}\label{Euler-Sobolev-estimate-2}
    \|\overline{V}(t)\|_{H^r} \leq W(t) \leq W(0)^{e^{C_r t}} = (\|\overline{V}_0\|_{H^r} +e)^{e^{C_r t}} \leq (M+e)^{e^{C_r t}} =: \theta_{M,r}(t).
\end{equation}

We need the following lemma from \cite{LO97}. 
\begin{lemma} \label{lemma-non-small}
For $f, g\in \mathcal{D}(e^{\tau A}: H^{r+\frac{1}{2}})$ where $r>\frac{5}{2}$ and $\tau\geq0$, one has
\begin{eqnarray*}
&&\hskip-.78in
\Big|\langle A^r e^{\tau A} (f\cdot \nabla g), A^r e^{\tau A} g  \rangle\Big| \leq C_r (\|A^r f\| \|A^{r}  g\|^2 + \|\nabla\cdot f\|_{L^\infty} \|A^{r} e^{\tau A} g\|^2 ) \nonumber\\
&&\hskip.98in
+ C_r \tau \|A^{r+\frac{1}{2}} e^{\tau A} f\| \|A^{r+\frac{1}{2}} e^{\tau A} g\|^2.
\end{eqnarray*}
Moreover, if $r>3$, then $\|A^{r+\frac{1}{2}} e^{\tau A} f\|$ can be replaced by $\|A^{r} e^{\tau A} f\|$.
\end{lemma}
Based on Lemma \ref{lemma-non-small}, the authors in \cite{LO97} proved the global existence of solutions to system (\ref{Euler-1})--(\ref{Euler-3}) for initial data in the space of analytic functions. For completion, we state it here, with slight difference from the original statement in \cite{LO97}.
\begin{proposition}\label{global-Euler}
    Assume $\overline{V}_0 \in \mathcal{S}\cap \mathcal{D}(e^{\tau_0 A}: H^r(\mathbb{T}^3))$ with $r>3$ and $\tau_0 >0$, and suppose that $\|e^{\tau_0 A} \overline{V}_0\|_{H^r} \leq M$ for some $M\geq 0$. Then there exists a non-increasing function 
    \begin{eqnarray}
    \tau(t) = \tau_0 \exp\Big( -C_r \int_0^t    h(s)ds\Big), \label{tau-Euler}
    \end{eqnarray}
    where 
   $
      h^2(t) := \|e^{\tau_0 A} \overline{V}_0\|_{H^r}^2 + C_r\int_0^t \theta_{M,r}^3(s)ds,  
    $
    and $\theta_{M,r}(t)$ defined in (\ref{Euler-Sobolev-estimate-2}),
    such that for any given time $\mathcal{T}>0$, there exists a unique solution 
   $
        \overline{V}\in L^\infty(0, \mathcal{T}; \mathcal{D}(e^{\tau(t) A}: H^r(\mathbb{T}^3)))
    $
    to system (\ref{Euler-1})--(\ref{Euler-3}). Moreover, there exist constants $C >1$ and $C_r>1$ such that 
\begin{eqnarray}
 \|e^{\tau(t) A} \overline{V}(t)\|_{H^r}^2 \leq h^2(t) \leq C ^{\exp(C_r t)}.  \label{estimate-Euler}
\end{eqnarray}
\end{proposition}
\subsection{Long time existence of the $3D$ inviscid PEs}
The following is the main theorem of this section, which concerns the long time existence of solutions to system (\ref{local-system-1})--(\ref{local-system-2}) in the case when the analytic norm of $\widetilde{v}_0$ is small.
\begin{theorem} \label{theorem-longtime} 
Assume $\overline{v}_0 \in \mathcal{S}\cap\mathcal{D}(e^{\tau_0 A}: H^{r+1}(\mathbb{T}^3))$, $\widetilde{v}_0 \in \mathcal{S}\cap \mathcal{D}(e^{\tau_0 A}: H^{r}(\mathbb{T}^3))$ with $r>\frac{5}{2}$ and $\tau_0>0$. Let $\Omega \in \mathbb{R}$ be arbitrary and fixed. Let $M\geq0$ and $\epsilon\geq 0$, and suppose that $\|e^{\tau_0 A}\overline{v}_0\|_{H^{r+1}} \leq M$ and $\|e^{\tau_0 A}\widetilde{v}_0\|_{H^{r}} \leq \epsilon$. Then there are constants $C  > 1$ and $C_r > 1$, and a function $K(t) = C ^{\exp(C_r t)}$, such that if $\mathcal{T}=\mathcal{T}(\tau_0, \epsilon, M, r)$ satisfies \begin{eqnarray}\label{T-longtime-type1}
 \int_0^\mathcal{T} e^{K(s)}ds = \frac{\tau_0}{2\epsilon},
\end{eqnarray}
then the unique solution obtained in Theorem \ref{theorem-local} satisfies $(\overline{v},\widetilde{v}) \in  L^\infty(0,\mathcal{T};\mathcal{S}\cap\mathcal{D}(e^{\tau(t) A}: H^r(\mathbb{T}^3)))$, with 
\begin{eqnarray}
\tau(t) = e^{ - \int_0^t K(s)ds} (\tau_0 - \epsilon   \int_0^t e^{K(s)}ds ). \label{longtime-tau-1}
\end{eqnarray}
In particular, from (\ref{T-longtime-type1}), $\mathcal{T}\gtrsim \ln(\ln(\ln(\frac{1}{\epsilon})))\rightarrow \infty$, as $\epsilon \rightarrow 0^+$.
\end{theorem}
Thanks to Lemma \ref{lemma-decomposition}, we immediately have the following corollary.
\begin{corollary}
Assume $v_0 \in \mathcal{S}\cap \mathcal{D}(e^{\tau_0 A}: H^{r+1}(\mathbb{T}^3))$, and the conditions of Theorem \ref{theorem-longtime} hold. Then the unique solution obtained in Corollary \ref{corollary-local} satisfies $v \in  L^\infty(0,\mathcal{T};\mathcal{S}\cap\mathcal{D}(e^{\tau(t) A}: H^r(\mathbb{T}^3)))$, with $\mathcal{T}$ defined in (\ref{T-longtime-type1}) and $\tau$ defined in (\ref{longtime-tau-1}). 
\end{corollary}
\begin{remark}
For the proof of Theorem \ref{theorem-longtime}, we only establish formal energy estimates. However, these formal estimates can be justified rigorously by establishing them first for the Galerkin approximation system and then passing to the limit using the Aubin-Lions compactness theorem.
\end{remark}
\begin{remark}\label{remark-kt}
The constants $C $ and $C_r$ in $K(t)$, in the proof below, may change from step to step, and are always taken to be larger than $1$. When necessary, we use $K_1(t), K_2(t), ...$ to emphasize the changes. At the end, we choose some suitable and large enough $C $ and $C_r$ for the $K(t)$ in Theorem \ref{theorem-longtime}. 
\end{remark}
\begin{proof}[Proof of Theorem \ref{theorem-longtime}]
 Let $\overline{V}$ be the unique global solution to the $2D$ Euler equations (\ref{Euler-1})--(\ref{Euler-3}) in the space $\mathcal{D}(e^{\tau_1(t) A}: H^{r+1}(\mathbb{T}^3))$,
with initial condition $\overline{V}_0 = \overline{v}_0$ and $\tau_1(t)$ satisfying (\ref{tau-Euler}). Let $\overline{\phi} = \overline{v}-\overline{V}$. Then from \eqref{local-system-1}--\eqref{local-system-2} one has
\begin{eqnarray*}
&&\hskip-.8in 
\partial_t \overline{\phi} + \mathbb{P}_h\Big(\overline{\phi}\cdot \nabla \overline{\phi} + \overline{\phi}\cdot \nabla \overline{V} + \overline{V}\cdot \nabla \overline{\phi}\Big) + \mathbb{P}_h P_0\Big( (\nabla\cdot \widetilde{v}) \widetilde{v} 
+ \widetilde{v}\cdot \nabla \widetilde{v} \Big)  = 0, \label{difference-original-bar} \\
\nonumber \\
&&\hskip-.8in 
\partial_t \widetilde{v} + \widetilde{v} \cdot \nabla \widetilde{v} + \overline{\phi}\cdot \nabla \widetilde{v} + \overline{V}\cdot \nabla \widetilde{v} 
+ \widetilde{v} \cdot\nabla \overline{\phi} + \widetilde{v} \cdot\nabla \overline{V} 
 - P_0\Big( (\nabla\cdot \widetilde{v}) \widetilde{v} 
+ \widetilde{v}\cdot \nabla \widetilde{v} \Big) \nonumber \\
&&\hskip1.4in 
- \big(\int_0^z \nabla\cdot \widetilde{v}(s) ds\big)\partial_z \widetilde{v} + \Omega \widetilde{v}^\perp  = 0,
\label{difference-original-tilde}
\end{eqnarray*}
with initial condition
\begin{eqnarray*}
\overline{\phi}(0) = \overline{v}_0 - \overline{V}_0 = 0, \hspace{0.2in} \widetilde{v}(0) = \widetilde{v}_0. \label{difference-initial}
\end{eqnarray*}
Thanks to Lemma \ref{lemma-projection} and Lemma \ref{lemma-leray}, we have
\begin{eqnarray*}
&&\hskip-.58in \frac{1}{2} \frac{d}{dt} \|A^r e^{\tau A} \overline{\phi} \|^2 = \dot{\tau} \|A^{r+\frac{1}{2}} e^{\tau A} \overline{\phi} \|^2 - \Big\langle A^r e^{\tau A} (\overline{\phi}\cdot \nabla \overline{\phi}), A^r e^{\tau A} \overline{\phi} \Big\rangle - \Big\langle A^r e^{\tau A} (\overline{\phi}\cdot \nabla \overline{V}), A^r e^{\tau A} \overline{\phi} \Big\rangle \nonumber \\
&&\hskip-.38in
- \Big\langle A^r e^{\tau A} (\overline{V}\cdot \nabla \overline{\phi}), A^r e^{\tau A} \overline{\phi} \Big\rangle -  \Big\langle A^r e^{\tau A} (\widetilde{v}\cdot \nabla \widetilde{v}), A^r e^{\tau A} \overline{\phi} \Big\rangle  - \Big\langle A^r e^{\tau A} ((\nabla\cdot \widetilde{v})\widetilde{v}), A^r e^{\tau A} \overline{\phi} \Big\rangle, \label{energy-phibar}
\end{eqnarray*}
\begin{eqnarray*}
&&\hskip-.58in \frac{1}{2} \frac{d}{dt} \|A^r e^{\tau A} \widetilde{v} \|^2 = \dot{\tau} \|A^{r+\frac{1}{2}} e^{\tau A} \widetilde{v} \|^2 - \Big\langle A^r e^{\tau A} (\widetilde{v}\cdot \nabla \widetilde{v}), A^r e^{\tau A} \widetilde{v} \Big\rangle- \Big\langle A^r e^{\tau A} (\overline{\phi}\cdot \nabla \widetilde{v}), A^r e^{\tau A} \widetilde{v} \Big\rangle \nonumber \\
&&\hskip-.38in
- \Big\langle A^r e^{\tau A} (\overline{V}\cdot \nabla \widetilde{v}), A^r e^{\tau A} \widetilde{v} \Big\rangle
-  \Big\langle A^r e^{\tau A} (\widetilde{v}\cdot \nabla \overline{\phi}), A^r e^{\tau A} \widetilde{v} \Big\rangle  - \Big\langle A^r e^{\tau A} (\widetilde{v}\cdot\nabla\overline{V}), A^r e^{\tau A} \widetilde{v} \Big\rangle \nonumber\\
&&\hskip.38in + \Big\langle A^r e^{\tau A} \big((\int_0^z \nabla\cdot \widetilde{v}(\boldsymbol{x}',s)ds)\partial_z \widetilde{v} \big), A^r e^{\tau A} \widetilde{v} \Big\rangle 
. \label{energy-vtilde}
\end{eqnarray*}
By using Lemma \ref{lemma-type1}--\ref{lemma-type3}, we have
\begin{eqnarray*}
&&\hskip-.58in
\Big|\Big\langle A^r e^{\tau A} (\overline{\phi}\cdot \nabla \overline{\phi}), A^r e^{\tau A} \overline{\phi} \Big\rangle\Big| + \Big|  \Big\langle A^r e^{\tau A} (\widetilde{v}\cdot \nabla \widetilde{v}), A^r e^{\tau A} \overline{\phi} \Big\rangle\Big| + \Big|\Big\langle A^r e^{\tau A} ((\nabla\cdot \widetilde{v})\widetilde{v}), A^r e^{\tau A} \overline{\phi} \Big\rangle \Big| \nonumber\\
&&\hskip-.38in 
+  \Big|\Big\langle A^r e^{\tau A} (\widetilde{v}\cdot \nabla \widetilde{v}), A^r e^{\tau A} \widetilde{v} \Big\rangle \Big| 
+ \Big|\Big\langle A^r e^{\tau A} (\overline{\phi}\cdot \nabla \widetilde{v}), A^r e^{\tau A} \widetilde{v} \Big\rangle \Big| + \Big| \Big\langle A^r e^{\tau A} (\widetilde{v}\cdot \nabla \overline{\phi}), A^r e^{\tau A} \widetilde{v} \Big\rangle \Big| \nonumber\\
&&\hskip-.38in 
+ \Big| \Big\langle A^r e^{\tau A} \big((\int_0^z \nabla\cdot \widetilde{v}(\boldsymbol{x}',s)ds)\partial_z \widetilde{v} \big), A^r e^{\tau A} \widetilde{v} \Big\rangle \Big| \nonumber\\
&&\hskip-.68in 
\leq C_r (\|A^{r} e^{\tau A} \overline{\phi} \| + \|A^{r} e^{\tau A} \widetilde{v} \|) (\|A^{r+\frac{1}{2}} e^{\tau A} \overline{\phi} \|^2 + \|A^{r+\frac{1}{2}} e^{\tau A} \widetilde{v} \|^2).
\end{eqnarray*}
Here we use the fact that $\overline{\phi}$ has zero mean value since $\overline{V}$ and $\overline{v}$ both have zero mean value, and $\widetilde{v}$ has zero mean value since $\overline{\widetilde{v}} = 0$. By virtue of Lemma \ref{lemma-non-small}, since $\nabla\cdot \overline{V}=0$, one obtains
\begin{eqnarray*}
&&\hskip-.58in
\Big|\Big\langle A^r e^{\tau A} (\overline{V}\cdot \nabla \overline{\phi}), A^r e^{\tau A} \overline{\phi} \Big\rangle \Big| + \Big|\Big\langle A^r e^{\tau A} (\overline{V}\cdot \nabla \widetilde{v}), A^r e^{\tau A} \widetilde{v}  \Big\rangle \Big| \nonumber\\
&&\hskip-.73in
\leq C_r \|A^r \overline{V}\| (\|A^r \overline{\phi}\|^2 + \|A^r \widetilde{v}\|^2 ) + C_r\tau \|A^{r+\frac{1}{2}} e^{\tau A} \overline{V}\| (\|A^{r+\frac{1}{2}} e^{\tau A} \overline{\phi}\|^2 + \|A^{r+\frac{1}{2}} e^{\tau A} \widetilde{v}\|^2).
\end{eqnarray*}
From Lemma \ref{lemma-banach-algebra}, thanks to the Cauchy–Schwarz inequality and since $\widetilde{v}$ and $\overline{\phi}$ have zero mean, we have
\begin{eqnarray*}
&&\hskip-.58in
\Big| \Big\langle A^r e^{\tau A} (\widetilde{v}\cdot\nabla\overline{V}), A^r e^{\tau A} \widetilde{v} \Big\rangle \Big| + \Big| \Big\langle A^r e^{\tau A} (\overline{\phi}\cdot\nabla\overline{V}), A^r e^{\tau A} \overline{\phi} \Big\rangle \Big| \nonumber\\
&&\hskip-.68in
\leq C_r \| e^{\tau A} \overline{V}\|_{H^{r+1}} (\|A^{r} e^{\tau A} \overline{\phi}\|^2 + \|A^{r} e^{\tau A} \widetilde{v}\|^2 ).
\end{eqnarray*}
Combining all the estimates above, we have
\begin{eqnarray}
&&\hskip-.58in \frac{1}{2} \frac{d}{dt} (\|A^r e^{\tau A} \overline{\phi} \|^2 + \|A^r e^{\tau A} \widetilde{v} \|^2) \nonumber \\
&&\hskip-.68in
\leq \Big(\dot{\tau}+ C_r(\|A^r e^{\tau A} \overline{\phi} \| + \|A^r e^{\tau A} \widetilde{v} \|) + C_r\tau \| e^{\tau A} \overline{V}\|_{H^{r+1}} \Big) \Big(\|A^{r+\frac{1}{2}} e^{\tau A} \overline{\phi} \|^2 + \|A^{r+\frac{1}{2}} e^{\tau A} \widetilde{v} \|^2 \Big) \nonumber \\
&&\hskip-.58in 
+  C_r\| e^{\tau A} \overline{V}\|_{H^{r+1}}\Big(\|A^{r} e^{\tau A} \overline{\phi}\|^2 + \|A^{r} e^{\tau A} \widetilde{v}\|^2 \Big).
\label{energy-together1}
\end{eqnarray}
As indicated in Remark \ref{remark-kt}, we will use $K_0, K_1, K_2,...$ to indicate the change in $K(t)$ from step to step, and all of them are increasing double exponentially in $t$.
Recall that $\tau_1$ is defined by (\ref{tau-Euler}). Indeed, there exists a function $K_0(t)$ such that $\tau_1(t) \geq \tau_0 e^{-\int_0^t K_0(s)ds}$. Let $\tau \leq \tau_1.$ Recall from (\ref{estimate-Euler}), we have
\begin{equation*}
    \| e^{\tau(t) A} \overline{V}(t)\|_{H^{r+1}} \leq  \| e^{\tau_1(t) A} \overline{V}(t)\|_{H^{r+1}} \leq C ^{\exp(\widetilde{C_r} t)}=: K_1(t).
\end{equation*}
Denote by 
\begin{equation*}
    F = \|A^r e^{\tau A} \overline{\phi} \|^2 + \|A^r e^{\tau A} \widetilde{v} \|^2, \;\;\;\;
    G = \|A^{r+\frac{1}{2}} e^{\tau A} \overline{\phi} \|^2 + \|A^{r+\frac{1}{2}} e^{\tau A} \widetilde{v} \|^2 . 
\end{equation*}
We can rewrite (\ref{energy-together1}) as
\begin{eqnarray*}
&&\hskip-.58in \frac{d}{dt} F \leq 2(\dot{\tau} + C_r F^{\frac{1}{2}} + \tau K_2)G + K_2 F.
\label{energy-together2}
\end{eqnarray*}
Notice that when $\tau$ satisfies
\begin{equation}\label{condition-tau}
    \dot{\tau} + C_r F^{\frac{1}{2}} +  \tau  K_2 \leq 0, 
\end{equation}
we have
\begin{eqnarray}
F(t) \leq F(0)e^{\int_0^t K_2(s)ds} \leq F(0)e^{K_3(t)}, \label{longtime-analytic}
\end{eqnarray}
and therefore 
$
C_r F(t)^{\frac{1}{2}} \leq F(0)^{\frac{1}{2}} e^{K_4 (t)}.
$
Notice that $F(0) = \|A^r e^{\tau_0 A} \widetilde{v}_0 \|^2 \leq \| e^{\tau_0 A} \widetilde{v}_0\|^2_{H^{r}} \leq  \epsilon^2 $. From (\ref{condition-tau}), we require that
\begin{eqnarray}\label{dtau-1}
\frac{d}{dt}(\tau e^{ \int_0^t K_2(s)ds}) + \epsilon e^{ \int_0^t K_2(s) ds} e^{ K_4(t)}\leq 0  .
\end{eqnarray}
It is clear that
$e^{ \int_0^t K_2(s) ds} e^{ K_4(t)}  \leq  e^{K_5(t)}$
for some new $K_5(t)$. Therefore, instead of (\ref{dtau-1}), we require that
\begin{eqnarray}\label{dtau-2}
\frac{d}{dt}(\tau e^{\int_0^t K_2(s)ds}) + \epsilon e^{ K_5(t)} \leq 0.
\end{eqnarray}
Integrating (\ref{dtau-2}) from $0$ to $t$ in time, we have
\begin{eqnarray}\label{dtau-3}
\tau(t) e^{ \int_0^t K_2(s) ds} \leq \tau_0 - \epsilon \int_0^t  e^{ K_5(s)} ds.
\end{eqnarray}
Recall that we also need $\tau(t)\leq \tau_1(t)$ and we know that $\tau_1(t) \geq \tau_0 e^{-\int_0^t K_0(s)ds}$. Therefore, for a new and suitable function $K(t)$, we can set 
\begin{eqnarray}
\tau(t) = e^{ - \int_0^t K(s)ds} (\tau_0 - \epsilon   \int_0^t e^{K(s)}ds ) \label{longtime-tau}
\end{eqnarray}
such that $\tau(t)$ satisfies the condition in (\ref{condition-tau}) and also $\tau(t)\leq \tau_1(t)$. One can see $\tau(t)>0$ on $t\in[0,\mathcal{T}]$ when $\mathcal{T}$ satisfies
$
 \int_0^\mathcal{T} e^{K(s)}ds = \frac{\tau_0}{2\epsilon}.
$
Since $K(t)$ is double exponential in time, and $\int_0^\mathcal{T} e^{K(s)}ds \leq \mathcal{T} e^{K(\mathcal{T})} \leq e^{2K(\mathcal{T})}$, we have $\mathcal{T}\gtrsim \ln(\ln(\ln(\frac{1}{\epsilon})))\rightarrow\infty $ as $\epsilon \rightarrow 0^+$. 

From (\ref{longtime-analytic}), since $\overline{\phi}$ and $\widetilde{v}$ have zero mean, we can apply the Poincar\'e inequality to obtain 
\begin{eqnarray}\label{longtime-convergence-inequality}
\|e^{\tau(t) A} \overline{\phi}(t)\|_{H^r}^2 + \|e^{\tau(t) A} \widetilde{v}(t)\|_{H^r}^2 \leq \epsilon^2 e^{K(t)}
\end{eqnarray}
when $K(t)$ is chosen suitably, on $t\in[0,\mathcal{T}],$ with $\tau(t)$ defined by (\ref{longtime-tau}). From (\ref{estimate-Euler}), and since $\tau \leq \tau_1$, we know $\|e^{\tau(t) A} \overline{V}(t)\|_{H^r}$ is also bounded on $t\in[0,\mathcal{T}]$. By triangle inequality, we have
\begin{eqnarray*}
\|e^{\tau(t) A} \overline{v}(t)\|_{H^r} + \|e^{\tau(t) A} \widetilde{v}(t)\|_{H^r}\leq \|e^{\tau(t) A} \overline{\phi}(t)\|_{H^r} + \|e^{\tau(t) A} \overline{V}(t)\|_{H^r} +  \|e^{\tau(t) A} \widetilde{v}(t)\|_{H^r} < \infty
\end{eqnarray*}
for $t\in[0,\mathcal{T}]$.
Therefore, the time of existence of the solution to system (\ref{local-system-1})--(\ref{local-system-2}) satisfies (\ref{T-longtime-type1}).
\end{proof}

\subsection{Convergence to the $2D$ Euler equations}
Based on Theorem \ref{theorem-longtime}, we have the following result concerning the convergence of solutions of the $3D$ inviscid PEs (\ref{local-system-1})--(\ref{local-system-2}) to solutions of the $2D$ Euler equations (\ref{Euler-1})--(\ref{Euler-3}) in the space of analytic functions.
\begin{theorem} \label{convergence}
Assume a sequence of initial data $\{\overline{v}^{n}_0 = \overline{v}_0 \}_{n\in\mathbb{N}} \subset \mathcal{S}\cap \mathcal{D}(e^{\tau_0 A}: H^{r+1}(\mathbb{T}^3))$ and $\{ \widetilde{v}^{n}_0 \}_{n\in\mathbb{N}} \subset \mathcal{S}\cap \mathcal{D}(e^{\tau_0 A}: H^{r}(\mathbb{T}^3))$ with $r>\frac{5}{2}$ and $\tau_0 >0$. Let $\Omega \in \mathbb{R}$ be arbitrary and fixed. Suppose $\|e^{\tau_0 A}\overline{v}_0\|_{H^{r+1}} \leq M$ for some $M \geq 0$, and $\|e^{\tau_0 A} \widetilde{v}^{n}_0\|_{H^{r}} \leq\epsilon_n$ with $\epsilon_n \rightarrow 0$, as $n\rightarrow \infty$. Then there are constants $C  > 1$, $C_r > 1$, and a function $K(t) = C ^{\exp(C_r t)}$, such that for each $n\in\mathbb{N}$, if the function $\tau^n(t)$ and the time $\mathcal{T}_n$ satisfy
\begin{eqnarray*}
\tau^n(t) = e^{-\int_0^t K(s)ds}(\tau_0 - \epsilon_n \int_0^t e^{K(s)}ds), \;\;\;\;\;\; \int_0^{\mathcal{T}_n} e^{K(s)}ds = \frac{\tau_0}{2\epsilon_n},
\end{eqnarray*}
the solution to system (\ref{local-system-1})--(\ref{local-system-2}) with initial data $(\overline{v}^{n}_0, \widetilde{v}^{n}_0)$ satisfies  $(\overline{v}^n, \widetilde{v}^n) \in  L^\infty(0,\mathcal{T}_n;\mathcal{S} \cap\mathcal{D}(e^{\tau^n A}: H^{r}(\mathbb{T}^3)))$. Let $\overline{V}\in L^\infty(0,\infty;  \mathcal{S}\cap \mathcal{D}(e^{\tau^0(t) A}: H^{r}(\mathbb{T}^3)))$ be the unique global solution to the $2D$ Euler equations (\ref{Euler-1})--(\ref{Euler-3}) with initial data $\overline{V}(0) = \overline{v}_0$. Then, $(\overline{v}^n, \widetilde{v}^n)$ converges to $\overline{V}$ for $t\in[0,\mathcal{T}_0]$, as $n\rightarrow \infty$, in the following sense:
\begin{eqnarray}
\|e^{\tau^{0}(t) A} (\overline{v}^n + \widetilde{v}^n - \overline{V})(t) \|_{H^r}  \leq \epsilon_n e^{K(t)} \rightarrow 0, \;\; \text{as}\;\; n\rightarrow \infty.\label{convergence-estimate}
\end{eqnarray}
\end{theorem}
\begin{proof}
Denote by $\overline{\phi}^n = \overline{v}^n- \overline{V}$. By virtue of the proof of Theorem \ref{theorem-longtime}, we just need to prove the estimate (\ref{convergence-estimate}). Since $\tau^0(t) \leq \tau^n(t)$ for any $n\in \mathbb{N}$, from \eqref{longtime-convergence-inequality}, one has
\begin{eqnarray*}
\|e^{\tau^0(t) A} \widetilde{v}^{n}(t)\|_{H^{r}} + \|e^{\tau^0(t) A} \overline{\phi}^{n}(t)\|_{H^{r}} \leq 
\|e^{\tau^n(t) A} \widetilde{v}^{n}(t)\|_{H^{r}} + \|e^{\tau^n(t) A} \overline{\phi}^{n}(t)\|_{H^{r}}  \leq \epsilon_n e^{K(t)}
\end{eqnarray*}
when the function $K(t)$ is chosen suitably. Therefore, we have
\begin{equation*}
    \|e^{\tau^{0}(t) A} (\overline{v}^n + \widetilde{v}^n - \overline{V})(t) \|_{H^r} \leq \|e^{\tau^{0}(t) A} \widetilde{v}^{n}(t)\|_{H^{r}} + \|e^{\tau^{0}(t) A} \overline{\phi}^{n}(t)\|_{H^{r}}  \leq \epsilon_n e^{K(t)} \rightarrow 0, \;\; \text{as}\;\; n\rightarrow \infty.
\end{equation*}
\end{proof}

\section{Limit resonant system}
In this section, we derive the formal resonant limit resonant system of the original system (\ref{local-system-1})--(\ref{local-system-2}) as $|\Omega|\rightarrow \infty$, and establish some properties of the limit resonant system. Recall from (\ref{up3}), we have 
\begin{eqnarray}\label{upl}
&&\hskip-.8in \partial_t u_+ = -e^{i\Omega t} \Big(\underbrace{u_+ \cdot \nabla u_+ - P_0( u_+ \cdot \nabla u_+ + (\nabla \cdot u_+) u_+) - (\int_0^z \nabla\cdot u_+(\boldsymbol{x}',s)ds ) \partial_z u_+ }_{:=I_1}\Big) \nonumber \\
&&\hskip-.28in -  \Big(\underbrace{\overline{v} \cdot \nabla u_+  + \frac{1}{2}(u_+ \cdot \nabla)(\overline{v}+i\overline{v}^\perp)}_{:=I_0} \Big)  \nonumber \\
&&\hskip-.28in - e^{-i\Omega t} \Big(\underbrace{u_- \cdot \nabla u_+ - P_0( u_- \cdot \nabla u_+ + (\nabla \cdot u_-) u_+) - (\int_0^z \nabla\cdot u_-(\boldsymbol{x}',s)ds ) \partial_z u_+}_{:=I_3} \Big) \nonumber\\
&&\hskip-.28in
- e^{-2i\Omega t} \underbrace{\frac{1}{2} (u_- \cdot \nabla)(\overline{v}+i\overline{v}^\perp)}_{:=I_4} = -e^{i\Omega t}I_1 - I_0 - e^{-i\Omega t}I_{-1} - e^{-2i\Omega t}I_{-2}.
\end{eqnarray}
Observe that $I_0$ is a typical resonant term. Unlike the case of the $3D$ Euler equations where there are frequency selection resonances, all frequencies resonate in $I_0$. We can rewrite (\ref{upl}) as
\begin{eqnarray*}
&&\hskip-.58in \partial_t \Big[u_+ - \frac{i}{\Omega} \Big(e^{i\Omega t}I_1 - e^{-i\Omega t} I_{-1} - \frac{1}{2} e^{-2i\Omega t} I_{-2} \Big) \Big] = - \frac{i}{\Omega} \Big(e^{i\Omega t}\partial_t I_1 - e^{-i\Omega t}\partial_t I_{-1} - \frac{1}{2}e^{-2i\Omega t}\partial_t I_{-2}\Big) - I_0.
\label{up4} 
\end{eqnarray*}
Denote by the formal limits of $u_+, u_-, \overline{v}$ to be $U_+, U_-, \overline{V}$. By taking limit $\Omega \rightarrow \infty $ formally, since $u_-$ is the complex conjugate of $u_+$, we obtain the limit resonant equations of $u_\pm$:
\begin{eqnarray}
&&\hskip-.8in \partial_t U_\pm = - (\overline{V} \cdot \nabla)U_\pm - \frac{1}{2} (U_\pm \cdot \nabla)(\overline{V} \pm i\overline{V}^\perp) .
\label{upmlimit} 
\end{eqnarray}
For the limit equation of $\overline{v}$, recall from (\ref{barotropic-evolution-4}) that
\begin{eqnarray*}
&&\hskip-.8in \partial_t \overline{v} + \mathbb{P}_h(\overline{v}\cdot \nabla \overline{v}) + e^{2i\Omega t}\mathbb{P}_h P_0 \Big(u_+ \cdot \nabla u_+  + (\nabla \cdot u_+) u_+ \Big)  \nonumber\\
&&\hskip.22in
 +e^{-2i\Omega t}\mathbb{P}_h P_0\Big(u_- \cdot \nabla u_-  + (\nabla \cdot u_-) u_-\Big) = 0.
\end{eqnarray*}
Observe that $\mathbb{P}_h(\overline{v}\cdot \nabla \overline{v})$ is a typical resonant term. Using the similar method in the derivation of $U_+$, we can derive the limit resonant equation for $\overline{v}$ as
\begin{eqnarray}
&&\hskip-.8in \partial_t \overline{V} + \mathbb{P}_h(\overline{V}\cdot \nabla \overline{V}) =0 .
\label{vbarlimit} 
\end{eqnarray}
Observe that \eqref{vbarlimit} is the $2D$ Euler system. Consider the initial conditions 
\begin{equation*}
    (\overline{V}_0, (U_+)_0, (U_-)_0) = (\overline{v}_0, \frac{1}{2}(\widetilde{v}_0 + i\widetilde{v}_0^\perp), \frac{1}{2}(\widetilde{v}_0 - i\widetilde{v}_0^\perp) )
\end{equation*}
for system (\ref{upmlimit})--(\ref{vbarlimit}). Since $v_0 \in \mathcal{S}$, we have $\nabla\cdot \overline{V}=0$, $P_0 \overline{V} = \overline{V}$, and $P_0 U_\pm =0$. 

Besides the equations for $U_\pm$ and $\overline{V}$, we also want a baroclinic mode $\widetilde{V}$ similar as in the original system. Since initially $U_\pm(0) = \frac{1}{2}(\widetilde{v}_0 \pm i\widetilde{v}_0^\perp)$, we 
define $\widetilde{V} := U_+ + U_-$ so that $U_\pm = \frac{1}{2}(\widetilde{V} \pm i \widetilde{V}^\perp)$. From (\ref{upmlimit}), we have
\begin{eqnarray}\label{vtildelimit}
\partial_t \widetilde{V} + (\overline{V}\cdot \nabla)\widetilde{V} + \frac{1}{2}(\widetilde{V}\cdot \nabla \overline{V} -  \widetilde{V}^\perp\cdot \nabla \overline{V}^\perp) =0.
\end{eqnarray}
Since $\nabla \cdot \overline{V} = 0$, (\ref{vtildelimit}) is equivalent to
\begin{eqnarray*}
&&\hskip-.8in \partial_t \widetilde{V} + \overline{V} \cdot \nabla \widetilde{V} + \frac{1}{2} \widetilde{V}^\perp (\nabla^\perp \cdot \overline{V}) = 0.
\end{eqnarray*}
Since $P_0 U_\pm=0$, we see $P_0 \widetilde{V} = 0$. 
Therefore, we consider the following limit resonant system
\begin{eqnarray}
&&\hskip-.8in \partial_t \overline{V} + \mathbb{P}_h(\overline{V}\cdot \nabla \overline{V} ) = 0,
\label{systemVbar-1} \\
&&\hskip-.8in \partial_t \widetilde{V} + \overline{V} \cdot \nabla \widetilde{V} + \frac{1}{2} \widetilde{V}^\perp (\nabla^\perp \cdot \overline{V}) = 0 , \label{systemVtilde} \\
&&\hskip-.8in 
\overline{V}(0) = \overline{V}_0 ,  \;\; \widetilde{V}(0) = \widetilde{V}_0, \label{systemlimit-initial}
\end{eqnarray}
with $P_0 \overline{V} = \overline{V}$ and $P_0 \widetilde{V} = 0$. Observe that (\ref{systemVbar-1}) is the $2D$ Euler system, and (\ref{systemVtilde}) is a linear transport equation with an additional stretching term. 

Next, we establish the global well-posedness of limit resonant system (\ref{systemVbar-1})--(\ref{systemlimit-initial}) in both Sobolev spaces and the space of analytic functions. Recall that the global well-posedness of (\ref{systemVbar-1}) has been established in Proposition \ref{global-Euler}.
\begin{proposition}\label{global-limit}
    Assume $\overline{V}_0 \in  \mathcal{S}\cap H^{r+1}(\mathbb{T}^3)$ and $\widetilde{V}_0 \in \mathcal{S}\cap H^{r}(\mathbb{T}^3)$ with $r>\frac{5}{2}$. Let $M\geq 0$, and suppose that $\|\overline{V}_0\|_{H^{r+1}} \leq M$. Then there exist constants $C >1$ and $C_r>1$, and a function $K(t) := C ^{\exp(C_r t)}$, such that for any give time $\mathcal{T}>0$,  there exists a unique solution $\overline{V} \in  L^\infty(0,\mathcal{T};  \mathcal{S}\cap H^{r+1}(\mathbb{T}^3))$ and $\widetilde{V} \in  L^\infty(0,\mathcal{T};\mathcal{S}\cap  H^{r}(\mathbb{T}^3))$ of system  (\ref{systemVbar-1})--(\ref{systemlimit-initial}) on $[0,\mathcal{T}]$, and satisfies 
    \begin{eqnarray}\label{sobolev-growth-limit}
    \|\overline{V}(t)\|_{H^{r+1}} \leq  K(t), \;\;\;\; \|\widetilde{V}(t)\|_{H^{r}} \leq  \|\widetilde{V}_0\|_{H^{r}} e^{K(t)}.
    \end{eqnarray}
    Moreover, assume $\overline{V}_0 \in \mathcal{D}(e^{\tau_0 A}: H^{r+1}(\mathbb{T}^3))$ and $\widetilde{V}_0 \in \mathcal{D}(e^{\tau_0 A}: H^{r}(\mathbb{T}^3))$ with $r>\frac{5}{2}$ and $\tau_0 >0$, and suppose that $\|e^{\tau_0 A} \overline{V}_0\|_{H^{r+1}}\leq M$. Then there exists a function
    \begin{equation}\label{tau-limit-system}
            \tau(t) = \tau_0 \exp(-\int_0^t K(s)ds),
    \end{equation}
    such that for any given time $\mathcal{T}>0$, there exists a unique solution $\overline{V}\in  L^\infty(0,\mathcal{T};\mathcal{S}\cap \mathcal{D}(e^{\tau(t) A}: H^{r+1}(\mathbb{T}^3)))$ and $\widetilde{V}\in  L^\infty(0,\mathcal{T};\mathcal{S}\cap \mathcal{D}(e^{\tau(t) A}: H^{r}(\mathbb{T}^3)))$ of system (\ref{systemVbar-1})--(\ref{systemlimit-initial}) on $[0,\mathcal{T}]$ such that
    \begin{equation*}\label{analytic-growth-limit}
            \|e^{\tau(t) A} \overline{V}(t)\|_{H^{r+1}} \leq K(t), \;\;\;\;
            \|e^{\tau(t) A}\widetilde{V}(t)\|_{H^r} \leq \|e^{\tau_0 A}\widetilde{V}_0\|_{H^r} e^{K(t)}.
    \end{equation*}
\end{proposition}
\begin{proof}
We will use the notation $K_1, K_2, ...$ as indicated in Remark \ref{remark-kt}. The global well-posedness of the $2D$ Euler equations in Sobolev spaces and corresponding growth estimate is classical, see \cite{BM02}. From $(\ref{Euler-Sobolev-estimate-2})$, we obtain that $\|\overline{V}\|_{H^{r+1}} \leq K_1(t)$ for some function $K_1(t)$. For the growth of $\|\widetilde{V}\|_{H^{r}}$, by standard energy estimate, since $\nabla\cdot \overline{V}=0$ and $r>\frac{5}{2}$, we have
\begin{equation*}
    \frac{d}{dt}\|\widetilde{V}\|^2_{H^{r}} \leq C_r\|\overline{V}\|_{H^{r+1}} \|\widetilde{V}\|^2_{H^{r}}.
\end{equation*}
By the Gr\"onwall inequality, and by virtue of the growth of $\|\overline{V}\|_{H^{r+1}}$, we obtain that
\begin{equation*}
    \|\widetilde{V}(t)\|_{H^{r}}\leq \|\widetilde{V}_0\|_{H^{r}}\exp(\frac{1}{2}C_r \int_0^t K_1(s)ds) \leq \|\widetilde{V}_0\|_{H^{r}} e^{K(t)}
\end{equation*}
for some suitable function $K(t)$, such that $\|\overline{V}(t)\|_{H^{r+1}} \leq K(t)$ also holds. By virtue of these formal energy estimates, the global well-posedness of system (\ref{systemVbar-1})--(\ref{systemlimit-initial}) in Sobolev spaces follows.

The global well-posedness of the $2D$ Euler equations in the space of analytic functions and the corresponding growth estimate are established in Proposition \ref{global-Euler}. From Proposition \ref{global-Euler}, we can first choose some suitable functions $K_1(t)$ and $K_2(t)$ such that $\tau(t) \leq \tau_0 \exp(-\int_0^t K_1(s)ds)$ and $\|e^{\tau(t) A} \overline{V}(t)\|_{H^{r+1}} \leq K_2(t)$. 
For the baroclinic mode $\widetilde{V}$, first, it is easy to see the $L^2$ energy is conserved. Next, using Lemma \ref{lemma-banach-algebra} and Lemma \ref{lemma-non-small}, since $r>\frac{5}{2}$ and $\int_{\mathbb{T}^3} \widetilde{V}(\boldsymbol{x}) d\boldsymbol{x}=0$, we have
\begin{equation*}
    \begin{split}
        &\frac{1}{2} \frac{d}{dt} \|A^r e^{\tau A} \widetilde{V} \|^2 = \dot{\tau} \|A^{r+\frac{1}{2}} e^{\tau A} \widetilde{V} \|^2 - \Big\langle A^r e^{\tau A} (\overline{V}\cdot \nabla \widetilde{V}), A^r e^{\tau A} \widetilde{V} \Big\rangle - \frac{1}{2} \Big\langle A^r e^{\tau A} (\nabla^\perp \cdot \overline{V}) \widetilde{V}^\perp, A^r e^{\tau A} \widetilde{V} \Big\rangle
        \\
        & \qquad \qquad\leq (\dot{\tau} + C_r\tau \|A^{r+1} e^{\tau A} \overline{V} \|) \|A^{r+\frac{1}{2}} e^{\tau A} \widetilde{V} \|^2 +  C_r\|e^{\tau A} \overline{V} \|_{H^{r+1}} \|A^{r} e^{\tau A} \widetilde{V} \|^2  .
    \end{split}
\end{equation*}
For suitable $K_1(t)$ and $K_2(t)$, we have
$
    \dot{\tau} + C_r\tau \|A^{r+1} e^{\tau A} \overline{V} \| \leq \tau (-K_1 + C_r K_2) \leq 0.
$
Therefore, by the Gr\"onwall inequality, for some suitable function $K(t)$, we have
\begin{eqnarray*}
&&\hskip-.68in
\|A^r e^{\tau(t) A} \widetilde{V}(t) \|^2 \leq \|A^r e^{\tau_0 A} \widetilde{V}_0 \|^2 \exp( \int_0^t C_r\| e^{\tau(s) A} \overline{V} (s)\|_{H^{r+1}}ds ) 
\leq \|e^{\tau_0 A}\widetilde{V}_0\|_{H^r}^2 e^{K(t)}.
\end{eqnarray*}
Since $L^2$ energy is conserved, we have
\begin{eqnarray*}
&&\hskip-.68in
\|e^{\tau(t) A}\widetilde{V}(t)\|_{H^r}
\leq \|e^{\tau_0 A}\widetilde{V}_0\|_{H^r} e^{K(t)}. 
\end{eqnarray*}
We can choose $K(t)$ large enough such that $\tau(t) = \tau_0 \exp(-\int_0^t K(s)ds)$ and $\|e^{\tau(t) A} \overline{V}\|_{H^{r+1}} \leq K(t)$. Notice that $\tau(\mathcal{T})>0$ for any finite time $\mathcal{T}<\infty$. Therefore, the solution $(\overline{V},\widetilde{V})$ exists in the space of analytic functions globally in time.
\end{proof}
\begin{remark}
The use of $K(t)$ above still follows Remark \ref{remark-kt}. The conclusion is that the growth of $\|\overline{V}(t)\|_{H^{r+1}}$ and  $\|e^{\tau(t) A}\overline{V}(t)\|_{H^{r+1}}$ are double exponential in time, while the growth of $\|\widetilde{V}(t)\|_{H^{r}}$ and  $\|e^{\tau(t) A}\widetilde{V}(t)\|_{H^{r}}$ are triple exponential in time.
\end{remark}
\begin{remark}\label{remark-steadyEuler}
Suppose $(\overline{V}, \widetilde{V})$ is a solution of system \eqref{systemVbar-1}--\eqref{systemlimit-initial}. In the special case when $\overline{V}$ is uniformly bounded in time, i.e., $\|\overline{V}(t)\|_{H^{r+1}} \leq C_{M,r}$ and $\|e^{\tau(t) A} \overline{V}(t)\|_{H^{r+1}} \leq C_{M,r}$ for $t\in[0,\infty)$, it is easy to see that the growths of $\|\widetilde{V}(t)\|_{H^{r}}$ and  $\|e^{\tau(t) A}\widetilde{V}(t)\|_{H^{r}}$ become only exponential in time. Moreover, $(0,\widetilde{V}_0)$ is always a steady state.
\end{remark}
\begin{remark}\label{remark-equivalence-upm-vtilde}
Since $U_\pm = \frac{1}{2}(\widetilde{V}+i\widetilde{V}^\perp)$, similar as Lemma \ref{lemma-vtilde-upm}, for $r\geq 0$ and $\tau \geq 0$, we have 
$
    \|U_+\|^2 = \|U_-\|^2 = \frac{1}{2}\|\widetilde{V}\|^2 $ and $
   \|e^{\tau A} U_+\|_{H^r}^2 = \|e^{\tau A} U_-\|_{H^r}^2 = \frac{1}{2}\|e^{\tau A} \widetilde{V}\|_{H^r} ^2 .
$
Therefore, the growing bounds of $\|\widetilde{V}\|_{H^{r}}$ and  $\|e^{\tau(t) A}\widetilde{V}(t)\|_{H^{r}}$ also apply to $\|U_\pm(t)\|_{H^{r}}$ and  $\|e^{\tau(t) A}U_\pm(t)\|_{H^{r}}$.
\end{remark}

\section{effect of rotation}
In section 4, we see that by requiring $\|e^{\tau_0 A}\widetilde{v}_0\|_{H^{r}} \leq \epsilon$, the life-span of the solution to system (\ref{local-system-1})--(\ref{local-system-2}) has a lower bound $\mathcal{T} \gtrsim \ln(\ln(\ln(\frac{1}{\epsilon})))$, as $\epsilon \rightarrow 0^+$, and this result is uniform in $\Omega\in\mathbb{R}$. In this section, we establish the effect of the rate of rotation $|\Omega|$ on the life-span $\mathcal{T}$. With the help of fast rotation, i.e., when $|\Omega|$ is large, we show that the time of existence of the solution in the space of analytic functions can be prolonged as long as the Sobolev norm $\|\widetilde{v}_0\|_{H^{r}}$ is small depending on $\Omega$, while the analytic norm $\|e^{\tau_0 A}\widetilde{v}_0\|_{H^{r}}$ can be large (of order $1$). We call such initial data as ``well-prepared" initial data. The following theorem is the main result of this paper.

\begin{theorem}\label{theorem-main}
Assume $\overline{v}_0 \in \mathcal{S}\cap \mathcal{D}(e^{\tau_0 A}: H^{r+3}(\mathbb{T}^3))$, $\widetilde{v}_0 \in \mathcal{S}\cap\mathcal{D}(e^{\tau_0 A}: H^{r+2}(\mathbb{T}^3))$ with $r>\frac{5}{2}$ and $\tau_0>0$. Let $M\geq 0$ and $\delta>0$, then there exist constants $C_{\tau_0} >1$, $C_{M,\tau_0}>1$, $C_r>1$, $\widetilde{C}_{M,\tau_0}>1$, $\widetilde{C}_r>1$, and functions $\widetilde{K}(t):=e^{C_{M,\tau_0}^{\exp(C_r t)}}$, $\widetilde{K}_0(t):=e^{\widetilde{C}_{M,\tau_0}^{\exp(\widetilde{C}_r t)}}$, with $\widetilde{K}(t)>\widetilde{K}_0(t)$. Suppose that
$|\Omega_0| \geq C_{\tau_0} e^{\widetilde{K}(1)}$, and that $\|e^{\tau_0 A}\overline{v}_0\|_{H^{r+3}} + \|e^{\tau_0 A}\widetilde{v}_0\|_{H^{r+2}} \leq M$ with $\|\widetilde{v}_0\|_{H^{3+\delta}} \leq  \frac{1}{|\Omega_0|}$. Then there exists a time $\mathcal{T} = \mathcal{T}(\tau_0, |\Omega_0|, M, r)\geq 1$ satisfying 
\begin{eqnarray}\label{T-longtime-type2}
 C_{\tau_0}e^{ \widetilde{K}(\mathcal{T})}   = |\Omega_0|,
\end{eqnarray}
such that when $|\Omega|\geq |\Omega_0|$, the unique solution $(\overline{v}, \widetilde{v})$ to system (\ref{local-system-1})--(\ref{local-system-2}) obtained in Theorem \ref{theorem-local} satisfies  $(\overline{v}, \widetilde{v}) \in  L^\infty(0,\mathcal{T};\mathcal{S} \cap \mathcal{D}(e^{\tau(t) A}: H^r(\mathbb{T}^3)))$, with 
\begin{equation}\label{main-tau}
    \tau(t) = \Big(\tau_0 - \int_0^t\frac{ e^{\widetilde{K}_0(s)} }{\sqrt{|\Omega_0| -  e^{\widetilde{K}_0(s)}  } } ds- \int_0^t \frac{e^{\widetilde{K}_0(s)}}{|\Omega_0|} ds\Big) e^{-\int_0^t \widetilde{K}_0(s)ds} >0.
\end{equation}
In particular, from (\ref{T-longtime-type2}), $\mathcal{T}\gtrsim \ln(\ln(\ln(\ln|\Omega_0|)))\rightarrow\infty$, as $|\Omega_0| \rightarrow \infty$.
\end{theorem}
Thanks to Lemma \ref{1st-equivalence} and Lemma \ref{lemma-decomposition}, we immediately have the following corollary.
\begin{corollary}
Suppose $v_0 \in \mathcal{S}\cap \mathcal{D}(e^{\tau_0 A}: H^{r+3}(\mathbb{T}^3))$, and the conditions of Theorem \ref{theorem-main} hold. Then the unique solution $v$ obtained in Corollary \ref{corollary-local} satisfies  $v \in  L^\infty(0,\mathcal{T};\mathcal{S} \cap \mathcal{D}(e^{\tau(t) A}: H^r(\mathbb{T}^3)))$, when $|\Omega|\geq |\Omega_0|$, with $\mathcal{T}$ defined in (\ref{T-longtime-type2}) and $\tau$ defined in (\ref{main-tau}).  
\end{corollary}

In this section, we focus on system (\ref{up3})--(\ref{barotropic-evolution-4}), which is equivalent to system (\ref{local-system-1})--(\ref{local-system-2}) due to Lemma \ref{1st-equivalence} and Lemma \ref{2nd-equivalence}. To prove Theorem \ref{theorem-main}, in section 6.2, we consider the difference between the original system (\ref{up3})--(\ref{barotropic-evolution-4}) and the limit resonant system (\ref{upmlimit})--(\ref{vbarlimit}). We call such difference system as perturbed system. In section \ref{section-proof-main}, by the formal energy estimate, we show that the solution to the perturbed system exists for a long time. This together with the global existence of the solution to system (\ref{upmlimit})--(\ref{vbarlimit}) give us the long time existence of the solution to system (\ref{up3})--(\ref{barotropic-evolution-4}), and therefore the long time existence of the solution to system (\ref{local-system-1})--(\ref{local-system-2}). 

In section 6.1, we first give a rational behind the smallness of the initial baroclinic mode.

\subsection{A rational behind the smallness of the initial baroclinic mode}
The result of Theorem \ref{theorem-main} is for ``well-prepared" initial data, namely, for a given fixed $\delta>0$, $\|\widetilde{v}_0\|_{H^{3+\delta}} \leq \frac{1}{|\Omega_0|}$. Before we go into the proof of Theorem \ref{theorem-main}, we briefly rationalize, below, the reason behind this smallness condition on the baroclinic mode. 

Consider the linear inviscid PEs:
\begin{eqnarray*}\label{inviscidPE-linear}
&\partial_t \mathcal{V} +\Omega \mathcal{V}^\perp + \nabla p = 0, \\
&\partial_z p  =0,\\
&\nabla \cdot \mathcal{V} + \partial_z w =0,
\end{eqnarray*}
whose explicit solution is
\begin{equation*}
    \mathcal{V}(\boldsymbol{x},t)=\overline{\mathcal V}_0(\boldsymbol{x}')+\mathcal R(t)\widetilde{\mathcal V}_0(\boldsymbol{x}),
\end{equation*}
where 
\begin{equation*}
    \mathcal R(t):=\left(
		\begin{array}{cccc}
		 \cos(\Omega t) & \sin(\Omega t)  \\
		-\sin(\Omega t) &  \cos(\Omega t)  
		\end{array}
		\right).
\end{equation*}
We see there is no ``decay" due to rotation in the linear level. This is different from the linearized $3D$ Euler equations with rotation, for which one can obtain certain decay due to dispersion/averaging mechanism, see, e.g., \cite{D05,KLT14}. 

Now let us look back to our nonlinear inviscid PEs (\ref{local-system-1})--(\ref{local-system-2}). The first equation (\ref{local-system-1}) is the evolution of the barotropic mode, which is the $2D$ Euler with source terms coming from the baroclinic mode. The second equation (\ref{local-system-2}) is the evolution of the baroclinic mode, which is the Burger's equations with rotation and other nonlinear coupling terms. For the Burger's equations with rotation, it is shown in \cite{BIT11,LT04} that when the rotation rate $|\Omega|$ is large enough depending on the initial data, the solution exists globally in time because of the absence of resonance between the rotation and nonlinearity, which allows a very strong averaging mechanism that weakens the nonlinearity. In our case, however, the additional coupling nonlinear terms in (\ref{local-system-2}) resonate with the rotation term, which does not allow for this simple scenario to take place. However, thanks to the smallness assumption on the initial baroclinic mode, the additional coupling nonlinear terms are initially small, which allows us to push this argument further.

Another reason behind this smallness assumption is indicated in \cite{ILT20}, where a finite-time blowup of solutions to the inviscid PEs with rotation is established. Indeed, for the initial data
\begin{equation*}\label{initial-data-different-scale}
    v_0(\boldsymbol{x}) = v_0 (x,z) = \Big( \lambda(-z^2+ \frac{1}{3}) \sin x,  -\Omega \sin x\Big)
\end{equation*}
with $\lambda > 0$, it is shown that $\frac{9}{2\lambda}$ is an upper bound for the blowup time. Notably here $\overline{v}_0 = (0, -\Omega \sin x)$ and $\widetilde{v}_0 = ( \lambda(-z^2+ \frac{1}{3}) \sin x,0)$. Therefore, when $|\Omega|\gg 1$, we have:
\begin{itemize}
    \item when $\lambda = |\Omega|$, the baroclinic mode satisfies $\widetilde{v}_0 \sim |\Omega|$, and the whole initial data satisfies $v_0 \sim |\Omega|$. An upper bound of blowup time in this case satisfies $\mathcal{T}\sim \frac{1}{|\Omega|}$;
    \item when $\lambda = 1$, the baroclinic mode satisfies $\widetilde{v}_0 \sim 1$, while the whole initial data satisfies $v_0 \sim |\Omega|$. An upper bound of blowup time in this case satisfies $\mathcal{T}\sim 1$;
    \item when $\lambda = \frac{1}{|\Omega|}$, this implies a smallness condition on the baroclinic $\widetilde{v}_0 \sim \frac{1}{|\Omega|}$, while the whole initial data satisfies $v_0 \sim |\Omega|$. An upper bound of blowup time in this case satisfies $\mathcal{T}\sim |\Omega|$.
\end{itemize}
The above, in particular, the last item suggest that the smallness condition on the baroclinic mode is required to guarantee the long time existence of solutions to the $3D$ inviscid PEs with fast rotation.

Further reasoning for the smallness condition on the initial baroclinic mode will be provided in Remark \ref{remark-term-type1} and Remark \ref{remark-term-type2}, below.

\subsection{The perturbed system around $|\Omega|=\infty$}\label{section-difference-system}
Since the limit resonant system (\ref{upmlimit})--(\ref{vbarlimit}) is globally well-posed, the idea to show long time existence of the solution is to consider the difference between the original system (\ref{up3})--(\ref{barotropic-evolution-4}) and the limit resonant system (\ref{upmlimit})--(\ref{vbarlimit}).
Denote by 
$ \overline{\phi} = \overline{v} - \overline{V}$, and $\phi_\pm = u_\pm - U_\pm.$
Taking the difference between system (\ref{up3})--(\ref{barotropic-evolution-4}) and system (\ref{upmlimit})--(\ref{vbarlimit}), we obtain
\begin{eqnarray}
&&\hskip-.8in \partial_t \overline{\phi} + \mathbb{P}_h\Big[\overline{\phi}\cdot \nabla \overline{V} + \overline{\phi}\cdot \nabla \overline{\phi} + \overline{V}\cdot \nabla \overline{\phi}  +   e^{2i\Omega t} P_0\Big( Q_{1,+,+} + Q_{2,+,+} \Big)\nonumber\\
&&\hskip-.1in
 + e^{-2i\Omega t} P_0\Big( Q_{1,-,-} + Q_{2,-,-} \Big) \Big] = 0, \quad \text{and}\label{difference-phibar}  \\
\nonumber \\
&&\hskip-.8in 
\partial_t \phi_\pm + \overline{\phi} \cdot \nabla U_\pm + \overline{\phi} \cdot \nabla \phi_\pm + \overline{V} \cdot \nabla \phi_\pm  + \frac{1}{2}(\phi_\pm \cdot \nabla)(\overline{V} \pm i\overline{V}^\perp) + \frac{1}{2}(\phi_\pm \cdot \nabla)(\overline{\phi} \pm i\overline{\phi}^\perp) \nonumber\\
&&\hskip-.5in 
+ \frac{1}{2}(U_\pm \cdot \nabla)(\overline{\phi} \pm i\overline{\phi}^\perp) + e^{\pm i\Omega t}\Big(Q_{1,\pm, \pm} - P_0 Q_{1,\pm,\pm} - P_0 Q_{2,\pm,\pm} - Q_{3,\pm,\pm}\Big) \nonumber \\
&&\hskip-.5in
+ e^{\mp i\Omega t} \Big(Q_{1,\mp, \pm} - P_0 Q_{1,\mp,\pm} - P_0 Q_{2,\mp,\pm} - Q_{3,\mp,\pm}\Big) + e^{\mp 2i\Omega t}Q_{4,\mp,\pm}  = 0 , \label{difference-phipm} 
\end{eqnarray}
where
\begin{eqnarray*}
&&\hskip-.9in
Q_{1,\pm,\mp} = \phi_\pm \cdot \nabla U_\mp + \phi_\pm \cdot \nabla \phi_\mp + U_\pm \cdot \nabla \phi_\mp + U_\pm \cdot \nabla U_\mp, \\
&&\hskip-.9in
Q_{2,\pm,\mp} = (\nabla\cdot \phi_\pm) U_\mp + (\nabla\cdot \phi_\pm) \phi_\mp + (\nabla\cdot U_\pm) \phi_\mp + (\nabla\cdot U_\pm) U_\mp,
\\
&&\hskip-.9in
Q_{3,\pm,\mp} = (\int_0^z \nabla\cdot \phi_\pm (\boldsymbol{x}',s)ds) \partial_z U_\mp +(\int_0^z \nabla\cdot \phi_\pm (\boldsymbol{x}',s)ds) \partial_z \phi_\mp \nonumber
\\
&&\hskip-.3in
+(\int_0^z \nabla\cdot U_\pm (\boldsymbol{x}',s)ds) \partial_z \phi_\mp+(\int_0^z \nabla\cdot U_\pm (\boldsymbol{x}',s)ds) \partial_z U_\mp,\\
&&\hskip-.9in
Q_{4,\pm,\mp} = \frac{1}{2}\Big[(\phi_\pm \cdot \nabla)(\overline{V} \mp i\overline{V}^\perp) + (\phi_\pm \cdot \nabla)(\overline{\phi} \mp i\overline{\phi}^\perp)
\nonumber\\
&&\hskip-.3in
+ (U_\pm \cdot \nabla)(\overline{\phi} \mp i\overline{\phi}^\perp) + (U_\pm \cdot \nabla)(\overline{V} \mp i\overline{V}^\perp) \Big].
\end{eqnarray*}
We supplement the initial conditions for the limit resonant system (\ref{upmlimit})--(\ref{vbarlimit}) as 
\begin{equation}\label{limit-ic}
    \overline{V}_0 = \overline{v}_0, \;\;\; (U_\pm)_0 = (u_\pm)_0 = \frac{1}{2}(\widetilde{v}_0\pm i\widetilde{v}_0^\perp).
\end{equation}
Therefore, the initial conditions for the perturbed system is 
\begin{eqnarray}\label{difference-ic}
  \overline{\phi}_0=0, \;\; (\phi_\pm)_0 = 0.
\end{eqnarray}


\subsection{Proof of Theorem \ref{theorem-main}}\label{section-proof-main}
In this subsection, we prove Theorem \ref{theorem-main}. From Proposition \ref{global-limit}, let $\overline{V}$ and $U_\pm$ be the global solution in $\mathcal{S}\cap \mathcal{D}(e^{\tau(t) A}: H^{r+3}(\mathbb{T}^3))$ and $\mathcal{S}\cap \mathcal{D}(e^{\tau(t) A}: H^{r+2}(\mathbb{T}^3))$, respectively, to system (\ref{upmlimit})-(\ref{vbarlimit}), with initial data (\ref{limit-ic}) and $\tau(t)$ defined by (\ref{tau-limit-system}). Applying $A^r e^{\tau A}$ to (\ref{difference-phibar})--(\ref{difference-phipm}), and taking the $L^2$ inner product of (\ref{difference-phibar}) with $A^r e^{\tau A} \overline{\phi}$, (\ref{difference-phipm}) with $2A^r e^{\tau A} \phi_\mp$, thanks to Lemma \ref{lemma-projection} and Lemma \ref{lemma-leray}, we obtain
\begin{eqnarray}
&&\hskip-.58in \frac{1}{2} \frac{d}{dt} \|A^r e^{\tau A} \overline{\phi} \|^2 = \dot{\tau} \|A^{r+\frac{1}{2}} e^{\tau A} \overline{\phi} \|^2 - \Big\langle A^r e^{\tau A} (\overline{\phi}\cdot \nabla \overline{V}), A^r e^{\tau A} \overline{\phi} \Big\rangle - \Big\langle A^r e^{\tau A} (\overline{\phi}\cdot \nabla \overline{\phi}), A^r e^{\tau A} \overline{\phi} \Big\rangle \nonumber \\
&&\hskip.38in
- \Big\langle A^r e^{\tau A} (\overline{V}\cdot \nabla \overline{\phi}), A^r e^{\tau A} \overline{\phi} \Big\rangle  - e^{2i\Omega t}\Big\langle A^r e^{\tau A}(Q_{1,+,+} + Q_{2,+,+}), A^r e^{\tau A} \overline{\phi}\Big\rangle\nonumber \\
&&\hskip.38in
- e^{-2i\Omega t}\Big\langle A^r e^{\tau A}(Q_{1,-,-} + Q_{2,-,-}), A^r e^{\tau A} \overline{\phi}\Big\rangle, \label{energy-difference-phibar}
\end{eqnarray}
and
\begin{eqnarray}
&&\hskip-.48in \frac{d}{dt} (\|A^r e^{\tau A} \phi_+ \|^2 + \|A^r e^{\tau A} \phi_- \|^2) = 2\dot{\tau} (\|A^{r+\frac{1}{2}} e^{\tau A} \phi_+ \|^2 + \|A^{r+\frac{1}{2}} e^{\tau A} \phi_- \|^2)  \nonumber \\
&&\hskip-.18in
- 2\Big\langle A^r e^{\tau A} (\overline{\phi}\cdot \nabla U_+), A^r e^{\tau A} \phi_- \Big\rangle - 2\Big\langle A^r e^{\tau A} (\overline{\phi}\cdot \nabla U_-), A^r e^{\tau A} \phi_+ \Big\rangle \nonumber\\
&&\hskip-.18in
- 2\Big\langle A^r e^{\tau A} (\overline{\phi}\cdot \nabla \phi_+), A^r e^{\tau A} \phi_- \Big\rangle - 2\Big\langle A^r e^{\tau A} (\overline{\phi}\cdot \nabla \phi_-), A^r e^{\tau A} \phi_+ \Big\rangle \nonumber\\
&&\hskip-.18in
- 2\Big\langle A^r e^{\tau A} (\overline{V}\cdot \nabla \phi_+), A^r e^{\tau A} \phi_- \Big\rangle - 2\Big\langle A^r e^{\tau A} (\overline{V}\cdot \nabla \phi_-), A^r e^{\tau A} \phi_+ \Big\rangle \nonumber\\
&&\hskip-.18in
- \Big\langle A^r e^{\tau A} (\phi_+\cdot \nabla (\overline{V}+i\overline{V}^\perp)), A^r e^{\tau A} \phi_- \Big\rangle - \Big\langle A^r e^{\tau A} (\phi_- \cdot \nabla (\overline{V}-i\overline{V}^\perp)), A^r e^{\tau A} \phi_+ \Big\rangle \nonumber\\
&&\hskip-.18in
- \Big\langle A^r e^{\tau A} (\phi_+\cdot \nabla (\overline{\phi}+i\overline{\phi}^\perp)), A^r e^{\tau A} \phi_- \Big\rangle - \Big\langle A^r e^{\tau A} (\phi_- \cdot \nabla (\overline{\phi}-i\overline{\phi}^\perp)), A^r e^{\tau A} \phi_+ \Big\rangle \nonumber\\
&&\hskip-.18in
- \Big\langle A^r e^{\tau A} (U_+\cdot \nabla (\overline{\phi}+i\overline{\phi}^\perp)), A^r e^{\tau A} \phi_- \Big\rangle - \Big\langle A^r e^{\tau A} (U_- \cdot \nabla (\overline{\phi}-i\overline{\phi}^\perp)), A^r e^{\tau A} \phi_+ \Big\rangle \nonumber\\
&&\hskip-.18in
- 2e^{i\Omega t}\Big(\Big\langle A^r e^{\tau A}(Q_{1,+,+} - Q_{3,+,+}), A^r e^{\tau A} \phi_-\Big\rangle + \Big\langle A^r e^{\tau A}(Q_{1,+,-} - Q_{3,+,-}), A^r e^{\tau A} \phi_+\Big\rangle\Big) \nonumber\\
&&\hskip-.18in
- 2e^{-i\Omega t}\Big(\Big\langle A^r e^{\tau A}(Q_{1,-,+} - Q_{3,-,+}), A^r e^{\tau A} \phi_-\Big\rangle + \Big\langle A^r e^{\tau A}(Q_{1,-,-} - Q_{3,-,-}), A^r e^{\tau A} \phi_+\Big\rangle\Big) \nonumber\\
&&\hskip-.18in
-2e^{2i\Omega t}\Big\langle A^r e^{\tau A}Q_{4,+,-}, A^r e^{\tau A} \phi_+\Big\rangle - 2e^{-2i\Omega t}\Big\langle A^r e^{\tau A}Q_{4,-,+}, A^r e^{\tau A} \phi_-\Big\rangle .
\label{energy-difference-phipm}
\end{eqnarray}

There are totally $71$ different nonlinear terms in (\ref{energy-difference-phibar}) and (\ref{energy-difference-phipm}). We separate them into the following four different types. We use $V$ to denote the velocity field of the limit resonant system, i.e., $\overline{V}$ and $U_\pm$, and use $\phi$ to denote the velocity filed of the perturbed system, i.e., $\overline{\phi}$ and $\phi_\pm$.
\begin{itemize}
    \item Type 1: terms that are trilinear in $\phi$, e.g., $\Big\langle A^r e^{\tau A} (\overline{\phi}\cdot \nabla \overline{\phi}), A^r e^{\tau A} \overline{\phi} \Big\rangle$.
    \item Type 2: terms that are bilinear in $\phi$ with no derivative of $\phi$, e.g.,  $\Big\langle A^r e^{\tau A} (\overline{\phi}\cdot \nabla \overline{V}), A^r e^{\tau A} \overline{\phi} \Big\rangle$.
    \item Type 3: terms that are linear in $\phi$, e.g., $e^{2i\Omega t}\Big\langle A^r e^{\tau A}(U_+\cdot\nabla U_+), A^r e^{\tau A} \overline{\phi}\Big\rangle$.
    \item Type 4: terms that are bilinear in $\phi$ and a derivative of $\phi$, e.g., $\Big\langle A^r e^{\tau A} (\overline{V}\cdot \nabla \overline{\phi}), A^r e^{\tau A} \overline{\phi} \Big\rangle$.
\end{itemize}
\subsubsection{Estimates of Type 1 and Type 2 terms}
For type 1 nonlinear terms (19 terms), using Lemma \ref{lemma-type1}--\ref{lemma-type3}, and for type 2 nonlinear terms (15 terms), using Lemma \ref{lemma-banach-algebra}, since $\overline{\phi}$, $\phi_\pm$, $\overline{V}$ and $U_\pm$ all have zero mean value in $\mathbb{T}^3$, we have
\begin{align}\label{estimate-general-terms}
        &\Big|\Big\langle A^r e^{\tau A} (\overline{\phi}\cdot \nabla \overline{V}), A^r e^{\tau A} \overline{\phi} \Big\rangle \Big|  +\Big|\Big\langle A^r e^{\tau A} (\overline{\phi}\cdot \nabla \overline{\phi}), A^r e^{\tau A} \overline{\phi} \Big\rangle\Big| \nonumber\\
        + &\Big|e^{2i\Omega t} \Big\langle A^r e^{\tau A} \Big(\phi_+\cdot \nabla U_+ + \phi_+\cdot \nabla \phi_+ + (\nabla\cdot U_+)\phi_+ + (\nabla\cdot \phi_+)\phi_+\Big), A^r e^{\tau A} \overline{\phi} \Big\rangle \Big|\nonumber\\
        + &\Big|e^{-2i\Omega t} \Big\langle A^r e^{\tau A} \Big(\phi_-\cdot \nabla U_- + \phi_-\cdot \nabla \phi_- + (\nabla\cdot U_-)\phi_- + (\nabla\cdot \phi_-)\phi_-\Big), A^r e^{\tau A} \overline{\phi} \Big\rangle \Big|\nonumber\\
        + &2\Big|\Big\langle A^r e^{\tau A} (\overline{\phi}\cdot \nabla U_+), A^r e^{\tau A} \phi_- \Big\rangle \Big| + 2 \Big|\Big\langle A^r e^{\tau A} (\overline{\phi}\cdot \nabla U_-), A^r e^{\tau A} \phi_+ \Big\rangle \Big|\nonumber\\
        + &2\Big|\Big\langle A^r e^{\tau A} (\overline{\phi}\cdot \nabla \phi_+), A^r e^{\tau A} \phi_- \Big\rangle \Big| +  2\Big|\Big\langle A^r e^{\tau A} (\overline{\phi}\cdot \nabla \phi_-), A^r e^{\tau A} \phi_+ \Big\rangle \Big|\nonumber\\
        + &\Big| \Big\langle A^r e^{\tau A} \Big(\phi_+\cdot \nabla (\overline{V}+i\overline{V}^\perp)\Big), A^r e^{\tau A} \phi_- \Big\rangle\Big| + \Big| \Big\langle A^r e^{\tau A} \Big(\phi_-\cdot \nabla (\overline{V}-i\overline{V}^\perp)\Big), A^r e^{\tau A} \phi_+ \Big\rangle\Big|\nonumber\\
        + &\Big| \Big\langle A^r e^{\tau A} \Big(\phi_+\cdot \nabla (\overline{\phi}+i\overline{\phi}^\perp)\Big), A^r e^{\tau A} \phi_- \Big\rangle\Big| + \Big| \Big\langle A^r e^{\tau A} \Big(\phi_-\cdot \nabla (\overline{\phi}-i\overline{\phi}^\perp)\Big), A^r e^{\tau A} \phi_+ \Big\rangle\Big|\nonumber\\
        + &2\Big|e^{i\Omega t}  \Big\langle A^r e^{\tau A} \Big(\phi_+\cdot \nabla U_+ + \phi_+\cdot \nabla \phi_+ -(\int_0^z \nabla\cdot \phi_+(\boldsymbol{x}',s)ds) \partial_z \phi_+ \Big), A^r e^{\tau A} \phi_- \Big\rangle\Big|\nonumber\\
        + &2\Big|e^{i\Omega t}  \Big\langle A^r e^{\tau A} \Big(\phi_+\cdot \nabla U_- + \phi_+\cdot \nabla \phi_- -(\int_0^z \nabla\cdot \phi_+(\boldsymbol{x}',s)ds) \partial_z \phi_- \Big), A^r e^{\tau A} \phi_+ \Big\rangle\Big|\nonumber\\
        + &2\Big|e^{-i\Omega t}  \Big\langle A^r e^{\tau A} \Big(\phi_-\cdot \nabla U_+ + \phi_-\cdot \nabla \phi_+ -(\int_0^z \nabla\cdot \phi_-(\boldsymbol{x}',s)ds) \partial_z \phi_+ \Big), A^r e^{\tau A} \phi_- \Big\rangle\Big|\nonumber\\
        + &2\Big|e^{-i\Omega t}  \Big\langle A^r e^{\tau A} \Big(\phi_-\cdot \nabla U_- + \phi_-\cdot \nabla \phi_- -(\int_0^z \nabla\cdot \phi_-(\boldsymbol{x}',s)ds) \partial_z \phi_- \Big), A^r e^{\tau A} \phi_+ \Big\rangle\Big|\nonumber\\
        + &\Big|e^{2i\Omega t} \Big\langle A^r e^{\tau A} \Big(\phi_+\cdot \nabla (\overline{V}-i\overline{V}^\perp) + \phi_+\cdot \nabla (\overline{\phi}-i\overline{\phi}^\perp)\Big), A^r e^{\tau A} \phi_+ \Big\rangle \Big| \nonumber\\
        + &\Big|e^{-2i\Omega t} \Big\langle A^r e^{\tau A} \Big(\phi_-\cdot \nabla (\overline{V}+i\overline{V}^\perp) + \phi_-\cdot \nabla (\overline{\phi}+i\overline{\phi}^\perp)\Big), A^r e^{\tau A} \phi_- \Big\rangle \Big|\nonumber\\
        \leq C_r&\Big(\|A^{r+1} e^{\tau A} \overline{V}\| + \|A^{r+1} e^{\tau A} U_+\| + \|A^{r+1} e^{\tau A} U_-\|\Big)\Big(\frac{1}{2}\|A^{r} e^{\tau A} \overline{\phi} \|^2 + \|A^{r} e^{\tau A} \phi_+ \|^2 + \|A^{r} e^{\tau A} \phi_- \|^2\Big)\nonumber\\
        +C_r&\Big(\|A^{r} e^{\tau A} \overline{\phi}\| + \|A^{r} e^{\tau A} \phi_+\| + \|A^{r} e^{\tau A} \phi_-\|\Big)\Big(\|A^{r+\frac{1}{2}} e^{\tau A} \overline{\phi} \|^2 + \|A^{r+\frac{1}{2}} e^{\tau A} \phi_+ \|^2 + \|A^{r+\frac{1}{2}} e^{\tau A} \phi_- \|^2\Big).
    \end{align}
\subsubsection{Estimates of Type 3 terms}    
For type 3 nonlinear terms (14 terms), when $\Omega\neq 0$, we first explain the idea on the sample term $e^{2i\Omega t} \Big\langle A^{r} e^{\tau A} (U_+ \cdot \nabla U_+), A^{r} e^{\tau A} \overline{\phi}\Big\rangle$. Indeed, by differentiation by parts, we have
\begin{eqnarray*}
&&\hskip-.58in
 e^{2i\Omega t} \Big\langle A^{r} e^{\tau A} (U_+ \cdot \nabla U_+), A^{r} e^{\tau A} \overline{\phi}\Big\rangle \nonumber \\ 
&&\hskip-.78in
= \frac{1}{2i\Omega} \partial_t\Big(e^{2i\Omega t} \Big\langle A^{r} e^{\tau A} (U_+ \cdot \nabla U_+), A^{r} e^{\tau A} \overline{\phi}\Big\rangle \Big) - \frac{1}{2i\Omega} e^{2i\Omega t}\partial_t \Big(\Big\langle A^{r} e^{\tau A} (U_+ \cdot \nabla U_+), A^{r} e^{\tau A} \overline{\phi}\Big\rangle \Big).
\end{eqnarray*}
We leave the first term until integrating in time. For the second term, we have
\begin{eqnarray}
&&\hskip-.58in
- \frac{1}{2i\Omega} e^{2i\Omega t} \partial_t \Big(\Big\langle A^{r} e^{\tau A} (U_+ \cdot \nabla U_+), A^{r} e^{\tau A} \overline{\phi}\Big\rangle \Big) \nonumber \\
&&\hskip-.68in
\leq \frac{1}{|\Omega|} |\dot{\tau}| \Big|\Big\langle A^{r+1} e^{\tau A} (U_+ \cdot \nabla U_+), A^{r} e^{\tau A} \overline{\phi}\Big\rangle\Big| 
+\frac{1}{2|\Omega|}\Big| \Big\langle A^{r} e^{\tau A} \partial_t (U_+ \cdot \nabla U_+), A^{r} e^{\tau A} \overline{\phi}\Big\rangle\Big|\nonumber\\
&&\hskip-.58in
+\frac{1}{2|\Omega|}\Big| \Big\langle A^{r} e^{\tau A} (U_+ \cdot \nabla U_+), A^{r} e^{\tau A} \partial_t \overline{\phi}\Big\rangle\Big| 
: = I_1 + I_2 + I_3. \label{large-term-nonlinear}
\end{eqnarray}
Thanks to the Cauchy–Schwarz inequality, Lemma \ref{lemma-banach-algebra}, and Lemma \ref{lemma-leray}, since $\overline{\phi}$, $\phi_\pm$, $\overline{V}$ and $U_\pm$ all have zero mean value in $\mathbb{T}^3$, and since $r>\frac{5}{2}$, from (\ref{upmlimit}) and (\ref{difference-phibar}), we have
\begin{eqnarray*}
&&\hskip-.58in
I_1 \leq  \frac{C_r}{|\Omega|} |\dot{\tau}| \|A^{r+1} e^{\tau A} U_+\| \|A^{r+2} e^{\tau A} U_+\| \|A^{r} e^{\tau A} \overline{\phi}\| \nonumber\\
&&\hskip-.38in
\leq  \frac{C_r}{|\Omega|^2} |\dot{\tau}|^2 + C_r \|A^{r+2} e^{\tau A} U_+\|^4 \|A^{r} e^{\tau A} \overline{\phi}\|^2, 
\end{eqnarray*}
\begin{eqnarray*}
&&\hskip-.58in
I_2 \leq \frac{C}{|\Omega|}\Big( \Big| \Big\langle A^{r} e^{\tau A} \Big\{\Big(\overline{V}\cdot \nabla U_+ + \frac{1}{2}(U_+ \cdot \nabla)(\overline{V}+ i\overline{V}^\perp)\Big)\cdot \nabla U_+ \Big\}, A^{r} e^{\tau A} \overline{\phi}\Big\rangle   \Big| \nonumber \\
&&\hskip.15in
+ \Big| \Big\langle A^{r} e^{\tau A} \Big\{U_+ \cdot \nabla \Big(\overline{V}\cdot \nabla U_+ + \frac{1}{2}(U_+ \cdot \nabla)(\overline{V}+ i\overline{V}^\perp)\Big)\Big\}, A^{r} e^{\tau A} \overline{\phi}\Big\rangle   \Big| \Big) \nonumber \\
&&\hskip-.45in 
\leq \frac{C_r}{|\Omega|} \|A^{r+2} e^{\tau A} U_+\|^2 \|A^{r+2} e^{\tau A} \overline{V}\| \|A^{r} e^{\tau A} \overline{\phi}\| \nonumber \\
&&\hskip-.45in 
\leq C_r \|A^{r+2} e^{\tau A} U_+\|^2 \|A^{r+2} e^{\tau A} \overline{V}\|^2 \|A^{r} e^{\tau A} \overline{\phi}\|^2 + \frac{C_r}{|\Omega|^2} \|A^{r+2} e^{\tau A} U_+\|^2, \quad \text{and}
\end{eqnarray*}
\begin{eqnarray*}
&&\hskip-.58in
I_3 \leq \frac{C}{|\Omega|} \Big| \Big\langle A^{r} e^{\tau A} \mathbb{P}_h (U_+ \cdot \nabla U_+), A^{r} e^{\tau A} \Big\{\overline{\phi}\cdot \nabla \overline{V} + \overline{\phi}\cdot \nabla \overline{\phi} + \overline{V}\cdot \nabla \overline{\phi}  \nonumber \\
&&\hskip.15in
+   e^{2i\Omega t} P_0\Big( Q_{1,+,+} + Q_{2,+,+} \Big) + e^{-2i\Omega t} P_0\Big( Q_{1,-,-} + Q_{2,-,-} \Big)  \Big\} \Big\rangle \Big|\nonumber \\
&&\hskip-.35in 
\leq \frac{C}{|\Omega|} \Big| \Big\langle A^{r+1} e^{\tau A} \mathbb{P}_h (U_+ \cdot \nabla U_+), A^{r-1} e^{\tau A} \Big\{\overline{\phi}\cdot \nabla \overline{V} + \overline{\phi}\cdot \nabla \overline{\phi} + \overline{V}\cdot \nabla \overline{\phi}  \nonumber \\
&&\hskip.15in
+   e^{2i\Omega t} P_0\Big( Q_{1,+,+} + Q_{2,+,+} \Big) + e^{-2i\Omega t} P_0\Big( Q_{1,-,-} + Q_{2,-,-} \Big)  \Big\} \Big\rangle \Big|\nonumber \\
&&\hskip-.35in
\leq \frac{C_r}{|\Omega|} \|A^{r+2} e^{\tau A} U_+\|^2 \Big[ \|A^{r} e^{\tau A} \overline{V}\|^2 + \|A^{r} e^{\tau A} U_+\|^2 + \|A^{r} e^{\tau A} U_-\|^2 \nonumber \\
&&\hskip.85in
+ \|A^{r} e^{\tau A} \overline{\phi}\|^2 + \|A^{r} e^{\tau A} \phi_+\|^2 + \|A^{r} e^{\tau A} \phi_-\|^2\Big].
\end{eqnarray*}
Applying differentiation by parts to all the type 3 nonlinear terms (14 terms), one obtains
\begin{eqnarray*}
&&\hskip-.58in
-e^{2i\Omega t} \Big[\Big\langle A^{r} e^{\tau A} (U_+ \cdot \nabla U_+), A^{r} e^{\tau A} \overline{\phi}\Big\rangle + \Big\langle A^{r} e^{\tau A} \Big((\nabla\cdot U_+) U_+\Big), A^{r} e^{\tau A} \overline{\phi}\Big\rangle \nonumber \\ 
&&\hskip-.05in  + \Big\langle A^{r} e^{\tau A} \Big((U_+\cdot \nabla) (\overline{V}-i\overline{V}^\perp)\Big), A^{r} e^{\tau A} \phi_+\Big\rangle \Big] \nonumber\\
&&\hskip-.58in - e^{-2i\Omega t} \Big[\Big\langle A^{r} e^{\tau A} (U_- \cdot \nabla U_-), A^{r} e^{\tau A} \overline{\phi}\Big\rangle + \Big\langle A^{r} e^{\tau A} \Big((\nabla\cdot U_-) U_-\Big), A^{r} e^{\tau A} \overline{\phi}\Big\rangle \nonumber \\ 
&&\hskip-.05in  + \Big\langle A^{r} e^{\tau A} \Big((U_-\cdot \nabla) (\overline{V}+i\overline{V}^\perp)\Big), A^{r} e^{\tau A} \phi_-\Big\rangle \Big] \nonumber\\
&&\hskip-.58in -2e^{i\Omega t}\Big[ \Big\langle A^{r} e^{\tau A} (U_+ \cdot \nabla U_+), A^{r} e^{\tau A} \phi_-\Big\rangle -  \Big\langle A^{r} e^{\tau A} \Big((\int_0^z \nabla \cdot U_+(\boldsymbol{x}',s)ds) \partial_z U_+\Big), A^{r} e^{\tau A} \phi_-\Big\rangle \nonumber\\
&&\hskip-.05in
+ \Big\langle A^{r} e^{\tau A} (U_+ \cdot \nabla U_-), A^{r} e^{\tau A} \phi_+\Big\rangle -  \Big\langle A^{r} e^{\tau A} \Big((\int_0^z \nabla \cdot U_+(\boldsymbol{x}',s)ds) \partial_z U_-\Big), A^{r} e^{\tau A} \phi_+\Big\rangle \Big] \nonumber\\
&&\hskip-.58in -2e^{-i\Omega t}\Big[ \Big\langle A^{r} e^{\tau A} (U_- \cdot \nabla U_+), A^{r} e^{\tau A} \phi_-\Big\rangle -  \Big\langle A^{r} e^{\tau A} \Big((\int_0^z \nabla \cdot U_-(\boldsymbol{x}',s)ds) \partial_z U_+\Big), A^{r} e^{\tau A} \phi_-\Big\rangle \nonumber\\
&&\hskip-.05in
+ \Big\langle A^{r} e^{\tau A} (U_- \cdot \nabla U_-), A^{r} e^{\tau A} \phi_+\Big\rangle -  \Big\langle A^{r} e^{\tau A} \Big((\int_0^z \nabla \cdot U_-(\boldsymbol{x}',s)ds) \partial_z U_-\Big), A^{r} e^{\tau A} \phi_+\Big\rangle \Big] \nonumber \\
&&\hskip-.58in
= -\frac{1}{2i\Omega} \partial_t\Big\{ e^{2i\Omega t} \Big[\Big\langle A^{r} e^{\tau A} (U_+ \cdot \nabla U_+), A^{r} e^{\tau A} \overline{\phi}\Big\rangle + \Big\langle A^{r} e^{\tau A} \Big((\nabla\cdot U_+) U_+\Big), A^{r} e^{\tau A} \overline{\phi}\Big\rangle \nonumber \\ 
&&\hskip-.05in  + \Big\langle A^{r} e^{\tau A} \Big((U_+\cdot \nabla) (\overline{V}-i\overline{V}^\perp)\Big), A^{r} e^{\tau A} \phi_+\Big\rangle \Big] \Big\} \nonumber\\
&&\hskip-.48in
+ \frac{1}{2i\Omega} \partial_t\Big\{ e^{-2i\Omega t} \Big[\Big\langle A^{r} e^{\tau A} (U_- \cdot \nabla U_-), A^{r} e^{\tau A} \overline{\phi}\Big\rangle + \Big\langle A^{r} e^{\tau A} \Big((\nabla\cdot U_-) U_-\Big), A^{r} e^{\tau A} \overline{\phi}\Big\rangle \nonumber \\ 
&&\hskip-.05in  + \Big\langle A^{r} e^{\tau A} \Big((U_-\cdot \nabla) (\overline{V}+i\overline{V}^\perp)\Big), A^{r} e^{\tau A} \phi_-\Big\rangle \Big] \Big\} \nonumber\\
&&\hskip-.48in
- \frac{2}{i\Omega} \partial_t\Big\{ e^{i\Omega t}\Big[ \Big\langle A^{r} e^{\tau A} (U_+ \cdot \nabla U_+), A^{r} e^{\tau A} \phi_-\Big\rangle -  \Big\langle A^{r} e^{\tau A} \Big((\int_0^z \nabla \cdot U_+(\boldsymbol{x}',s)ds) \partial_z U_+\Big), A^{r} e^{\tau A} \phi_-\Big\rangle \nonumber\\
&&\hskip-.05in
+ \Big\langle A^{r} e^{\tau A} (U_+ \cdot \nabla U_-), A^{r} e^{\tau A} \phi_+\Big\rangle -  \Big\langle A^{r} e^{\tau A} \Big((\int_0^z \nabla \cdot U_+(\boldsymbol{x}',s)ds) \partial_z U_-\Big), A^{r} e^{\tau A} \phi_+\Big\rangle \Big]  \Big\} \nonumber \\
&&\hskip-.48in
+ \frac{2}{i\Omega} \partial_t\Big\{ e^{-i\Omega t}\Big[ \Big\langle A^{r} e^{\tau A} (U_- \cdot \nabla U_+), A^{r} e^{\tau A} \phi_-\Big\rangle -  \Big\langle A^{r} e^{\tau A} \Big((\int_0^z \nabla \cdot U_-(\boldsymbol{x}',s)ds) \partial_z U_+\Big), A^{r} e^{\tau A} \phi_-\Big\rangle \nonumber\\
&&\hskip-.05in
+ \Big\langle A^{r} e^{\tau A} (U_- \cdot \nabla U_-), A^{r} e^{\tau A} \phi_+\Big\rangle -  \Big\langle A^{r} e^{\tau A} \Big((\int_0^z \nabla \cdot U_-(\boldsymbol{x}',s)ds) \partial_z U_-\Big), A^{r} e^{\tau A} \phi_+\Big\rangle \Big]  \Big\} \nonumber\\
&&\hskip-.48in +R =: \partial_t N + R,
 \end{eqnarray*}
where $R$ corresponds the remaining terms. Using the similar estimates as (\ref{large-term-nonlinear}), thanks to Young's inequality, when $|\Omega|>1$, we have
\begin{eqnarray}
&&\hskip-.58in
|R| \leq  C_r\Big( \|A^{r+2} e^{\tau A} \overline{V}\|^4 + \|A^{r+2} e^{\tau A} U_+\|^4 + \|A^{r+2} e^{\tau A} U_-\|^4 + 1 \Big)\nonumber\\
&&\hskip0.58in
\times\Big( \frac{1}{2} \|A^{r} e^{\tau A} \overline{\phi}\|^2 + \|A^{r} e^{\tau A} \phi\|^2 + \|A^{r} e^{\tau A} \phi\|^2 \Big)\nonumber \\
&&\hskip-.18in
+ \frac{C_r}{|\Omega|}\Big( |\dot{\tau}|^2+ \|A^{r+2} e^{\tau A} \overline{V}\|^4 + \|A^{r+2} e^{\tau A} U_+\|^4 + \|A^{r+2} e^{\tau A} U_-\|^4 + 1 \Big). \label{estimate-R}
\end{eqnarray} 
For $\partial_t N$, since $\overline{\phi}(0) = \phi_+(0) = \phi_-(0) = 0$, using Lemma \ref{lemma-banach-algebra}, since $\overline{V}$ and $U_\pm$ have zero mean value in $\mathbb{T}^3$, by Young's inequality, we have
\begin{eqnarray}\label{N}
&&\hskip-.78in
|\int_0^t \partial_s N(s)ds| = |N(t)| \leq \frac{C_r}{|\Omega|}\Big( \|A^{r+1} e^{\tau A} \overline{V} \|^2 + \|A^{r+1} e^{\tau A} U_+\|^2 + \|A^{r+1} e^{\tau A} U_- \|^2 \Big) \nonumber\\
&&\hskip.95in
\times \Big(\|A^r e^{\tau A} \overline{\phi} \| + \|A^r e^{\tau A} \phi_+ \|+ \|A^r e^{\tau A} \phi_+ \| \Big).
\end{eqnarray}
\subsubsection{Estimates of Type 4 terms}
The difficulties are on the estimate of type 4 nonlinear terms (23 terms). Thanks to Lemma \ref{lemma-non-small}, since $\nabla\cdot \overline{V}=0$, we have
\begin{eqnarray}
&&\hskip-.58in
\Big|\Big\langle A^r e^{\tau A} (\overline{V}\cdot \nabla \overline{\phi}), A^r e^{\tau A} \overline{\phi} \Big\rangle \Big| \leq C_r \|A^{r} e^{\tau A} \overline{V}\|\|A^{r} e^{\tau A} \overline{\phi}\|^2 + C_r \tau \|A^{r+\frac{1}{2}} e^{\tau A} \overline{V}\| \|A^{r+\frac{1}{2}} e^{\tau A} \overline{\phi}\|^2.
\end{eqnarray}
Thanks to Lemma \ref{lemma-difference-type1}, by integration by parts, we have
\begin{eqnarray}
&&\hskip-.58in
\Big|\Big\langle A^r e^{\tau A} (\overline{V}\cdot \nabla \phi_+), A^r e^{\tau A} \phi_- \Big\rangle + \Big\langle A^r e^{\tau A} (\overline{V}\cdot \nabla \phi_-), A^r e^{\tau A} \phi_+ \Big\rangle\Big| \nonumber \\
&&\hskip-.68in
 \leq\Big|\Big\langle A^r e^{\tau A} (\overline{V}\cdot \nabla \phi_+), A^r e^{\tau A} \phi_- \Big\rangle - \Big\langle  \overline{V}\cdot \nabla A^r e^{\tau A} \phi_+, A^r e^{\tau A} \phi_- \Big\rangle\Big| \nonumber \\
 &&\hskip-.58in
 + \Big|\Big\langle A^r e^{\tau A} (\overline{V}\cdot \nabla \phi_-), A^r e^{\tau A} \phi_+ \Big\rangle - \Big\langle  \overline{V}\cdot \nabla A^r e^{\tau A} \phi_-, A^r e^{\tau A} \phi_+ \Big\rangle \Big|\nonumber \\
 &&\hskip-.58in
 + \Big| \Big\langle  \overline{V}\cdot \nabla A^r e^{\tau A} \phi_+, A^r e^{\tau A} \phi_- \Big\rangle + \Big\langle  \overline{V}\cdot \nabla A^r e^{\tau A} \phi_-, A^r e^{\tau A} \phi_+ \Big\rangle\Big|\nonumber \\
 &&\hskip-.68in
 \leq C_r \|A^{r} e^{\tau A} \overline{V}\| (\|A^{r} e^{\tau A} \phi_+\|^2 + \|A^{r} e^{\tau A} \phi_-\|^2)\nonumber \\
  &&\hskip-.48in
  + C_r\tau \|A^{r+\frac{1}{2}} e^{\tau A} \overline{V}\| (\|A^{r+\frac{1}{2}} e^{\tau A} \phi_+\|^2 + \|A^{r+\frac{1}{2}} e^{\tau A} \phi_-\|^2),
\end{eqnarray}
where we have used
$
\Big| \Big\langle  \overline{V}\cdot \nabla A^r e^{\tau A} \phi_+, A^r e^{\tau A} \phi_- \Big\rangle + \Big\langle  \overline{V}\cdot \nabla A^r e^{\tau A} \phi_-, A^r e^{\tau A} \phi_+ \Big\rangle \Big| = 0
$
by integration by parts and $\nabla\cdot \overline{V}=0$. 
Thanks to Lemma \ref{lemma-difference-type1} and Lemma \ref{lemma-difference-type3}, since $r>\frac{5}{2}$, by integration by parts and by the Sobolev inequality, we have
\begin{eqnarray}
&&\hskip-.68in
\Big|e^{i\Omega t}\Big\langle A^r e^{\tau A} (U_+\cdot \nabla \phi_+), A^r e^{\tau A} \phi_- \Big\rangle + e^{i\Omega t} \Big\langle A^r e^{\tau A} (U_+\cdot \nabla \phi_-), A^r e^{\tau A} \phi_+ \Big\rangle \nonumber \\
&&\hskip-.58in
 - e^{i\Omega t}\Big\langle A^r e^{\tau A} \Big((\int_0^z \nabla\cdot U_+(\boldsymbol{x}',s)ds) \partial_z \phi_+\Big), A^r e^{\tau A} \phi_- \Big\rangle  \nonumber \\
&&\hskip-.58in
 - e^{i\Omega t} \Big\langle A^r e^{\tau A} \Big((\int_0^z \nabla\cdot U_+(\boldsymbol{x}',s)ds) \partial_z \phi_-\Big), A^r e^{\tau A} \phi_+ \Big\rangle\Big| \nonumber \\
&&\hskip-.88in
\leq \Big|\Big\langle A^r e^{\tau A} (U_+\cdot \nabla \phi_+), A^r e^{\tau A} \phi_- \Big\rangle - \Big\langle U_+\cdot \nabla A^r e^{\tau A} \phi_+, A^r e^{\tau A} \phi_- \Big\rangle\Big|\nonumber\\
&&\hskip-.68in
+ \Big| \Big\langle A^r e^{\tau A} (U_+\cdot \nabla \phi_-), A^r e^{\tau A} \phi_+ \Big\rangle - \Big\langle U_+\cdot \nabla  A^r e^{\tau A}  \phi_-, A^r e^{\tau A} \phi_+ \Big\rangle\Big|\nonumber\\
&&\hskip-.68in
+ \Big| \Big\langle A^r e^{\tau A} \Big((\int_0^z \nabla\cdot U_+(\boldsymbol{x}',s)ds) \partial_z \phi_+\Big), A^r e^{\tau A} \phi_- \Big\rangle  \nonumber \\
&&\hskip-.38in
- \Big\langle  (\int_0^z \nabla\cdot U_+(\boldsymbol{x}',s)ds)  A^r e^{\tau A} \partial_z \phi_+, A^r e^{\tau A} \phi_- \Big\rangle  \Big| \nonumber \\
&&\hskip-.68in
+ \Big| \Big\langle A^r e^{\tau A} \Big((\int_0^z \nabla\cdot U_+(\boldsymbol{x}',s)ds) \partial_z \phi_-\Big), A^r e^{\tau A} \phi_+ \Big\rangle  \nonumber \\
&&\hskip-.38in
- \Big\langle  (\int_0^z \nabla\cdot U_+(\boldsymbol{x}',s)ds)  A^r e^{\tau A} \partial_z \phi_-, A^r e^{\tau A} \phi_+ \Big\rangle  \Big| \nonumber \\
&&\hskip-.68in
+ \Big|\Big\langle U_+\cdot \nabla A^r e^{\tau A} \phi_+, A^r e^{\tau A} \phi_- \Big\rangle + \Big\langle U_+\cdot \nabla  A^r e^{\tau A}  \phi_-, A^r e^{\tau A} \phi_+ \Big\rangle \nonumber\\
&&\hskip-.48in
- \Big\langle  (\int_0^z \nabla\cdot U_+(\boldsymbol{x}',s)ds)  A^r e^{\tau A} \partial_z \phi_+, A^r e^{\tau A} \phi_- \Big\rangle \nonumber\\
&&\hskip-.48in
-\Big\langle  (\int_0^z \nabla\cdot U_+(\boldsymbol{x}',s)ds)  A^r e^{\tau A} \partial_z \phi_-, A^r e^{\tau A} \phi_+ \Big\rangle\Big| \nonumber\\
&&\hskip-.88in
\leq C_r \|A^{r+1} e^{\tau A} U_+\| (\|A^{r} e^{\tau A} \phi_+\|^2 + \|A^{r} e^{\tau A} \phi_-\|^2)\nonumber \\
&&\hskip-.48in
+ C_r\tau \|A^{r+\frac{3}{2}} e^{\tau A} U_+\| (\|A^{r+\frac{1}{2}} e^{\tau A} \phi_+\|^2 + \|A^{r+\frac{1}{2}} e^{\tau A} \phi_-\|^2),
\end{eqnarray}
where we have used
\begin{eqnarray*}
&&\hskip-.68in
\Big|\Big\langle U_+\cdot \nabla A^r e^{\tau A} \phi_+, A^r e^{\tau A} \phi_- \Big\rangle + \Big\langle U_+\cdot \nabla  A^r e^{\tau A}  \phi_-, A^r e^{\tau A} \phi_+ \Big\rangle \nonumber\\
&&\hskip-.68in
- \Big\langle  (\int_0^z \nabla\cdot U_+(\boldsymbol{x}',s)ds)  A^r e^{\tau A} \partial_z \phi_+, A^r e^{\tau A} \phi_- \Big\rangle \nonumber\\
&&\hskip-.68in
-\Big\langle  (\int_0^z \nabla\cdot U_+(\boldsymbol{x}',s)ds)  A^r e^{\tau A} \partial_z \phi_-, A^r e^{\tau A} \phi_+ \Big\rangle\Big| =0
\end{eqnarray*}
by integration by parts. Similarly, we have
\begin{eqnarray}
&&\hskip-.68in
\Big|e^{-i\Omega t}\Big\langle A^r e^{\tau A} (U_-\cdot \nabla \phi_+), A^r e^{\tau A} \phi_- \Big\rangle + e^{-i\Omega t} \Big\langle A^r e^{\tau A} (U_-\cdot \nabla \phi_-), A^r e^{\tau A} \phi_+ \Big\rangle \nonumber \\
&&\hskip-.58in
- e^{-i\Omega t}\Big\langle A^r e^{\tau A} \Big((\int_0^z \nabla\cdot U_-(\boldsymbol{x}',s)ds) \partial_z \phi_+\Big), A^r e^{\tau A} \phi_- \Big\rangle  \nonumber \\
&&\hskip-.58in
 - e^{-i\Omega t} \Big\langle A^r e^{\tau A} \Big((\int_0^z \nabla\cdot U_-(\boldsymbol{x}',s)ds) \partial_z \phi_-\Big), A^r e^{\tau A} \phi_+ \Big\rangle \Big|
\nonumber \\
&&\hskip-.88in
\leq C_r \|A^{r+1} e^{\tau A} U_-\| (\|A^{r} e^{\tau A} \phi_+\|^2 + \|A^{r} e^{\tau A} \phi_-\|^2)\nonumber \\
&&\hskip-.48in
+ C_r\tau \|A^{r+\frac{3}{2}} e^{\tau A} U_-\| (\|A^{r+\frac{1}{2}} e^{\tau A} \phi_+\|^2 + \|A^{r+\frac{1}{2}} e^{\tau A} \phi_-\|^2).
\end{eqnarray}
Next, since $-iU_+ = U_+^\perp$, we have
\begin{eqnarray*}
&&\hskip-.58in
\Big|\Big\langle U_+ \cdot \nabla A^r e^{\tau A}\phi_+ , A^r e^{\tau A}\overline{\phi} \Big\rangle + \Big\langle(\nabla\cdot A^r e^{\tau A} \phi_+)U_+ , A^r e^{\tau A}\overline{\phi} \Big\rangle \nonumber \\
&&\hskip-.38in
+ \Big\langle U_+  \cdot \nabla A^r e^{\tau A} (\overline{\phi} - i \overline{\phi}^\perp ), A^r e^{\tau A} \phi_+\Big\rangle \Big|\nonumber \\
&&\hskip-.58in
\leq \Big|\Big\langle U_+ \cdot \nabla A^r e^{\tau A}\phi_+ , A^r e^{\tau A}\overline{\phi} \Big\rangle +  \Big\langle U_+  \cdot \nabla A^r e^{\tau A} \overline{\phi} , A^r e^{\tau A} \phi_+\Big\rangle  \Big|\nonumber \\
&&\hskip-.38in
+ \Big| \Big\langle(\nabla\cdot A^r e^{\tau A} \phi_+)U_+ , A^r e^{\tau A}\overline{\phi} \Big\rangle  +  \Big\langle U_+^\perp  \cdot \nabla A^r e^{\tau A} \overline{\phi}^\perp , A^r e^{\tau A} \phi_+\Big\rangle\Big| \nonumber\\
&&\hskip-.58in
\leq \Big|\Big\langle(\nabla \cdot U_+) A^r e^{\tau A}\phi_+ , A^r e^{\tau A}\overline{\phi} \Big\rangle  \Big| + \Big| \Big\langle A^r e^{\tau A} \phi_+ \cdot \nabla U_+, A^r e^{\tau A} \overline{\phi}\Big\rangle \Big| \nonumber \\
&&\hskip-.38in
+ \Big| \Big\langle U_+^\perp  \cdot \nabla A^r e^{\tau A} \overline{\phi}^\perp , A^r e^{\tau A} \phi_+\Big\rangle - \Big\langle A^r e^{\tau A} \phi_+ \cdot \nabla A^r e^{\tau A} \overline{\phi} , U_+\Big\rangle \Big|.
\end{eqnarray*}
Notice that 
\begin{eqnarray*}
&&\hskip-.58in
\Big| \Big\langle U_+^\perp  \cdot \nabla A^r e^{\tau A} \overline{\phi}^\perp , A^r e^{\tau A} \phi_+\Big\rangle - \Big\langle A^r e^{\tau A} \phi_+ \cdot \nabla A^r e^{\tau A} \overline{\phi} , U_+\Big\rangle \Big| \nonumber \\
&&\hskip-.73in 
= \Big| \Big\langle(\nabla\cdot A^r e^{\tau A}\overline{\phi} ) U_+, A^r e^{\tau A} \phi_+\Big\rangle \Big| = 0.
\end{eqnarray*}
Therefore, by the Sobolev inequality and the H\"older inequality, since $r>\frac{5}{2}$, we have
\begin{eqnarray*}
&&\hskip-.58in
\Big|\Big\langle U_+ \cdot \nabla A^r e^{\tau A}\phi_+ , A^r e^{\tau A}\overline{\phi} \Big\rangle + \Big\langle(\nabla\cdot A^r e^{\tau A} \phi_+)U_+ , A^r e^{\tau A}\overline{\phi} \Big\rangle \nonumber \\
&&\hskip-.38in
+ \Big\langle U_+  \cdot \nabla A^r e^{\tau A} (\overline{\phi} - i \overline{\phi}^\perp ), A^r e^{\tau A} \phi_+\Big\rangle \Big|\nonumber \\
&&\hskip-.58in
\leq  C_r \|\nabla U_+\|_{L^\infty}( \frac{1}{2} \|A^r e^{\tau A}\overline{\phi}\|^2 + \|A^r e^{\tau A}\phi_+\|^2) 
\leq C_r \|A^{r} e^{\tau A} U_+\| ( \frac{1}{2} \|A^r e^{\tau A}\overline{\phi}\|^2 + \|A^r e^{\tau A}\phi_+\|^2).
\end{eqnarray*}
Based on this, thanks to Lemma \ref{lemma-difference-type1} and Lemma \ref{lemma-difference-type2}, we have
\begin{eqnarray}
&&\hskip-.58in
\Big|e^{2i\Omega t}\Big\langle A^r e^{\tau A} \Big(U_+ \cdot \nabla \phi_+ + (\nabla\cdot \phi_+) U_+ \Big) , A^r e^{\tau A}\overline{\phi} \Big\rangle  \nonumber \\
&&\hskip-.38in
+ e^{2i\Omega t} \Big\langle A^r e^{\tau A} \Big(U_+ \cdot \nabla (\overline{\phi} - i\overline{\phi}^\perp ) \Big) , A^r e^{\tau A}\phi_+ \Big\rangle \Big| \nonumber\\
&&\hskip-.78in
\leq \Big| \Big\langle A^r e^{\tau A} \Big(U_+ \cdot \nabla \phi_+  \Big), A^r e^{\tau A}\overline{\phi}\Big\rangle - \Big\langle  U_+ \cdot \nabla A^r e^{\tau A} \phi_+ , A^r e^{\tau A}\overline{\phi}\Big\rangle  \Big|\nonumber \\
&&\hskip-.38in 
+ \Big| \Big\langle A^r e^{\tau A} \Big((\nabla \cdot \phi_+) U_+\Big), A^r e^{\tau A}\overline{\phi}\Big\rangle - \Big\langle  (\nabla \cdot A^r e^{\tau A} \phi_+) U_+ , A^r e^{\tau A}\overline{\phi}\Big\rangle  \Big|\nonumber \\
&&\hskip-.38in 
+\Big| \Big\langle A^r e^{\tau A} \Big(U_+ \cdot \nabla \overline{\phi}  \Big), A^r e^{\tau A} \phi_+ \Big\rangle - \Big\langle  U_+ \cdot \nabla A^r e^{\tau A} \overline{\phi} , A^r e^{\tau A}\phi_+ \Big\rangle  \Big|\nonumber \\
&&\hskip-.38in 
+\Big| \Big\langle A^r e^{\tau A} \Big(U_+ \cdot \nabla \overline{\phi}^\perp  \Big), A^r e^{\tau A} \phi_+ \Big\rangle - \Big\langle  U_+ \cdot \nabla A^r e^{\tau A} \overline{\phi}^\perp  , A^r e^{\tau A}\phi_+  \Big\rangle \Big|\nonumber \\
&&\hskip-.38in 
+ \Big|\Big\langle U_+ \cdot \nabla A^r e^{\tau A}\phi_+ , A^r e^{\tau A}\overline{\phi} \Big\rangle + \Big\langle(\nabla\cdot A^r e^{\tau A} \phi_+)U_+ , A^r e^{\tau A}\overline{\phi} \Big\rangle \nonumber \\
&&\hskip-.18in
+ \Big\langle U_+  \cdot \nabla A^r e^{\tau A} (\overline{\phi} - i \overline{\phi}^\perp ), A^r e^{\tau A} \phi_+\Big\rangle \Big|\nonumber \\
&&\hskip-.78in
\leq C_r \|A^{r} e^{\tau A} U_+\| (\frac{1}{2}\|A^r e^{\tau A}\overline{\phi}\|^2 + \|A^r e^{\tau A}\phi_+\|^2) \nonumber \\
&&\hskip-.58in
+ C_r \tau \|A^{r+\frac{1}{2}} e^{\tau A} U_+\| (\|A^{r+\frac{1}{2}}  e^{\tau A}\overline{\phi}\|^2 + \|A^{r+\frac{1}{2}}  e^{\tau A}\phi_+\|^2).
\end{eqnarray}
Similarly, we have
\begin{eqnarray}
&&\hskip-.58in
\Big|e^{-2i\Omega t}\Big\langle A^r e^{\tau A} \Big(U_- \cdot \nabla \phi_- + (\nabla\cdot \phi_-) U_- \Big) , A^r e^{\tau A}\overline{\phi} \Big\rangle  \nonumber \\
&&\hskip-.38in
+ e^{-2i\Omega t} \Big\langle A^r e^{\tau A} \Big(U_- \cdot \nabla (\overline{\phi} + i\overline{\phi}^\perp ) \Big) , A^r e^{\tau A}\phi_- \Big\rangle \Big| \nonumber\\
&&\hskip-.78in
\leq C_r \|A^{r} e^{\tau A} U_-\| (\frac{1}{2} \|A^r e^{\tau A}\overline{\phi}\|^2 + \|A^r e^{\tau A}\phi_-\|^2) \nonumber \\
&&\hskip-.58in
+ C_r \tau \|A^{r+\frac{1}{2}} e^{\tau A} U_-\| (\|A^{r+\frac{1}{2}}  e^{\tau A}\overline{\phi}\|^2 + \|A^{r+\frac{1}{2}}  e^{\tau A}\phi_-\|^2).
\end{eqnarray}

For the rest parts in type 4, there is no cancellation as above. First, by the H\"older inequality, we have
\begin{eqnarray*}
&&\hskip-.58in
    \Big|\Big\langle  U_+\cdot \nabla A^r e^{\tau A} (\overline{\phi}+i\overline{\phi}^\perp)), A^r e^{\tau A} \phi_- \Big\rangle  \Big| \nonumber\\
&&\hskip-.68in
\leq \Big|\Big\langle  A^{\frac{1}{2}} U_+\cdot \nabla A^{r-\frac{1}{2}} e^{\tau A} (\overline{\phi}+i\overline{\phi}^\perp)), A^r e^{\tau A} \phi_- \Big\rangle  \Big| \nonumber\\
&&\hskip-.58in
+ \Big|\Big\langle   U_+\cdot \nabla A^{r-\frac{1}{2}} e^{\tau A} (\overline{\phi}+i\overline{\phi}^\perp)), A^{r+\frac{1}{2}} e^{\tau A} \phi_- \Big\rangle  \Big|\nonumber\\
&&\hskip-.68in
\leq C_r (\|U_+\|_{L^\infty} + \|A^{\frac{1}{2}} U_+\|_{L^\infty})(\|A^{r+\frac{1}{2}} e^{\tau A} \overline{\phi}\|^2 + \|A^{r+\frac{1}{2}} e^{\tau A} \phi_-\|^2).
\end{eqnarray*}
Based on this, using Lemma \ref{lemma-difference-type1} to \ref{lemma-difference-type4}, we have
\begin{eqnarray}\label{difficulty-1}
&&\hskip-.58in
\Big|\Big\langle A^r e^{\tau A} (U_+\cdot \nabla (\overline{\phi}+i\overline{\phi}^\perp)), A^r e^{\tau A} \phi_- \Big\rangle\Big| \nonumber\\
&&\hskip-.68in
\leq \Big|\Big\langle A^r e^{\tau A} (U_+\cdot \nabla (\overline{\phi}+i\overline{\phi}^\perp)), A^r e^{\tau A} \phi_- \Big\rangle - \Big\langle  U_+\cdot \nabla A^r e^{\tau A} (\overline{\phi}+i\overline{\phi}^\perp)), A^r e^{\tau A} \phi_- \Big\rangle \Big| \nonumber\\
&&\hskip-.58in
+ \Big|\Big\langle  U_+\cdot \nabla A^r e^{\tau A} (\overline{\phi}+i\overline{\phi}^\perp)), A^r e^{\tau A} \phi_- \Big\rangle  \Big|\nonumber\\
&&\hskip-.68in
\leq C_r\|A^r e^{\tau A} U_+\|(\frac{1}{2} \|A^r e^{\tau A} \overline{\phi}\|^2 + \|A^r e^{\tau A} \phi_-\|^2) \nonumber\\
&&\hskip-.58in
+ C_r(\tau \|A^{r+\frac{1}{2}} e^{\tau A} U_+\| + \|U_+\|_{L^\infty}+ \|A^{\frac{1}{2}} U_+\|_{L^\infty})(\|A^{r+\frac{1}{2}} e^{\tau A} \overline{\phi}\|^2 + \|A^{r+\frac{1}{2}} e^{\tau A} \phi_-\|^2).
\end{eqnarray}
Similarly,
\begin{eqnarray}\label{difficulty-2}
&&\hskip-.58in
\Big|\Big\langle A^r e^{\tau A} (U_-\cdot \nabla (\overline{\phi}-i\overline{\phi}^\perp)), A^r e^{\tau A} \phi_+ \Big\rangle\Big| \nonumber\\
&&\hskip-.68in
\leq C_r\|A^r e^{\tau A} U_-\|(\|A^r e^{\tau A} \overline{\phi}\|^2 + \|A^r e^{\tau A} \phi_+\|^2) \nonumber\\
&&\hskip-.58in
+ C_r(\tau \|A^{r+\frac{1}{2}} e^{\tau A} U_-\| + \| U_-\|_{L^\infty}+\|A^{\frac{1}{2}} U_-\|_{L^\infty})(\|A^{r+\frac{1}{2}} e^{\tau A} \overline{\phi}\|^2 + \|A^{r+\frac{1}{2}} e^{\tau A} \phi_+\|^2).
\end{eqnarray}

Next, by the H\"older inequality, we have
\begin{eqnarray*}
&&\hskip-.58in
\Big| \Big\langle (\partial_z U_+)  A^r e^{\tau A}  (\int_0^z \nabla \cdot \phi_+(\boldsymbol{x}',s)ds)  , A^r e^{\tau A}\phi_- \Big\rangle \Big| \nonumber \\
&&\hskip-.68in 
\leq \Big| \Big\langle (A^{\frac{1}{2}}\partial_z U_+)  A^{r-\frac{1}{2}} e^{\tau A}  (\int_0^z \nabla \cdot \phi_+(\boldsymbol{x}',s)ds)  , A^r e^{\tau A}\phi_- \Big\rangle \Big| \nonumber \\
&&\hskip-.48in 
+ \Big| \Big\langle (\partial_z U_+)  A^{r-\frac{1}{2}} e^{\tau A}  (\int_0^z \nabla \cdot \phi_+(\boldsymbol{x}',s)ds)  , A^{r+\frac{1}{2}} e^{\tau A}\phi_- \Big\rangle \Big| \nonumber \\
&&\hskip-.68in 
\leq C_r(\|\partial_z U_+\|_{L^\infty}+\|A^{\frac{1}{2}} \partial_z U_+\|_{L^\infty}) (\|A^{r+\frac{1}{2}} e^{\tau A} \phi_+\|^2 + \|A^{r+\frac{1}{2}} e^{\tau A} \phi_-\|^2).
\end{eqnarray*}
Based on this, thanks to Lemma \ref{lemma-difference-type1} to \ref{lemma-difference-type4}, we obtain
\begin{eqnarray}
&&\hskip-.58in
 \Big| \Big\langle A^r e^{\tau A} \Big( (\int_0^z \nabla \cdot \phi_+(\boldsymbol{x}',s)ds) \partial_z U_+ \Big) , A^r e^{\tau A}\phi_- \Big\rangle \Big|\nonumber \\
 &&\hskip-.68in
 \leq \Big| \Big\langle A^r e^{\tau A} \Big( (\int_0^z \nabla \cdot \phi_+(\boldsymbol{x}',s)ds) \partial_z U_+ \Big) , A^r e^{\tau A}\phi_- \Big\rangle \nonumber \\ 
 &&\hskip-.48in
 - \Big\langle (\partial_z U_+)  A^r e^{\tau A}  (\int_0^z \nabla \cdot \phi_+(\boldsymbol{x}',s)ds)  , A^r e^{\tau A}\phi_- \Big\rangle \Big|\nonumber\\
  &&\hskip-.48in
  + \Big| \Big\langle (\partial_z U_+)  A^r e^{\tau A}  (\int_0^z \nabla \cdot \phi_+(\boldsymbol{x}',s)ds)  , A^r e^{\tau A}\phi_- \Big\rangle \Big| \nonumber\\
 &&\hskip-.68in
 \leq C_r\|A^{r+1} e^{\tau A} U_+\|(\|A^r e^{\tau A} \phi_+\|^2 + \|A^r e^{\tau A} \phi_-\|^2) \nonumber\\
&&\hskip-.58in
+ C_r(\tau \|A^{r+\frac{3}{2}} e^{\tau A} U_+\| + \| \partial_z U_+\|_{L^\infty} + \|A^{\frac{1}{2}} \partial_z U_+\|_{L^\infty} )\nonumber\\
&&\hskip-.28in
\times (\|A^{r+\frac{1}{2}} e^{\tau A} \phi_+\|^2 + \|A^{r+\frac{1}{2}} e^{\tau A} \phi_-\|^2).\label{difficulty-3}
\end{eqnarray}
Similarly, we have
\begin{eqnarray}\label{difficulty-4}
&&\hskip-.68in
\Big| \Big\langle A^r e^{\tau A} \Big( (\int_0^z \nabla \cdot \phi_+(\boldsymbol{x}',s)ds) \partial_z U_- \Big) , A^r e^{\tau A}\phi_+ \Big\rangle \Big| \nonumber \\
&&\hskip-.58in
+ \Big| \Big\langle A^r e^{\tau A} \Big( (\int_0^z \nabla \cdot \phi_-(\boldsymbol{x}',s)ds) \partial_z U_+ \Big) , A^r e^{\tau A}\phi_- \Big\rangle \Big| \nonumber \\
&&\hskip-.58in
+ \Big| \Big\langle A^r e^{\tau A} \Big( (\int_0^z \nabla \cdot \phi_-(\boldsymbol{x}',s)ds) \partial_z U_- \Big) , A^r e^{\tau A}\phi_+ \Big\rangle \Big| \nonumber\\
&&\hskip-.78in
\leq C_r(\|A^{r+1} e^{\tau A} U_+\|+ \|A^{r+1} e^{\tau A} U_-\|)(\|A^r e^{\tau A} \phi_+\|^2 + \|A^r e^{\tau A} \phi_-\|^2) \nonumber\\
&&\hskip-.58in
+ C_r\Big(\tau \|A^{r+\frac{3}{2}} e^{\tau A} U_+\| + \tau \|A^{r+\frac{3}{2}} e^{\tau A} U_-\| + \| \partial_z U_+\|_{L^\infty} + \| \partial_z U_-\|_{L^\infty} \nonumber\\
&&\hskip-.28in
+ \|A^{\frac{1}{2}} \partial_z U_+\|_{L^\infty} + \|A^{\frac{1}{2}} \partial_z U_-\|_{L^\infty}\Big)
\Big(\|A^{r+\frac{1}{2}} e^{\tau A} \phi_+\|^2 + \|A^{r+\frac{1}{2}} e^{\tau A} \phi_-\|^2\Big).
\end{eqnarray}
\subsubsection{Finishing of the proof to Theorem \ref{theorem-main}}
Finally, taking summation of (\ref{energy-difference-phibar}) and (\ref{energy-difference-phipm}), and using estimates (\ref{estimate-general-terms})--(\ref{difficulty-4}) for all the nonlinear terms (71 terms), we obtain
\begin{eqnarray}\label{main-estimate-1}
&&\hskip-.58in  \frac{d}{dt} \Big(\frac{1}{2}\|A^r e^{\tau A} \overline{\phi} \|^2 + \|A^r e^{\tau A} \phi_+ \|^2 + \|A^r e^{\tau A} \phi_- \|^2 \Big) \nonumber \\
&&\hskip-.68in
\leq  \Big[\dot{\tau} + C_r\big( \|A^r e^{\tau A} \overline{\phi} \| + \|A^r e^{\tau A} \phi_+ \|+ \|A^r e^{\tau A} \phi_+ \| \big)  \nonumber\\
&&\hskip-.25in
+ C_r\tau \big( \|A^{r+\frac{1}{2}} e^{\tau A} \overline{V} \| + \|A^{r+\frac{3}{2}} e^{\tau A} U_+\| + \|A^{r+\frac{3}{2}} e^{\tau A} U_- \| \big)  \nonumber \\
&&\hskip-.25in
+ C_r\big(\| U_+\|_{L^\infty} + \|  U_-\|_{L^\infty} + \| \partial_z U_+\|_{L^\infty} + \| \partial_z U_-\|_{L^\infty}  \nonumber \\
&&\hskip.05in
+\|A^{\frac{1}{2}} U_+\|_{L^\infty} + \|A^{\frac{1}{2}}  U_-\|_{L^\infty} + \|A^{\frac{1}{2}} \partial_z U_+\|_{L^\infty} + \|A^{\frac{1}{2}} \partial_z U_-\|_{L^\infty} \big)\Big]
\nonumber \\
&&\hskip-.35in
\times \Big( \|A^{r+\frac{1}{2}} e^{\tau A} \overline{\phi} \|^2 + 2\|A^{r+\frac{1}{2}} e^{\tau A} \phi_+ \|^2 + 2\|A^{r+\frac{1}{2}} e^{\tau A} \phi_- \|^2\Big) \nonumber\\
&&\hskip-.58in
+ C_r \Big( \|A^{r+2} e^{\tau A} \overline{V}\|^4 + \|A^{r+2} e^{\tau A} U_+\|^4 + \|A^{r+2} e^{\tau A} U_-\|^4 + 1 \Big) \nonumber \\
&&\hskip-.05in
\times \Big(\frac{1}{2}\|A^r e^{\tau A} \overline{\phi} \|^2 + \|A^r e^{\tau A} \phi_+ \|^2 + \|A^r e^{\tau A} \phi_- \|^2 \Big) \nonumber\\
&&\hskip-.58in + \frac{C_r}{|\Omega|}\Big( |\dot{\tau}|^2 + \|A^{r+2} e^{\tau A} \overline{V}\|^4 + \|A^{r+2} e^{\tau A} U_+\|^4 + \|A^{r+2} e^{\tau A} U_-\|^4 + 1 \Big) + \partial_t N. 
\end{eqnarray}
Observe that eventually we will set 
\begin{eqnarray*}
&&\hskip-.58in
\dot{\tau} + C_r\big( \|A^r e^{\tau A} \overline{\phi} \| + \|A^r e^{\tau A} \phi_+ \|+ \|A^r e^{\tau A} \phi_+ \| \big)  \nonumber\\
&&\hskip-.35in
+ C_r\tau \big( \|A^{r+\frac{1}{2}} e^{\tau A} \overline{V} \| + \|A^{r+\frac{3}{2}} e^{\tau A} U_+\| + \|A^{r+\frac{3}{2}} e^{\tau A} U_- \| \big) \nonumber\\
&&\hskip-.35in
+ C_r\big(\| U_+\|_{L^\infty} + \|  U_-\|_{L^\infty} + \| \partial_z U_+\|_{L^\infty} + \| \partial_z U_-\|_{L^\infty}  \nonumber \\
&&\hskip-.05in
+\|A^{\frac{1}{2}} U_+\|_{L^\infty} + \|A^{\frac{1}{2}}  U_-\|_{L^\infty} + \|A^{\frac{1}{2}} \partial_z U_+\|_{L^\infty} + \|A^{\frac{1}{2}} \partial_z U_-\|_{L^\infty} \big) = 0.
\end{eqnarray*}
Therefore, by the Sobolev inequality, the Poincar\'e inequality, and Young's inequality, since $r>\frac{5}{2}$, $\tau \leq \tau_0$, and $U_\pm$ have zero mean value, we have
\begin{eqnarray*}
&&\hskip-.58in
|\dot{\tau}|^2 \leq C_r\big( \|A^r e^{\tau A} \overline{\phi} \|^2 + \|A^r e^{\tau A} \phi_+ \|^2+ \|A^r e^{\tau A} \phi_+ \|^2 \big)  \nonumber\\
&&\hskip-.15in
+ C_r(\tau^2_0 + 1) \big( \|A^{r+\frac{1}{2}} e^{\tau A} \overline{V} \|^2 + \|A^{r+\frac{3}{2}} e^{\tau A} U_+\|^2 + \|A^{r+\frac{3}{2}} e^{\tau A} U_- \|^2 \big).
\end{eqnarray*}
By Young's inequality, the term $\frac{|\dot{\tau}|^2}{|\Omega|}$ can be combined with other terms, and we can rewrite (\ref{main-estimate-1}) as
\begin{eqnarray}\label{main-estimate-2}
&&\hskip-.58in  \frac{d}{dt} \Big(\frac{1}{2}\|A^r e^{\tau A} \overline{\phi} \|^2 + \|A^r e^{\tau A} \phi_+ \|^2 + \|A^r e^{\tau A} \phi_- \|^2 \Big) \nonumber \\
&&\hskip-.68in
\leq \Big[\dot{\tau} + C_r\big( \|A^r e^{\tau A} \overline{\phi} \| + \|A^r e^{\tau A} \phi_+ \|+ \|A^r e^{\tau A} \phi_+ \| \big)  \nonumber\\
&&\hskip-.35in
+ C_r\tau \big( \|A^{r+\frac{1}{2}} e^{\tau A} \overline{V} \| + \|A^{r+\frac{3}{2}} e^{\tau A} U_+\| + \|A^{r+\frac{3}{2}} e^{\tau A} U_- \| \big) \nonumber \\
&&\hskip-.35in
+ C_r\big(\| U_+\|_{L^\infty} + \|  U_-\|_{L^\infty} + \| \partial_z U_+\|_{L^\infty} + \| \partial_z U_-\|_{L^\infty}  \nonumber \\
&&\hskip-.05in
+\|A^{\frac{1}{2}} U_+\|_{L^\infty} + \|A^{\frac{1}{2}}  U_-\|_{L^\infty} + \|A^{\frac{1}{2}} \partial_z U_+\|_{L^\infty} + \|A^{\frac{1}{2}} \partial_z U_-\|_{L^\infty} \big)\Big]\nonumber\\ 
&&\hskip-.55in
\times \Big[ \|A^{r+\frac{1}{2}} e^{\tau A} \overline{\phi} \|^2 + 2\|A^{r+\frac{1}{2}} e^{\tau A} \phi_+ \|^2 + 2\|A^{r+\frac{1}{2}} e^{\tau A} \phi_- \|^2\Big] \nonumber\\
&&\hskip-.58in
+ C_r \Big( \|A^{r+2} e^{\tau A} \overline{V}\|^4 + \|A^{r+2} e^{\tau A} U_+\|^4 + \|A^{r+2} e^{\tau A} U_-\|^4 + 1 \Big) \nonumber \\
&&\hskip-.05in
\times \Big(\frac{1}{2}\|A^r e^{\tau A} \overline{\phi} \|^2 + \|A^r e^{\tau A} \phi_+ \|^2 + \|A^r e^{\tau A} \phi_- \|^2 \Big) \nonumber\\
&&\hskip-.58in + \frac{C_{r,\tau_0}}{|\Omega|}\Big( \|A^{r+2} e^{\tau A} \overline{V}\|^4 + \|A^{r+2} e^{\tau A} U_+\|^4 + \|A^{r+2} e^{\tau A} U_-\|^4 + 1 \Big) + \partial_t N.
\end{eqnarray}
Denote by
\begin{equation*}
    F := \frac{1}{2}\|A^r e^{\tau A} \overline{\phi} \|^2 + \|A^r e^{\tau A} \phi_+ \|^2 + \|A^r e^{\tau A} \phi_- \|^2 ,
\end{equation*}
\begin{equation*}
    G := \|A^{r+\frac{1}{2}} e^{\tau A} \overline{\phi} \|^2 + 2\|A^{r+\frac{1}{2}} e^{\tau A} \phi_+ \|^2 + 2\|A^{r+\frac{1}{2}} e^{\tau A} \phi_- \|^2, \text{ and }
\end{equation*}
\begin{equation*}
   K(t) := C_{M,\tau_0}^{\exp(C_r t)}, \qquad \widetilde{K}(t) : = e^{K(t)},
\end{equation*}
which are double exponential and triple exponential in time. We will follow the rule on the use of notation as indicated in Remark \ref{remark-kt}. From Proposition \ref{global-limit} and thanks to Lemma \ref{lemma-vtilde-upm}, one has
$\|e^{\tau(t) A} \overline{V}(t)
    \|_{H^{r+3}}  \leq K(t) \leq \widetilde{K}_1(t)$ and $e^{\tau(t) A} U_\pm(t)\|_{H^{r+2}} \leq   \widetilde{K}_1(t),$
provided that $\tau(t)$ satisfies (\ref{tau-limit-system}). Observe that in (\ref{main-estimate-2}), $\|U_\pm\|_{L^\infty}$, $\|A^{\frac{1}{2}} U_\pm\|_{L^\infty}$,  $\|\partial_z U_\pm\|_{L^\infty}$, and $\|A^{\frac{1}{2}}\partial_z U_\pm\|_{L^\infty}$ are the terms force the smallness assumption on Sobolev norm of the baroclinic mode. For $\delta >0$, by Proposition \ref{global-limit} and Lemma \ref{lemma-vtilde-upm}, thanks to the Sobolev inequality, when $\|\widetilde{V}_0\|_{H^{3+\delta}} = \|\widetilde{v}_0\|_{H^{3+\delta}}\leq \frac{1}{|\Omega_0|}$,
we have 
\begin{equation*}
    \|U_\pm\|_{L^\infty} + \|A^{\frac{1}{2}} U_\pm\|_{L^\infty}+ \|\partial_z U_\pm\|_{L^\infty} + \|A^{\frac{1}{2}}\partial_z U_\pm\|_{L^\infty} \leq C\| \widetilde{V}\|_{H^{3+\delta } }\leq \frac{C\widetilde{K}_1(t) }{|\Omega_0|}.
\end{equation*}
Since $|\Omega|\geq |\Omega_0|$, we can rewrite (\ref{main-estimate-2}) as 
\begin{eqnarray}\label{main-estimate-3}
&&\hskip-.58in
\frac{dF}{dt} \leq (\dot{\tau} +C_r F^{\frac{1}{2}} + \tau \widetilde{K}_2 + \frac{\widetilde{K}_2}{|\Omega_0|})G + \widetilde{K}_2 F + \frac{\widetilde{K}_2}{|\Omega_0|} +\partial_t N.
\end{eqnarray}
By setting $ \dot{\tau} +C_r F^{\frac{1}{2}} + \tau \widetilde{K}_2 + \frac{\widetilde{K}_2}{|\Omega_0|} = 0$, it follows that
$
\frac{dF}{dt} \leq \widetilde{K}_2 F + \frac{\widetilde{K}_2}{|\Omega_0|} + \partial_t N.
$
By the Gr\"onwall inequality,
\begin{eqnarray*}\label{main-estimate-3-2}
\frac{d}{dt}(Fe^{-\int_0^t \widetilde{K}_2(s)ds}) \leq \frac{\widetilde{K}_2}{|\Omega_0|} +(\partial_t N) e^{-\int_0^t \widetilde{K}_2(s)ds}.
\end{eqnarray*}
Integrating from $0$ to $t$, noticing that $F(0)=0$, one obtains that
\begin{eqnarray*}\label{main-estimate-3-3}
F(t) e^{-\int_0^t \widetilde{K}_2(s)ds} \leq \frac{1}{|\Omega_0|}\int_0^t \widetilde{K}_2(s)ds  + \int_0^t (\partial_s N(s)) e^{-\int_0^s \widetilde{K}_2(\xi)d\xi} ds.
\end{eqnarray*}
From (\ref{N}), we know $|N(t)| \leq \frac{1}{|\Omega_0|}\widetilde{K}_3(t)F^{\frac{1}{2}}$. Moreover, $\frac{1}{|\Omega_0|}\widetilde{K}_3(t)F^{\frac{1}{2}}$ is increasing in time. By integration by parts in time, thanks to the Cauchy–Schwarz inequality, since $N(0)=0$, we have
\begin{eqnarray*}
&&\hskip-.58in
    \int_0^t (\partial_s N(s)) e^{-\int_0^s \widetilde{K}_2(\xi)d\xi} ds \leq |N(t)| + \int_0^t |N(s)||\partial_s e^{-\int_0^s \widetilde{K}_2(\xi)d\xi}| ds \nonumber\\
&&\hskip-.58in
    \leq \frac{1}{|\Omega_0|} \widetilde{K}_3 F^{\frac{1}{2}} + \frac{t}{|\Omega_0|} \widetilde{K}_3 F^{\frac{1}{2}} \widetilde{K}_2  \leq  \frac{1}{|\Omega_0|}\widetilde{K}_4 + \frac{1}{|\Omega_0|}F.
\end{eqnarray*}
Thus, one gets
$
F(t) \leq \frac{1}{|\Omega_0|} e^{\widetilde{K}_5(t)} + \frac{1}{|\Omega_0|} e^{\widetilde{K}_5(t)}F(t),
$
which is equivalent to 
\begin{eqnarray}\label{main-estimate-3-5}
F(t) \leq \frac{ e^{\widetilde{K}_5(t)} }{|\Omega_0| -  e^{\widetilde{K}_5(t)}  }.
\end{eqnarray}
Plugging this back to $ \dot{\tau} +C_r F^{\frac{1}{2}} + \tau \widetilde{K}_2 + \frac{\widetilde{K}_2}{|\Omega_0|} = 0$, one can require that 
\begin{eqnarray*}\label{main-estimate-tau-1}
\dot{\tau} + \frac{ e^{\widetilde{K}_6(t)} }{\sqrt{|\Omega_0| -  e^{\widetilde{K}_6(t)}}  } + \tau \widetilde{K}_6 + \frac{1}{|\Omega_0|}\widetilde{K}_6 \leq 0.
\end{eqnarray*}
By the Gr\"onwall inequality, one can require
\begin{eqnarray*}\label{main-estimate-tau-2}
\frac{d}{dt}(\tau e^{\int_0^t \widetilde{K}_6(s)ds}) \leq \frac{ -e^{\widetilde{K}_7(t)} }{\sqrt{|\Omega_0| -  e^{\widetilde{K}_6(t)}  }}  - \frac{e^{\widetilde{K}_7(t)}}{|\Omega_0|}.
\end{eqnarray*}
Integrating from $0$ to $t$, for some suitable function $\widetilde{K}_0(t)$, one can require that
\begin{eqnarray}\label{main-estimate-tau-3}
\tau(t) = \Big(\tau_0 - \int_0^t\frac{ e^{\widetilde{K}_0(s)} }{\sqrt{|\Omega_0| -  e^{\widetilde{K}_0(s)}  } } ds- \int_0^t \frac{e^{\widetilde{K}_0(s)}}{|\Omega_0|} ds\Big) e^{-\int_0^t \widetilde{K}_0(s)ds}.
\end{eqnarray}
Notice that $\tau$ in (\ref{main-estimate-tau-3}) also satisfies (\ref{tau-limit-system}) when $\widetilde{K}_0(t)$ is chosen suitably. In order to have $\tau(t) > 0$,  we just need to require that
\begin{eqnarray}\label{main-estimate-tau-4}
\tau_0 \geq \frac{ 3e^{\widetilde{K}_8(t)} }{\sqrt{|\Omega_0| -  e^{\widetilde{K}_8(t)} } } \text{     and     } \tau_0 \geq  \frac{3e^{\widetilde{K}_8(t)}}{|\Omega_0|}
\end{eqnarray}
for some suitable function $\widetilde{K}_8(t) > \widetilde{K}_0(t)$. For some new $\widetilde{K}(t) > \widetilde{K}_8(t)$ and the given $\Omega_0$, let $\mathcal{T}$ satisfy
\begin{equation}\label{main-estimate-tau-5}
    C_{\tau_0} e^{\widetilde{K}(\mathcal{T})} = |\Omega_0|,
\end{equation} 
then the two conditions in (\ref{main-estimate-tau-4}) are satisfied on $t\in[0,\mathcal{T}]$. Thus, $\tau(t)>0$ on $t\in[0,\mathcal{T}]$. From (\ref{main-estimate-tau-5}), we know that $e^{\widetilde{K}(\mathcal{T})}\geq \frac{|\Omega_0|}{2C_{\tau_0}},$ and thus the time  $\mathcal{T}$ satisfies 
\begin{eqnarray}\label{main-time}
\mathcal{T} \gtrsim \ln(\ln(\ln(\ln|\Omega_0|)))\rightarrow \infty,
\end{eqnarray}
as $|\Omega_0|\rightarrow \infty$. 

When $\widetilde{K}(t)$ is chosen suitably, from (\ref{main-estimate-3-5}), we know that 
\begin{equation*}
    \|A^r e^{\tau(t) A} \overline{\phi}(t) \|^2 + \|A^r e^{\tau(t) A} \phi_+(t) \|^2 + \|A^r e^{\tau(t) A} \phi_-(t) \|^2 \leq \frac{ e^{\widetilde{K}(t)} }{|\Omega_0| -  e^{\widetilde{K}(t)}  } < \infty
\end{equation*}
for $t\in[0,\mathcal{T}]$. Since $\overline{\phi}$ and $\phi_\pm$ have zero mean value in $\mathbb{T}^3$, by the Poincar\'e inequality, the $L^2$ norm can be bounded by the higher order norm. Therefore, one has
\begin{equation}\label{convergence-estimate-from-main}
    \| e^{\tau(t) A} \overline{\phi}(t) \|_{H^r}^2 + \|e^{\tau(t) A} \phi_+(t) \|_{H^r}^2 + \| e^{\tau(t) A} \phi_-(t) \|_{H^r}^2  \leq \frac{ 2e^{\widetilde{K}(t)} }{|\Omega_0| -  e^{\widetilde{K}(t)}  } < \infty
\end{equation}
for $t\in[0,\mathcal{T}]$. Since $\tau(t)$ satisfies (\ref{tau-limit-system}), it follows that
\begin{equation*}
    \| e^{\tau(t) A} \overline{V}(t) \|_{H^r}^2 + \|e^{\tau(t) A} U_+(t) \|_{H^r}^2 + \| e^{\tau(t) A} U_-(t) \|_{H^r}^2 < \infty
\end{equation*}
for $t\in[0,\mathcal{T}]$. Since $\overline{v}=\overline{\phi}+ \overline{V}$ and $\widetilde{u}_\pm = \widetilde{\phi}_\pm+\widetilde{U}_\pm$, by triangle inequality, thanks to Lemma \ref{lemma-vtilde-upm}, we have
$
    \| e^{\tau(t) A} \overline{v}(t) \|_{H^r}^2 + \|e^{\tau(t) A} \widetilde{v}(t) \|_{H^r}^2 < \infty
$
for $t\in[0,\mathcal{T}]$. Therefore, we obtain $
    (\overline{v},\widetilde{v})\in L^\infty(0,\mathcal{T}; \mathcal{D}(e^{\tau(t) A}: H^{r}(\mathbb{T}^3)) ).$
This completes the proof of Theorem \ref{theorem-main}.

\subsection{Approximation by the limit resonant system}
As a consequence of the proof of Theorem \ref{theorem-main}, the following theorem describes the approximation of the solution to the original system (\ref{local-system-1})--(\ref{local-system-2}) by the solution to the limit resonant system (\ref{systemVbar-1})--(\ref{systemlimit-initial}) in the space of analytic functions, for large rotation rate $|\Omega|$ and small initial baroclinic mode in Sobolev norm.
\begin{theorem}
Suppose the conditions in Theorem \ref{theorem-main} hold, and let $(\overline{V}, \widetilde{V})$ be the solution to system (\ref{systemVbar-1})--(\ref{systemlimit-initial}) with initial data $(\overline{v}_0, \widetilde{v}_0)$. Denote by $\overline{\phi} = \overline{v}- \overline{V}$ and $\widetilde{\phi} = \widetilde{v}- \widetilde{V}$, then, for $|\Omega|\geq |\Omega_0|$, one has
\begin{eqnarray*}
\|e^{\tau(t) A} \overline{\phi}(t)\|_{H^r} + \| e^{\tau(t) A} \widetilde{\phi}(t) \|_{H^r} \lesssim \frac{e^{\widetilde{K}(t)}}{|\Omega_0| - e^{\widetilde{K}(t)}}, 
\end{eqnarray*}
for $t\in[0,\mathcal{T}]$ with $\mathcal{T}$ given by (\ref{T-longtime-type2}) and $\tau(t)$ given by \eqref{main-tau}.
\end{theorem}
\begin{proof}
The proof is an immediate consequence of \eqref{convergence-estimate-from-main}.
\end{proof}

\subsection{Remarks and discussions}
\begin{remark}
 To emphasize the difference between smallness in analytic norm and in Sobolev norm, for $|\Omega|\gg 1$, consider
$
     \widetilde{v}_0 = c_{\boldsymbol{k}} e^{i\boldsymbol{k}\cdot \boldsymbol{x}}$ 
with $k_3 \neq 0$, $|\boldsymbol{k}| = \lc \tau_0^{-1}  \ln|\Omega| \rc$ and $|c_{\boldsymbol{k}}| = (\ln |\Omega|)^{-r-2}|\Omega|^{-1}$. When $0<\delta<1$, since $r>\frac{5}{2}$, we have $\| \widetilde{v}_0\|_{H^{3+\delta}}\leq \|\widetilde{v}_0\|_{H^{r+2}} \sim |\Omega|^{-1}$, $\|e^{\tau_0 A} \widetilde{v}_0\|_{H^{r+2}} \sim 1$. Therefore, one can construct a sequence of initial data
\begin{equation*}
   \{(\widetilde{v}_0)_{\Omega}\}={c_{\boldsymbol{k}(\Omega)}}e^{i\boldsymbol{k}(\Omega)\cdot \boldsymbol{x}}, 
\end{equation*}
 where $|\boldsymbol{k}(\Omega)| = \lc \tau_0^{-1}  \ln|\Omega| \rc$ and $|c_{\boldsymbol{k}(\Omega)}| = (\ln |\Omega|)^{-r-2}|\Omega|^{-1}$. Then as $|\Omega|\rightarrow \infty$, the existence time of solutions  $\mathcal{T}\rightarrow \infty,$ with initial condition $\|e^{\tau_0 A} (\widetilde{v}_0)_\Omega\|_{H^{r+2}} \sim 1$.  This result needs fast rotation, and is very different from Theorem \ref{theorem-longtime}.
\end{remark}
\begin{remark}
In the view of Remark \ref{remark-steadyEuler}, when the solution to the $2D$ Euler equations with initial data $\overline{v}_0$ is uniformly bounded in time, the function $\widetilde{K}(t)$ appears in the proof of Theorem \ref{theorem-main} becomes only exponentially in time. This reduces two logarithms in \eqref{main-time} and one concludes that 
$\mathcal{T} \gtrsim \ln(\ln|\Omega_0|)$. Moreover, when $\overline{v}_0=0$, the function $\widetilde{K}(t)$ is uniformly bounded in time, hence $\mathcal{T} \gtrsim \ln|\Omega_0|$.
\end{remark}
\begin{remark}\label{remark-term-type1}
 In estimate (\ref{difficulty-1}) we have the resonance term \begin{equation}\label{problematic-1}
     (U_+ \cdot \nabla)(\overline{\phi}+i\overline{\phi}^\perp) = (U_+ \cdot \nabla \overline{\phi} - U^\perp_+ \cdot \nabla \overline{\phi}^\perp) = U^\perp_+ (\nabla^\perp\cdot \overline{\phi}),
 \end{equation}
 which involves the vorticity $\nabla^\perp\cdot \overline{\phi}$. Notice that in the limit resonant system (\ref{systemVbar-1})--(\ref{systemVtilde}), the evolution of the barotropic mode $\overline{V}$ is independent of the baroclinic mode $\widetilde{V}$, and therefore we can control the vorticity $\nabla^\perp\cdot \overline{V}$. However, for the original system (\ref{local-system-1})--(\ref{local-system-2})  (or the perturbed system (\ref{difference-phibar})--(\ref{difference-phipm})), the evolution of the barotropic mode $\overline{v}$ (or $\overline{\phi}$) depends also on the baroclinic mode $\widetilde{v}$ (or $\phi_\pm$). Therefore, we are unable to control (\ref{problematic-1}) without the smallness condition on the initial baroclinic mode.
\end{remark}

\begin{remark}\label{remark-term-type2}
In estimate (\ref{difficulty-3}), we have the term $
    e^{i\Omega t}(\int_0^z \nabla\cdot \phi_+ (\boldsymbol{x}',s)ds) \partial_z U_+.
$
Despite the oscillation, we are unable to apply similar methods as in type 3 due to the loss of derivative on the baroclinic mode. For this term, we do not have cancellation as other terms in type 4. Therefore, we are forced to require the smallness condition on the initial baroclinic mode.
\end{remark}

\appendix

\section{Estimates of nonlinear terms}
In this appendix, we list the estimates of nonlinear terms in the space of analytic functions. Lemma \ref{lemma-type1}--\ref{lemma-type3} will be used to prove the local well-posedness.

First, we estimate nonlinear terms of the form $f\cdot \nabla g$.
\begin{lemma}\label{lemma-type1}
For $f, g, h\in \mathcal{D}(e^{\tau A}: H^{r+\frac{1}{2}})$, where $r>2$ and $\tau\geq 0$, one has
\begin{equation*}\label{lemma-type1-inequality}
    \begin{split}
        \Big|\Big\langle A^r e^{\tau A} (f\cdot \nabla g), A^r e^{\tau A} h  \Big\rangle\Big| \leq  C_r\Big[ &(\|A^r e^{\tau A} f\| + |\hat{f}_0| ) \|A^{r+\frac{1}{2}} e^{\tau A} g\| \|A^{r+\frac{1}{2}} e^{\tau A} h\| \\
       & + \|A^{r+\frac{1}{2}} e^{\tau A} f\| \|A^{r} e^{\tau A} g\| \|A^{r} e^{\tau A} h\|\Big].
    \end{split}
\end{equation*}
\end{lemma}
\begin{proof}
    First, notice that $\Big|\Big\langle A^r e^{\tau A} (f\cdot \nabla g), A^r e^{\tau A} h \Big\rangle\Big| = \Big|\Big\langle f\cdot \nabla g, A^r e^{\tau A} H  \Big\rangle\Big| $, where $H =  A^r e^{\tau A} h$.
    We use Fourier representation of $f, g$ and $H$, in which we can write
    \begin{eqnarray*}
    &&\hskip-.8in
     f(\boldsymbol{x}) = \sum\limits_{\boldsymbol{j}\in  \mathbb{Z}^3} \hat{f}_{\boldsymbol{j}} e^{2\pi i\boldsymbol{j}\cdot \boldsymbol{x}}, \quad
    g(\boldsymbol{x}) = \sum\limits_{\boldsymbol{k}\in \mathbb{Z}^3} \hat{g}_{\boldsymbol{k}} e^{2\pi i\boldsymbol{k}\cdot \boldsymbol{x}}, \quad
    A^r e^{\tau A} H(\boldsymbol{x}) = \sum\limits_{\boldsymbol{l}\in \mathbb{Z}^3} |\boldsymbol{l}|^r e^{\tau |\boldsymbol{l}|}\hat{H}_{\boldsymbol{l}} e^{2\pi i\boldsymbol{l}\cdot \boldsymbol{x}}. 
    \end{eqnarray*}
    Therefore, 
    \begin{eqnarray*}
    \Big|\Big\langle f\cdot \nabla g, A^r e^{\tau A} H  \Big\rangle\Big| \leq  \sum\limits_{\boldsymbol{j}+\boldsymbol{k}+\boldsymbol{l}=0} |\hat{f}_{\boldsymbol{j}}||\boldsymbol{k}||\hat{g}_{\boldsymbol{k}}||\boldsymbol{l}|^r e^{\tau |\boldsymbol{l}|} |\hat{H}_{\boldsymbol{l}}|.
    \end{eqnarray*}
From $|\boldsymbol{l}| = |\boldsymbol{j}+\boldsymbol{k}| \leq |\boldsymbol{j}|+|\boldsymbol{k}|$ we have the following inequalities: 
\begin{equation*}
    |\boldsymbol{l}|^r \leq (|\boldsymbol{j}|+|\boldsymbol{k}|)^r \leq C_r(|\boldsymbol{j}|^r + |\boldsymbol{k}|^r), \;\;\;  e^{\tau |\boldsymbol{l}|} \leq e^{\tau |\boldsymbol{j}|} e^{\tau |\boldsymbol{k}|}.
\end{equation*}
Applying these inequalities, we have
\begin{eqnarray*}
\Big|\Big\langle f\cdot \nabla g, A^r e^{\tau A} H  \Big\rangle\Big| \leq  \sum\limits_{\boldsymbol{j}+\boldsymbol{k}+\boldsymbol{l}=0} C_r|\hat{f}_{\boldsymbol{j}}||\boldsymbol{k}||\hat{g}_{\boldsymbol{k}}|(|\boldsymbol{j}|^r+|\boldsymbol{k}|^r)e^{\tau |\boldsymbol{j}|}e^{\tau |\boldsymbol{k}|}|\boldsymbol{l}|^r e^{\tau |\boldsymbol{l}|}|\hat{h}_{\boldsymbol{l}}|.
\end{eqnarray*}
Since $|\boldsymbol{k}|, |\boldsymbol{j}|, |\boldsymbol{l}|$ are all nonnegative, we have $|\boldsymbol{k}|^{\frac{1}{2}} \leq (|\boldsymbol{j}|+|\boldsymbol{l}|)^{\frac{1}{2}} \leq |\boldsymbol{j}|^{\frac{1}{2}} + |\boldsymbol{l}|^{\frac{1}{2}}$, therefore, 
\begin{eqnarray*}
&&\hskip-.8in 
\Big|\Big\langle f\cdot \nabla g, A^r e^{\tau A} H  \Big\rangle\Big| \leq  \sum\limits_{\boldsymbol{j}+\boldsymbol{k}+\boldsymbol{l}=0} C_r|\hat{f}_{\boldsymbol{j}}||\boldsymbol{k}|^{\frac{1}{2}}(|\boldsymbol{j}|^{\frac{1}{2}} + |\boldsymbol{l}|^{\frac{1}{2}})|\hat{g}_{\boldsymbol{k}}|(|\boldsymbol{j}|^r+|\boldsymbol{k}|^r)e^{\tau |\boldsymbol{j}|}e^{\tau |\boldsymbol{k}|}|\boldsymbol{l}|^r e^{\tau |\boldsymbol{l}|}|\hat{h}_{\boldsymbol{l}}|  \nonumber\\
&&\hskip-.8in 
\leq  \sum\limits_{\boldsymbol{j}+\boldsymbol{k}+\boldsymbol{l}=0} C_r \Big(|\boldsymbol{k}|^{\frac{1}{2}}|\boldsymbol{j}|^{r+\frac{1}{2}} |\boldsymbol{l}|^{r} + |\boldsymbol{k}|^{r+\frac{1}{2}}|\boldsymbol{j}|^{\frac{1}{2}} |\boldsymbol{l}|^{r} + |\boldsymbol{k}|^{\frac{1}{2}}|\boldsymbol{j}|^{r} |\boldsymbol{l}|^{r+\frac{1}{2}} + |\boldsymbol{k}|^{r+\frac{1}{2}}|\boldsymbol{l}|^{r+\frac{1}{2}} \Big) \nonumber\\
&&\hskip-.08in 
\times e^{\tau |\boldsymbol{j}|}e^{\tau |\boldsymbol{k}|} e^{\tau |\boldsymbol{l}|}  |\hat{f}_{\boldsymbol{j}}| |\hat{g}_{\boldsymbol{k}}| |\hat{h}_{\boldsymbol{l}}| =: A_1 + A_2 + A_3 + A_4.
\end{eqnarray*}

Thanks to the Cauchy–Schwarz inequality, since $r>2$, we have
\begin{eqnarray*}
&&\hskip-.8in 
A_1 = \sum\limits_{\boldsymbol{j}+\boldsymbol{k}+\boldsymbol{l}=0} C_r |\boldsymbol{k}|^{\frac{1}{2}}|\boldsymbol{j}|^{r+\frac{1}{2}} |\boldsymbol{l}|^{r} e^{\tau |\boldsymbol{j}|}e^{\tau |\boldsymbol{k}|} e^{\tau |\boldsymbol{l}|}  |\hat{f}_{\boldsymbol{j}}| |\hat{g}_{\boldsymbol{k}}| |\hat{h}_{\boldsymbol{l}}| \nonumber \\
&&\hskip-.58in 
= C_r \sum\limits_{\substack{\boldsymbol{k}\in \mathbb{Z}^3 \\ \boldsymbol{k}\neq 0} } |\boldsymbol{k}|^{\frac{1}{2}} |\hat{g}_{\boldsymbol{k}}| e^{\tau |\boldsymbol{k}|} \sum\limits_{\substack{\boldsymbol{j}\in \mathbb{Z}^3 \\ \boldsymbol{j}\neq 0, -\boldsymbol{k}} } |\boldsymbol{j}|^{r+\frac{1}{2}} e^{\tau |\boldsymbol{j}|}|\hat{f}_{\boldsymbol{j}}|  |\boldsymbol{j}+\boldsymbol{k}|^{r}e^{\tau |\boldsymbol{j}+\boldsymbol{k}|}|\hat{h}_{-\boldsymbol{j}-\boldsymbol{k}}| \nonumber\\
&&\hskip-.58in 
\leq C_r \Big( \sum\limits_{\substack{\boldsymbol{k}\in \mathbb{Z}^3 \\ \boldsymbol{k}\neq 0} } |\boldsymbol{k}|^{1-2r}\Big)^{\frac{1}{2}} \Big( \sum\limits_{\substack{\boldsymbol{k}\in \mathbb{Z}^3 \\ \boldsymbol{k}\neq 0} } |\boldsymbol{k}|^{2r} e^{2\tau |\boldsymbol{k}|} |\hat{g}_{\boldsymbol{k}}|^2\Big)^{\frac{1}{2}}  \nonumber \\
&&\hskip-.38in
\times \sup\limits_{\boldsymbol{k}\in \mathbb{Z}^3}\Big( \sum\limits_{\substack{\boldsymbol{j}\in \mathbb{Z}^3 \\ \boldsymbol{j}\neq 0, -\boldsymbol{k}} } |\boldsymbol{j}|^{2r+1}e^{2\tau |\boldsymbol{j}|} |\hat{f}_{\boldsymbol{j}}|^2\Big)^{\frac{1}{2}} \Big( \sum\limits_{\substack{\boldsymbol{j}\in \mathbb{Z}^3 \\ \boldsymbol{j}\neq 0, -\boldsymbol{k}} } |\boldsymbol{j}+\boldsymbol{k}|^{2r}e^{2\tau |\boldsymbol{j}+\boldsymbol{k}|} |\hat{h}_{-\boldsymbol{j}-\boldsymbol{k}}|^2\Big)^{\frac{1}{2}} \nonumber\\
&&\hskip-.58in \leq C_r \|A^{r+\frac{1}{2}} e^{\tau A} f\| \|A^{r} e^{\tau A} g\| \|A^{r} e^{\tau A} h\|,
\end{eqnarray*}
Similarly, we have
\begin{eqnarray*}
&&\hskip-.8in 
A_2 = \sum\limits_{\boldsymbol{j}+\boldsymbol{k}+\boldsymbol{l}=0} C_r |\boldsymbol{k}|^{r+\frac{1}{2}}|\boldsymbol{j}|^{\frac{1}{2}} |\boldsymbol{l}|^{r} e^{\tau |\boldsymbol{j}|}e^{\tau |\boldsymbol{k}|} e^{\tau |\boldsymbol{l}|}  |\hat{f}_{\boldsymbol{j}}| |\hat{g}_{\boldsymbol{k}}| |\hat{h}_{\boldsymbol{l}}| \leq C_r \|A^{r} e^{\tau A} f\| \|A^{r+\frac{1}{2}} e^{\tau A} g\| \|A^{r} e^{\tau A} h\| \text{  and }
\end{eqnarray*}
\begin{eqnarray*}
&&\hskip-.8in 
A_3 = \sum\limits_{\boldsymbol{j}+\boldsymbol{k}+\boldsymbol{l}=0} C_r |\boldsymbol{k}|^{\frac{1}{2}}|\boldsymbol{j}|^{r} |\boldsymbol{l}|^{r+\frac{1}{2}} e^{\tau |\boldsymbol{j}|}e^{\tau |\boldsymbol{k}|} e^{\tau |\boldsymbol{l}|}  |\hat{f}_{\boldsymbol{j}}| |\hat{g}_{\boldsymbol{k}}| |\hat{h}_{\boldsymbol{l}}| \leq C_r \|A^{r} e^{\tau A} f\| \|A^{r} e^{\tau A} g\| \|A^{r+\frac{1}{2}} e^{\tau A} h\|.
\end{eqnarray*}
For $A_4$, thanks to the Cauchy–Schwarz inequality, since $r>2$, we have
\begin{eqnarray*}
&&\hskip-.8in 
A_4 = \sum\limits_{\boldsymbol{j}+\boldsymbol{k}+\boldsymbol{l}=0} C_r |\boldsymbol{k}|^{r+\frac{1}{2}}|\boldsymbol{l}|^{r+\frac{1}{2}} e^{\tau |\boldsymbol{j}|}e^{\tau |\boldsymbol{k}|} e^{\tau |\boldsymbol{l}|}  |\hat{f}_{\boldsymbol{j}}| |\hat{g}_{\boldsymbol{k}}| |\hat{h}_{\boldsymbol{l}}| \nonumber \\
&&\hskip-.58in 
= C_r \sum\limits_{\boldsymbol{j}\in \mathbb{Z}^3 } e^{\tau |\boldsymbol{j}|}|\hat{f}_{\boldsymbol{j}}|  \sum\limits_{\substack{\boldsymbol{k}\in \mathbb{Z}^3 \\ \boldsymbol{k}\neq 0, -\boldsymbol{j}} } |\boldsymbol{k}|^{r+\frac{1}{2}} |\hat{g}_{\boldsymbol{k}}| e^{\tau |\boldsymbol{k}|}  |\boldsymbol{j}+\boldsymbol{k}|^{r+\frac{1}{2}}e^{\tau |\boldsymbol{j}+\boldsymbol{k}|}|\hat{h}_{-\boldsymbol{j}-\boldsymbol{k}}| \nonumber\\
&&\hskip-.58in 
\leq C_r \Big\{ |\hat{f}_0| + \Big( \sum\limits_{\substack{\boldsymbol{j}\in \mathbb{Z}^3 \\ \boldsymbol{j}\neq 0} } |\boldsymbol{j}|^{-2r}\Big)^{\frac{1}{2}} \Big( \sum\limits_{\substack{\boldsymbol{j}\in \mathbb{Z}^3 \\ \boldsymbol{j}\neq 0} } |\boldsymbol{j}|^{2r} e^{2\tau |\boldsymbol{j}|} |\hat{f}_{\boldsymbol{j}}|^2\Big)^{\frac{1}{2}} \Big\}  \nonumber \\
&&\hskip-.38in
\times \sup\limits_{\boldsymbol{j}\in \mathbb{Z}^3} \Big( \sum\limits_{\substack{\boldsymbol{k}\in \mathbb{Z}^3 \\ \boldsymbol{k}\neq 0, -\boldsymbol{j}} } |\boldsymbol{k}|^{2r+1}e^{2\tau |\boldsymbol{k}|} |\hat{g}_{\boldsymbol{k}}|^2\Big)^{\frac{1}{2}} \Big( \sum\limits_{\substack{\boldsymbol{k}\in \mathbb{Z}^3 \\ \boldsymbol{k}\neq 0, -\boldsymbol{j}} } |\boldsymbol{j}+\boldsymbol{k}|^{2r+1}e^{2\tau |\boldsymbol{j}+\boldsymbol{k}|} |\hat{h}_{-\boldsymbol{j}-\boldsymbol{k}}|^2\Big)^{\frac{1}{2}} \nonumber\\
&&\hskip-.58in \leq C_r (\|A^{r} e^{\tau A} f\| + |\hat{f}_0|) \|A^{r+\frac{1}{2}} e^{\tau A} g\| \|A^{r+\frac{1}{2}} e^{\tau A} h\|.
\end{eqnarray*}

Combine the estimates for $A_1$ to $A_4$, and since $\|A^{r} e^{\tau A} g\| \leq \|A^{r+\frac{1}{2}} e^{\tau A} g\|$, $\|A^{r} e^{\tau A} h\| \leq \|A^{r+\frac{1}{2}} e^{\tau A} h\|$, we achieve the desired inequality.
\end{proof}
Similarly, we estimate $(\nabla\cdot f)g$ in the following:
\begin{lemma}\label{lemma-type2}
For $f, g, h\in \mathcal{D}(e^{\tau A}: H^{r+\frac{1}{2}})$, where $r>2$ and $\tau\geq 0$, one has
\begin{equation*}
    \begin{split}
        \Big|\Big\langle A^r e^{\tau A} \big((\nabla\cdot f)g\big), A^r e^{\tau A} h  \Big\rangle\Big| \leq & C_r\Big[ (\|A^r e^{\tau A} g\| + |\hat{g}_0|) \|A^{r+\frac{1}{2}} e^{\tau A} f\| \|A^{r+\frac{1}{2}} e^{\tau A} h\| \\
        & + \|A^{r+\frac{1}{2}} e^{\tau A} g\| \|A^{r} e^{\tau A} f\| \|A^{r} e^{\tau A} h\|\Big]. 
    \end{split}
\end{equation*}
\end{lemma}
The proof of Lemma \ref{lemma-type2} is almost the same as Lemma \ref{lemma-type1}, so we omit it. 

Finally, we provide an estimate for $(\int_0^z \nabla\cdot f(\boldsymbol{x}',s)ds)\partial_z g$ in the following:
\begin{lemma}\label{lemma-type3}
For $f, g, h\in \mathcal{D}(e^{\tau A}: H^{r+\frac{1}{2}})$, where $r>2$, $\tau\geq 0$, and $\overline{f}=0$, one has
\begin{equation*}
\begin{split}
    \Big|\Big\langle A^r e^{\tau A} \big((\int_0^z \nabla\cdot f(\boldsymbol{x}',s)ds)\partial_z g\big), A^r e^{\tau A} h  \Big\rangle\Big| \leq  C_r\Big( \|A^r e^{\tau A} f\| \|A^{r+\frac{1}{2}} e^{\tau A} g\| \|A^{r+\frac{1}{2}} e^{\tau A} h\| \\
     + \|A^r e^{\tau A} g\| \|A^{r+\frac{1}{2}} e^{\tau A} f\| \|A^{r+\frac{1}{2}} e^{\tau A} h\| + \|A^r e^{\tau A} h\| \|A^{r+\frac{1}{2}} e^{\tau A} f\| \|A^{r+\frac{1}{2}} e^{\tau A} g\|\Big).
\end{split}
\end{equation*}
\end{lemma}
\begin{proof}
First, $\Big|\Big\langle A^r e^{\tau A} \big((\int_0^z \nabla\cdot f(\boldsymbol{x}',s)ds)\partial_z g\big), A^r e^{\tau A} h  \Big\rangle\Big| = \Big|\Big\langle (\int_0^z \nabla\cdot f(\boldsymbol{x}',s)ds)\partial_z g, A^r e^{\tau A} H  \Big\rangle\Big| $. Since $\overline{f}=0$, one has
$
    f(\boldsymbol{x}) = \sum\limits_{\substack{\boldsymbol{j}\in \mathbb{Z}^3 \\ j_3 \neq 0}} \hat{f}_{\boldsymbol{j}} e^{2\pi( i\boldsymbol{j}' \cdot \boldsymbol{x}' +  ij_3 z)} 
$
where $\boldsymbol{j}'=(j_1,j_2)$. Then we have
\begin{equation*}
  \int_0^z \nabla \cdot f(\boldsymbol{x}',s) ds = \sum\limits_{\substack{\boldsymbol{j}\in \mathbb{Z}^3 \\ j_3 \neq 0, \boldsymbol{j}' \neq 0}} \frac{1}{j_3}\boldsymbol{j}'\cdot\hat{f}_{\boldsymbol{j}} e^{2\pi( i\boldsymbol{j}' \cdot \boldsymbol{x}' +  ij_3 z)} - \sum\limits_{\substack{\boldsymbol{j}\in \mathbb{Z}^3 \\ j_3 \neq 0, \boldsymbol{j}' \neq 0}} \frac{1}{j_3}\boldsymbol{j}'\cdot\hat{f}_{\boldsymbol{j}} e^{2\pi i\boldsymbol{j}'\cdot \boldsymbol{x}'}.  
\end{equation*}
Therefore, 
\begin{eqnarray*}
&&\hskip-.8in
\Big|\Big\langle (\int_0^z \nabla\cdot f(s)ds)\partial_z g, A^r e^{\tau A} H  \Big\rangle\Big| \leq \Big|\Big\langle (\sum\limits_{\substack{\boldsymbol{j}\in \mathbb{Z}^3 \\ j_3 \neq 0, \boldsymbol{j}' \neq 0}} \frac{1}{j_3}\boldsymbol{j}'\cdot\hat{f}_{\boldsymbol{j}} e^{2\pi( i\boldsymbol{j}' \cdot \boldsymbol{x}' +  ij_3 z)})\partial_z g, A^r e^{\tau A} H  \Big\rangle\Big| \nonumber\\
&&\hskip.8in
+ \Big|\Big\langle (\sum\limits_{\substack{\boldsymbol{j}\in \mathbb{Z}^3 \\ j_3 \neq 0, \boldsymbol{j}' \neq 0}} \frac{1}{j_3}\boldsymbol{j}'\cdot\hat{f}_{\boldsymbol{j}} e^{i\boldsymbol{j}'\cdot \boldsymbol{x}'})\partial_z g, A^r e^{\tau A} H  \Big\rangle\Big| =: I_1 + I_2 .
\end{eqnarray*}

Let us estimate $I_2$ first. For $\boldsymbol{l} = (\boldsymbol{l}', l_3) = (-\boldsymbol{j}'-\boldsymbol{k}', -k_3)$, by using the inequalities 
\begin{equation*}
    |\boldsymbol{j}'|^{\frac{1}{2}} \leq C(|\boldsymbol{k}|^{\frac{1}{2}} + |\boldsymbol{l}|^{\frac{1}{2}}), \;\;\; |\boldsymbol{k}|^{\frac{1}{2}} \leq C(|\boldsymbol{j}'|^{\frac{1}{2}} + |\boldsymbol{l}|^{\frac{1}{2}}), \;\;\;
    |\boldsymbol{l}|^r  \leq C_r (|\boldsymbol{j}'|^r+|\boldsymbol{k}|^r),
\end{equation*}
one has
\begin{eqnarray*}
&&\hskip-.8in 
I_2 \leq \sum\limits_{\substack{\boldsymbol{j}'+\boldsymbol{k}'+\boldsymbol{l}'=0\\  k_3+l_3 = 0 \\  j_3, k_3, \boldsymbol{j}' \neq 0}} C_r\frac{1}{|j_3|}|\boldsymbol{j}'||k_3||\hat{f}_{\boldsymbol{j}}||\hat{g}_{\boldsymbol{k}}|(|\boldsymbol{j}'|^r+|\boldsymbol{k}|^r)e^{\tau |\boldsymbol{j}'|}e^{\tau |\boldsymbol{k}|}|\boldsymbol{l}|^r e^{\tau |\boldsymbol{l}|}|\hat{h}_{\boldsymbol{l}}|\nonumber\\
&&\hskip-.68in 
\leq \sum\limits_{\substack{\boldsymbol{j}'+\boldsymbol{k}'+\boldsymbol{l}'=0\\  k_3+l_3 = 0 \\  j_3, k_3, \boldsymbol{j}' \neq 0}} C_r\frac{1}{|j_3|}|\hat{f}_{\boldsymbol{j}}||\hat{g}_{\boldsymbol{k}}|(|\boldsymbol{j}'|^{r+1}|\boldsymbol{k}|+|\boldsymbol{j}'||\boldsymbol{k}|^{r+1})e^{\tau |\boldsymbol{j}|}e^{\tau |\boldsymbol{k}|}|\boldsymbol{l}|^r e^{\tau |\boldsymbol{l}|}|\hat{h}_{\boldsymbol{l}}| \nonumber\\
&&\hskip-.68in 
\leq  \sum\limits_{\substack{\boldsymbol{j}'+\boldsymbol{k}'+\boldsymbol{l}'=0\\  k_3+l_3 = 0 \\  j_3, k_3, \boldsymbol{j}' \neq 0}} C_r \frac{1}{|j_3|}\Big(|\boldsymbol{k}|^{\frac{3}{2}}|\boldsymbol{j}'|^{r+\frac{1}{2}} |\boldsymbol{l}|^{r} + |\boldsymbol{k}||\boldsymbol{j}'|^{r+\frac{1}{2}} |\boldsymbol{l}|^{r+\frac{1}{2}} + |\boldsymbol{j}'|^{\frac{3}{2}}|\boldsymbol{k}|^{r+\frac{1}{2}} |\boldsymbol{l}|^{r}\nonumber\\
&&\hskip.38in 
 + |\boldsymbol{j}'||\boldsymbol{k}|^{r+\frac{1}{2}}|\boldsymbol{l}|^{r+\frac{1}{2}} \Big) e^{\tau |\boldsymbol{j}|}e^{\tau |\boldsymbol{k}|} e^{\tau |\boldsymbol{l}|}  |\hat{f}_{\boldsymbol{j}}| |\hat{g}_{\boldsymbol{k}}| |\hat{h}_{\boldsymbol{l}}| =: B_1 + B_2 + B_3 + B_4.
\end{eqnarray*}

When $k_3 \neq 0$ and $r>2$, we know that $|\boldsymbol{k}|^{1-r} \leq |(\boldsymbol{k}',\pm 1)|^{1-r}$ and $\sum\limits_{\boldsymbol{k}'\in \mathbb{Z}^2} |(\boldsymbol{k}',\pm 1)|^{2-2r} \leq C_r$ is finite. Thanks to the Cauchy–Schwarz inequality, we have
\begin{eqnarray*}
&&\hskip-.8in 
B_1 = \sum\limits_{\substack{\boldsymbol{j}'+\boldsymbol{k}'+\boldsymbol{l}'=0\\  k_3+l_3 = 0 \\  j_3, k_3, \boldsymbol{j}' \neq 0}}  C_r \frac{1}{|j_3|} |\boldsymbol{k}|^{\frac{3}{2}}|\boldsymbol{j}'|^{r+\frac{1}{2}} |\boldsymbol{l}|^{r} e^{\tau |\boldsymbol{j}|}e^{\tau |\boldsymbol{k}|} e^{\tau |\boldsymbol{l}|}  |\hat{f}_{\boldsymbol{j}}| |\hat{g}_{\boldsymbol{k}}| |\hat{h}_{\boldsymbol{l}}| \nonumber \\
&&\hskip-.58in 
= C_r \sum\limits_{\substack{\boldsymbol{k}\in \mathbb{Z}^3 \\ k_3\neq 0} } |\boldsymbol{k}|^{\frac{3}{2}} |\hat{g}_{\boldsymbol{k}}| e^{\tau |\boldsymbol{k}|} \sum\limits_{\substack{\boldsymbol{j}\in \mathbb{Z}^3 \\ j_3, \boldsymbol{j}' \neq 0} } \frac{1}{|j_3|}  |\boldsymbol{j}'|^{r+\frac{1}{2}} e^{\tau |\boldsymbol{j}|}|\hat{f}_{\boldsymbol{j}}|  |(\boldsymbol{j}'+\boldsymbol{k}',k_3)|^{r}e^{\tau |(\boldsymbol{j}'+\boldsymbol{k}',k_3)|}|\hat{h}_{-(\boldsymbol{j}'+\boldsymbol{k}',k_3)}| \nonumber\\
&&\hskip-.58in 
\leq C_r \sum\limits_{\boldsymbol{k}'\in \mathbb{Z}^2} |(\boldsymbol{k}',\pm 1)|^{1-r}  \sum\limits_{k_3\neq 0}  |\boldsymbol{k}|^{r+\frac{1}{2}} |\hat{g}_{\boldsymbol{k}}| e^{\tau |\boldsymbol{k}|} \Big( \sum\limits_{\substack{\boldsymbol{j}\in \mathbb{Z}^3 \\ \boldsymbol{j}\neq 0} } |\boldsymbol{j}|^{2r+1} e^{2\tau |\boldsymbol{j}|} |\hat{f}_{\boldsymbol{j}}|^2\Big)^{\frac{1}{2}}  \nonumber \\
&&\hskip-.38in
\times \Big( \sum\limits_{j_3\neq 0} \frac{1}{|j_3|^2} \sum\limits_{\boldsymbol{j}'\in \mathbb{Z}^2 }  |(\boldsymbol{j}'+\boldsymbol{k}',k_3)|^{2r}e^{2\tau |(\boldsymbol{j}'+\boldsymbol{k}',k_3)|}|\hat{h}_{-(\boldsymbol{j}'+\boldsymbol{k}',k_3)}|^2\Big)^{\frac{1}{2}}  \nonumber\\
&&\hskip-.58in \leq C_r \|A^{r+\frac{1}{2}} e^{\tau A} f\| \sum\limits_{\boldsymbol{k}'\in \mathbb{Z}^2} |(\boldsymbol{k}',\pm 1)|^{1-r} \Big(\sum\limits_{k_3\neq 0}  |\boldsymbol{k}|^{2r+1} |\hat{g}_{\boldsymbol{k}}|^2 e^{2\tau |\boldsymbol{k}|} \Big)^{\frac{1}{2}}  \nonumber \\
&&\hskip-.38in
\times \Big( \sum\limits_{k_3\neq 0} \sum\limits_{\boldsymbol{j}'\in \mathbb{Z}^2 }  |(\boldsymbol{j}'+\boldsymbol{k}',k_3)|^{2r}e^{2\tau |(\boldsymbol{j}'+\boldsymbol{k}',k_3)|}|\hat{h}_{-(\boldsymbol{j}'+\boldsymbol{k}',k_3)}|^2\Big)^{\frac{1}{2}}  \nonumber\\
&&\hskip-.58in \leq C_r \|A^{r+\frac{1}{2}} e^{\tau A} f\| \|A^{r} e^{\tau A} h\| \Big(\sum\limits_{\boldsymbol{k}'\in \mathbb{Z}^2} |(\boldsymbol{k}',\pm 1)|^{2-2r}\Big)^{\frac{1}{2}} \Big( \sum\limits_{\boldsymbol{k}'\in \mathbb{Z}^2} \sum\limits_{k_3\neq 0}    |\boldsymbol{k}|^{2r+1} |\hat{g}_{\boldsymbol{k}}|^2 e^{2\tau |\boldsymbol{k}|} \Big)^{\frac{1}{2}}  \nonumber \\
&&\hskip-.58in
\leq C_r \|A^{r+\frac{1}{2}} e^{\tau A} f\|  \|A^{r+\frac{1}{2}} e^{\tau A} g\| \|A^{r} e^{\tau A} h\|.
\end{eqnarray*}
The estimate for $B_2$ is similar as $B_1$, and we can get $B_2 \leq C_r \|A^{r+\frac{1}{2}} e^{\tau A} f\|  \|A^{r} e^{\tau A} g\| \|A^{r+\frac{1}{2}} e^{\tau A} h\|$. For $B_3$, thanks to the Cauchy–Schwarz inequality, since $r>2$, we have
\begin{eqnarray*}
&&\hskip-.8in 
B_3 = \sum\limits_{\substack{\boldsymbol{j}'+\boldsymbol{k}'+\boldsymbol{l}'=0\\  k_3+l_3 = 0 \\  j_3, k_3, \boldsymbol{j}' \neq 0}}  C_r \frac{1}{|j_3|} |\boldsymbol{j}'|^{\frac{3}{2}}|\boldsymbol{k}|^{r+\frac{1}{2}} |\boldsymbol{l}|^{r} e^{\tau |\boldsymbol{j}|}e^{\tau |\boldsymbol{k}|} e^{\tau |\boldsymbol{l}|}  |\hat{f}_{\boldsymbol{j}}| |\hat{g}_{\boldsymbol{k}}| |\hat{h}_{\boldsymbol{l}}| \nonumber \\
&&\hskip-.58in 
= C_r \sum\limits_{\substack{\boldsymbol{j}\in \mathbb{Z}^3 \\ j_3, \boldsymbol{j}'\neq 0} }  \frac{1}{|j_3|}|\boldsymbol{j}'|^{\frac{3}{2}} |\hat{f}_{\boldsymbol{j}}| e^{\tau |\boldsymbol{j}|} \sum\limits_{\substack{\boldsymbol{k}\in \mathbb{Z}^3 \\ k_3 \neq 0} }  |\boldsymbol{k}|^{r+\frac{1}{2}} e^{\tau |\boldsymbol{k}|}|\hat{g}_{\boldsymbol{k}}|  |(\boldsymbol{j}'+\boldsymbol{k}',k_3)|^{r}e^{\tau |(\boldsymbol{j}'+\boldsymbol{k}',k_3)|}|\hat{h}_{-(\boldsymbol{j}'+\boldsymbol{k}',k_3)}| \nonumber\\
&&\hskip-.58in 
\leq C_r \Big(\sum\limits_{\substack{\boldsymbol{j}\in \mathbb{Z}^3 \\ j_3, \boldsymbol{j}'\neq 0} }  \frac{1}{|j_3|^2}|\boldsymbol{j}'|^{2-2r} \Big)^{\frac{1}{2}} \Big(\sum\limits_{\substack{\boldsymbol{j}\in \mathbb{Z}^3 \\ j_3, \boldsymbol{j}'\neq 0} }   |\boldsymbol{j}|^{2r+1} |\hat{f}_{\boldsymbol{j}}|^2 e^{2\tau |\boldsymbol{j}|} \Big)^{\frac{1}{2}} 
\Big( \sum\limits_{\substack{\boldsymbol{k}\in \mathbb{Z}^3 \\ k_3\neq 0} } |\boldsymbol{k}|^{2r+1} e^{2\tau |\boldsymbol{k}|} |\hat{g}_{\boldsymbol{k}}|^2\Big)^{\frac{1}{2}}  \nonumber \\
&&\hskip-.38in
\times \sup\limits_{\boldsymbol{j}\in \mathbb{Z}^3} \Big( \sum\limits_{\substack{\boldsymbol{k}\in \mathbb{Z}^3 \\ k_3\neq 0}}  |(\boldsymbol{j}'+\boldsymbol{k}',k_3)|^{2r}e^{2\tau |(\boldsymbol{j}'+\boldsymbol{k}',k_3)|}|\hat{h}_{-(\boldsymbol{j}'+\boldsymbol{k}',k_3)}|^2\Big)^{\frac{1}{2}}  \nonumber\\
&&\hskip-.58in
\leq C_r \|A^{r+\frac{1}{2}} e^{\tau A} f\|  \|A^{r+\frac{1}{2}} e^{\tau A} g\| \|A^{r} e^{\tau A} h\|.
\end{eqnarray*}
The estimate for $B_4$ is similar as $B_3$, and we can get $B_4 \leq C_r \|A^{r} e^{\tau A} f\|  \|A^{r+\frac{1}{2}} e^{\tau A} g\| \|A^{r+\frac{1}{2}} e^{\tau A} h\|$. The estimates of $B_1$ to $B_4$ indicate that $I_2$ satisfies the desired inequality. 

Now let us estimate on $I_1$.  For $  \boldsymbol{j}+\boldsymbol{k}+ \boldsymbol{l}=0$, by using the inequalities 
\begin{equation*}
    |\boldsymbol{j}|^{\frac{1}{2}} \leq C(|\boldsymbol{k}|^{\frac{1}{2}} + |\boldsymbol{l}|^{\frac{1}{2}}), \;\;\; |\boldsymbol{k}|^{\frac{1}{2}} \leq C(|\boldsymbol{j}|^{\frac{1}{2}} + |\boldsymbol{l}|^{\frac{1}{2}}), \;\;\;
    |\boldsymbol{l}|^r  \leq C_r (|\boldsymbol{j}|^r+|\boldsymbol{k}|^r),
\end{equation*}
we have
\begin{eqnarray*}
&&\hskip-.8in 
I_1 \leq \sum\limits_{\substack{\boldsymbol{j}+\boldsymbol{k}+\boldsymbol{l}=0\\    j_3, k_3, \boldsymbol{j}' \neq 0}} C_r\frac{1}{|j_3|}|\boldsymbol{j}'||k_3||\hat{f}_{\boldsymbol{j}}||\hat{g}_{\boldsymbol{k}}|(|\boldsymbol{j}|^r+|\boldsymbol{k}|^r)e^{\tau |\boldsymbol{j}|}e^{\tau |\boldsymbol{k}|}|\boldsymbol{l}|^r e^{\tau |\boldsymbol{l}|}|\hat{h}_{\boldsymbol{l}}|\nonumber\\
&&\hskip-.68in 
\leq \sum\limits_{\substack{\boldsymbol{j}+\boldsymbol{k}+\boldsymbol{l}=0\\    j_3, k_3, \boldsymbol{j}' \neq 0}} C_r\frac{1}{|j_3|}|\hat{f}_{\boldsymbol{j}}||\hat{g}_{\boldsymbol{k}}|(|\boldsymbol{j}|^{r+1}|\boldsymbol{k}|+|\boldsymbol{j}||\boldsymbol{k}|^{r+1})e^{\tau |\boldsymbol{j}|}e^{\tau |\boldsymbol{k}|}|\boldsymbol{l}|^r e^{\tau |\boldsymbol{l}|}|\hat{h}_{\boldsymbol{l}}| \nonumber\\
&&\hskip-.68in 
\leq  \sum\limits_{\substack{\boldsymbol{j}+\boldsymbol{k}+\boldsymbol{l}=0\\    j_3, k_3, \boldsymbol{j}' \neq 0}} C_r \frac{1}{|j_3|}\Big(|\boldsymbol{k}|^{\frac{3}{2}}|\boldsymbol{j}|^{r+\frac{1}{2}} |\boldsymbol{l}|^{r} + |\boldsymbol{k}||\boldsymbol{j}|^{r+\frac{1}{2}} |\boldsymbol{l}|^{r+\frac{1}{2}} + |\boldsymbol{j}|^{\frac{3}{2}}|\boldsymbol{k}|^{r+\frac{1}{2}} |\boldsymbol{l}|^{r} \nonumber\\
&&\hskip.48in 
  + |\boldsymbol{j}||\boldsymbol{k}|^{r+\frac{1}{2}}|\boldsymbol{l}|^{r+\frac{1}{2}} \Big) e^{\tau |\boldsymbol{j}|}e^{\tau |\boldsymbol{k}|} e^{\tau |\boldsymbol{l}|}  |\hat{f}_{\boldsymbol{j}}| |\hat{g}_{\boldsymbol{k}}| |\hat{h}_{\boldsymbol{l}}| =: \widetilde{B}_1 + \widetilde{B}_2 + \widetilde{B}_3 + \widetilde{B}_4.
\end{eqnarray*}
Thanks to the Cauchy–Schwarz inequality, since $r>2$, we have
\begin{eqnarray*}
&&\hskip-.8in 
\widetilde{B}_1 = \sum\limits_{\substack{\boldsymbol{j}+\boldsymbol{k}+\boldsymbol{l}=0\\  j_3, k_3, \boldsymbol{j}' \neq 0}}  C_r \frac{1}{|j_3|} |\boldsymbol{k}|^{\frac{3}{2}}|\boldsymbol{j}|^{r+\frac{1}{2}} |\boldsymbol{l}|^{r} e^{\tau |\boldsymbol{j}|}e^{\tau |\boldsymbol{k}|} e^{\tau |\boldsymbol{l}|}  |\hat{f}_{\boldsymbol{j}}| |\hat{g}_{\boldsymbol{k}}| |\hat{h}_{\boldsymbol{l}}| \nonumber \\
&&\hskip-.58in 
= C_r \sum\limits_{\substack{\boldsymbol{k}\in \mathbb{Z}^3 \\ k_3\neq 0} } |\boldsymbol{k}|^{\frac{3}{2}} |\hat{g}_{\boldsymbol{k}}| e^{\tau |\boldsymbol{k}|} \sum\limits_{\substack{\boldsymbol{j}\in \mathbb{Z}^3 \\ j_3, \boldsymbol{j}' \neq 0} } \frac{1}{|j_3|}  |\boldsymbol{j}|^{r+\frac{1}{2}} e^{\tau |\boldsymbol{j}|}|\hat{f}_{\boldsymbol{j}}|  |\boldsymbol{j}+\boldsymbol{k}|^{r}e^{\tau |\boldsymbol{j}+\boldsymbol{k}|}|\hat{h}_{-\boldsymbol{j}-\boldsymbol{k}}| \nonumber\\
&&\hskip-.58in 
\leq C_r \sum\limits_{\boldsymbol{k}'\in \mathbb{Z}^2} |(\boldsymbol{k}',\pm 1)|^{1-r}  \sum\limits_{k_3\neq 0}  |\boldsymbol{k}|^{r+\frac{1}{2}} |\hat{g}_{\boldsymbol{k}}| e^{\tau |\boldsymbol{k}|} \Big( \sum\limits_{\substack{\boldsymbol{j}\in \mathbb{Z}^3 \\ \boldsymbol{j}\neq 0} } |\boldsymbol{j}|^{2r+1} e^{2\tau |\boldsymbol{j}|} |\hat{f}_{\boldsymbol{j}}|^2\Big)^{\frac{1}{2}}  \nonumber \\
&&\hskip-.38in
\times \Big( \sum\limits_{j_3\neq 0} \frac{1}{|j_3|^2} \sum\limits_{\boldsymbol{j}'\in \mathbb{Z}^2 }  |(\boldsymbol{j}'+\boldsymbol{k}',j_3+k_3)|^{2r}e^{2\tau |(\boldsymbol{j}'+\boldsymbol{k}',j_3+k_3)|}|\hat{h}_{-(\boldsymbol{j}'+\boldsymbol{k}',j_3+k_3)}|^2\Big)^{\frac{1}{2}}  \nonumber\\
&&\hskip-.58in  \leq C_r \|A^{r+\frac{1}{2}} e^{\tau A} f\| \sum\limits_{\boldsymbol{k}'\in \mathbb{Z}^2} |(\boldsymbol{k}',\pm 1)|^{1-r} \Big(\sum\limits_{k_3\neq 0}  |\boldsymbol{k}|^{2r+1} |\hat{g}_{\boldsymbol{k}}|^2 e^{2\tau |\boldsymbol{k}|} \Big)^{\frac{1}{2}}  \nonumber \\
&&\hskip-.38in
\times \Big( \sum\limits_{j_3\neq 0} \frac{1}{|j_3|^2} \sum\limits_{k_3\neq 0} \sum\limits_{\boldsymbol{j}'\in \mathbb{Z}^2 }  |(\boldsymbol{j}'+\boldsymbol{k}',j_3+k_3)|^{2r}e^{2\tau |(\boldsymbol{j}'+\boldsymbol{k}',j_3+k_3)|}|\hat{h}_{-(\boldsymbol{j}'+\boldsymbol{k}',j_3+k_3)}|^2\Big)^{\frac{1}{2}} \nonumber \\
&&\hskip-.58in \leq C_r \|A^{r+\frac{1}{2}} e^{\tau A} f\| \|A^{r} e^{\tau A} h\| \Big(\sum\limits_{\boldsymbol{k}'\in \mathbb{Z}^2} |(\boldsymbol{k}',\pm 1)|^{2-2r}\Big)^{\frac{1}{2}} \Big(\sum\limits_{\boldsymbol{k}'\in \mathbb{Z}^2} \sum\limits_{k_3\neq 0}  |\boldsymbol{k}|^{2r+1} |\hat{g}_{\boldsymbol{k}}|^2 e^{2\tau |\boldsymbol{k}|} \Big)^{\frac{1}{2}}  \nonumber \\
&&\hskip-.58in
\leq C_r \|A^{r+\frac{1}{2}} e^{\tau A} f\|  \|A^{r+\frac{1}{2}} e^{\tau A} g\| \|A^{r} e^{\tau A} h\|,
\end{eqnarray*}
where in the second inequality, we use Fubini theorem to exchange the order of $\sum\limits_{j_3\neq 0}$ and $\sum\limits_{k_3\neq 0}$. The estimate for $\widetilde{B}_2$ is similar to $\widetilde{B}_1$, and we can get $\widetilde{B}_2 \leq C_r \|A^{r+\frac{1}{2}} e^{\tau A} f\|  \|A^{r} e^{\tau A} g\| \|A^{r+\frac{1}{2}} e^{\tau A} h\| $. 
For $\widetilde{B}_3$, thanks to the Cauchy–Schwarz inequality, since $r>2$, we have
\begin{eqnarray*}
&&\hskip-.8in 
\widetilde{B}_3 = \sum\limits_{\substack{\boldsymbol{j}+\boldsymbol{k}+\boldsymbol{l}=0\\  j_3, k_3, \boldsymbol{j}' \neq 0}}  C_r \frac{1}{|j_3|} |\boldsymbol{j}|^{\frac{3}{2}}|\boldsymbol{k}|^{r+\frac{1}{2}} |\boldsymbol{l}|^{r} e^{\tau |\boldsymbol{j}|}e^{\tau |\boldsymbol{k}|} e^{\tau |\boldsymbol{l}|}  |\hat{f}_{\boldsymbol{j}}| |\hat{g}_{\boldsymbol{k}}| |\hat{h}_{\boldsymbol{l}}| \nonumber \\
&&\hskip-.58in 
= C_r \sum\limits_{\substack{\boldsymbol{j}\in \mathbb{Z}^3 \\ j_3, \boldsymbol{j}' \neq 0} } \frac{1}{|j_3|}|\boldsymbol{j}|^{\frac{3}{2}} e^{\tau |\boldsymbol{j}|}|\hat{f}_{\boldsymbol{j}}|  \sum\limits_{\substack{\boldsymbol{k}\in \mathbb{Z}^3 \\ k_3\neq 0} } |\boldsymbol{k}|^{r+\frac{1}{2}} |\hat{g}_{\boldsymbol{k}}| e^{\tau |\boldsymbol{k}|}    |\boldsymbol{j}+\boldsymbol{k}|^{r}e^{\tau |\boldsymbol{j}+\boldsymbol{k}|}|\hat{h}_{-\boldsymbol{j}-\boldsymbol{k}}| \nonumber\\
&&\hskip-.58in 
\leq C_r \Big(\sum\limits_{\substack{\boldsymbol{j}\in \mathbb{Z}^3 \\ j_3, \boldsymbol{j}'\neq 0} }  \frac{1}{|j_3|^2}|\boldsymbol{j}'|^{2-2r} \Big)^{\frac{1}{2}} \Big(\sum\limits_{\substack{\boldsymbol{j}\in \mathbb{Z}^3 \\ j_3, \boldsymbol{j}'\neq 0} }   |\boldsymbol{j}|^{2r+1} |\hat{f}_{\boldsymbol{j}}|^2 e^{2\tau |\boldsymbol{j}|} \Big)^{\frac{1}{2}} 
\Big( \sum\limits_{\substack{\boldsymbol{k}\in \mathbb{Z}^3 \\ k_3\neq 0} } |\boldsymbol{k}|^{2r+1} e^{2\tau |\boldsymbol{k}|} |\hat{g}_{\boldsymbol{k}}|^2\Big)^{\frac{1}{2}}  \nonumber \\
&&\hskip-.38in
\times \sup\limits_{\boldsymbol{j}\in \mathbb{Z}^3} \Big( \sum\limits_{\substack{\boldsymbol{k}\in \mathbb{Z}^3 \\ k_3\neq 0}}  |\boldsymbol{j}+\boldsymbol{k}|^{2r}e^{2\tau |\boldsymbol{j}+\boldsymbol{k}|}|\hat{h}_{-\boldsymbol{j}-\boldsymbol{k}}|^2\Big)^{\frac{1}{2}}  
\leq C_r \|A^{r+\frac{1}{2}} e^{\tau A} f\|  \|A^{r+\frac{1}{2}} e^{\tau A} g\| \|A^{r} e^{\tau A} h\|,
\end{eqnarray*}
where in the first inequality we use $|\boldsymbol{j}|^{2-2r}\leq|\boldsymbol{j}'|^{2-2r}$ due to $r>2$. The estimate for $\widetilde{B}_4$ is similar as $\widetilde{B}_3$, and we can get $\widetilde{B}_4 \leq C_r \|A^{r} e^{\tau A} f\|  \|A^{r+\frac{1}{2}} e^{\tau A} g\| \|A^{r+\frac{1}{2}} e^{\tau A} h\|$. The estimates of $\widetilde{B}_1$ to $\widetilde{B}_4$ indicate that $I_1$ satisfies the desired inequality. 
\end{proof}

Lemma \ref{lemma-difference-type1}--\ref{lemma-difference-type4} play an essential role in the proof of Theorem \ref{theorem-main}. First, let us state the following:
\begin{lemma} \label{lemma-difference-type1}
For $f, g, h\in \mathcal{D}(e^{\tau A}: H^{r+\frac{1}{2}})$, where $r>\frac{5}{2}$ and $\tau\geq 0$, one has
\begin{eqnarray*}
&&\hskip-.8in
\Big|\Big\langle A^r e^{\tau A} (f\cdot \nabla g), A^r e^{\tau A} h  \Big\rangle - \Big\langle  f\cdot \nabla A^r e^{\tau A} g, A^r e^{\tau A} h  \Big\rangle\Big| \nonumber\\
&&\hskip-.9in
\leq C_r \|A^r f\| \|A^{r}  g\| \|A^{r}  h\| + C_r \tau \|A^{r+\frac{1}{2}} e^{\tau A} f\| \|A^{r+\frac{1}{2}} e^{\tau A} g\| \|A^{r+\frac{1}{2}} e^{\tau A} h\|.  
\end{eqnarray*}
\end{lemma}
\begin{lemma} \label{lemma-difference-type2}
For $f, g, h\in \mathcal{D}(e^{\tau A}: H^{r+\frac{1}{2}})$, where $r>\frac{5}{2}$ and $\tau\geq 0$, one has
\begin{eqnarray*}
&&\hskip-.8in
\Big|\Big\langle A^r e^{\tau A} \big( (\nabla\cdot f)g\big), A^r e^{\tau A} h  \Big\rangle - \Big\langle  (\nabla\cdot A^r e^{\tau A} f)g, A^r e^{\tau A} h  \Big\rangle\Big| \nonumber \\
&&\hskip-.9in
\leq C_r \|A^r f\| \|A^{r}  g\| \|A^{r}  h\| + C_r \tau \|A^{r+\frac{1}{2}} e^{\tau A} f\| \|A^{r+\frac{1}{2}} e^{\tau A} g\| \|A^{r+\frac{1}{2}} e^{\tau A} h\|.  
\end{eqnarray*}

\end{lemma}
The proof of Lemma \ref{lemma-difference-type1} is similarly to that of Lemma 8 in \cite{LO97} since it involves nonlinear term similar to that appearing in the Euler equations. The proof of Lemma \ref{lemma-difference-type2} is similarly to that of Lemma \ref{lemma-difference-type1}. Therefore, they are omitted. 

The next two lemmas provide the estimates for nonlinear terms which are specific to the structure of the PEs. 
\begin{lemma} \label{lemma-difference-type3}
For $f\in \mathcal{D}(e^{\tau A}: H^{r+\frac{3}{2}})$ and $g, h\in \mathcal{D}(e^{\tau A}: H^{r+\frac{1}{2}})$, where $r>\frac{5}{2}$, $\tau\geq 0$, and $\overline{f} = 0$, one has
\begin{eqnarray*}
&&\hskip-.8in
\Big|\Big\langle A^r e^{\tau A} \Big( (\int_0^z \nabla\cdot f(\boldsymbol{x}',s)ds) \partial_z g  \Big), A^r e^{\tau A} h  \Big\rangle - \Big\langle  (\int_0^z \nabla\cdot f(\boldsymbol{x}',s)ds)A^r e^{\tau A} \partial_z g , A^r e^{\tau A} h  \Big\rangle\Big| \\
&&\hskip-.9in
\leq C_r \|A^{r+1} f\| \|A^{r}  g\| \|A^{r}  h\| + C_r \tau \|A^{r+\frac{3}{2}} e^{\tau A} f\| \|A^{r+\frac{1}{2}} e^{\tau A} g\| \|A^{r+\frac{1}{2}} e^{\tau A} h\|.  
\end{eqnarray*}

\end{lemma}

\begin{lemma} \label{lemma-difference-type4}
For $g\in \mathcal{D}(e^{\tau A}: H^{r+\frac{3}{2}})$ and $f, h\in \mathcal{D}(e^{\tau A}: H^{r+\frac{1}{2}})$, where $r>\frac{5}{2}$, $\tau\geq 0$, and $\overline{f} = 0$, one has
\begin{eqnarray*}
&&\hskip-.8in
\Big|\Big\langle A^r e^{\tau A} \Big( (\int_0^z \nabla\cdot f(\boldsymbol{x}',s)ds) \partial_z g  \Big), A^r e^{\tau A} h  \Big\rangle -  \Big\langle \partial_z g A^r e^{\tau A} (\int_0^z \nabla\cdot f(\boldsymbol{x}',s)ds)  , A^r e^{\tau A} h  \Big\rangle\Big| \\
&&\hskip-.9in
\leq C_r \|A^{r+1} g\| \|A^{r}  f\| \|A^{r}  h\| + C_r \tau \|A^{r+\frac{3}{2}} e^{\tau A} g\| \|A^{r+\frac{1}{2}} e^{\tau A} f\| \|A^{r+\frac{1}{2}} e^{\tau A} h\|.  
\end{eqnarray*}
\end{lemma}
Since the proof of Lemma \ref{lemma-difference-type3} is similar to that of Lemma \ref{lemma-difference-type4}, we first focus below on the proof of Lemma \ref{lemma-difference-type4}, and later we sketch the proof of Lemma \ref{lemma-difference-type3} with emphasis on the main differences.

\begin{proof}(proof of Lemma \ref{lemma-difference-type4})
First, denote by $H = A^r e^{\tau A} h$, and let
\begin{eqnarray*}
&&\hskip-.8in
I:=\Big|\Big\langle A^r e^{\tau A} \Big( (\int_0^z \nabla\cdot f(\boldsymbol{x}',s)ds) \partial_z g  \Big), A^r e^{\tau A} h  \Big\rangle -  \Big\langle \partial_z g A^r e^{\tau A} (\int_0^z \nabla\cdot f(\boldsymbol{x}',s)ds)  , A^r e^{\tau A} h  \Big\rangle\Big| \nonumber\\
&&\hskip-.62in
= \Big| \Big\langle (\int_0^z \nabla\cdot f(\boldsymbol{x}',s)ds) \partial_z g, A^r e^{\tau A} H \Big\rangle - \Big\langle \partial_z g A^r e^{\tau A} (\int_0^z \nabla\cdot f(\boldsymbol{x}',s)ds)  , H  \Big\rangle\Big|.
\end{eqnarray*}
Similar to the proof of Lemma \ref{lemma-type3}, using Fourier representation of $f$, since $\overline{f}=0$, we have
\begin{equation*}
  \int_0^z \nabla \cdot f(\boldsymbol{x}',s) ds = \sum\limits_{\substack{\boldsymbol{j}\in \mathbb{Z}^3 \\ j_3 \neq 0}} \frac{1}{j_3}\boldsymbol{j}'\cdot\hat{f}_{\boldsymbol{j}} e^{2\pi( i\boldsymbol{j}' \cdot \boldsymbol{x}' +  ij_3 z)} - \sum\limits_{\substack{\boldsymbol{j}\in \mathbb{Z}^3 \\ j_3 \neq 0}} \frac{1}{j_3}\boldsymbol{j}'\cdot\hat{f}_{\boldsymbol{j}} e^{2\pi i\boldsymbol{j}'\cdot \boldsymbol{x}'},  
\end{equation*}
where $\boldsymbol{j}'=(j_1,j_2)$. Using Fourier representation of $g$ and $H$, we have
\begin{eqnarray*}
&&\hskip-.8in
I \leq C \sum\limits_{\substack{\boldsymbol{j}+\boldsymbol{k}+\boldsymbol{l}=0\\    j_3, k_3, \boldsymbol{j}' \neq 0}} \frac{1}{|j_3|} |\hat{f}_{\boldsymbol{j}}||\hat{g}_{\boldsymbol{k}}||\hat{H}_{\boldsymbol{l}}| |\boldsymbol{j}'||\boldsymbol{k}| \Big| |\boldsymbol{l}|^r e^{\tau|\boldsymbol{l}|} - |\boldsymbol{j}|^r e^{\tau|\boldsymbol{j}|} \Big| \nonumber\\
&&\hskip-.6in
+ C \sum\limits_{\substack{\boldsymbol{j}'+\boldsymbol{k}'+\boldsymbol{l}'=0\\  k_3+l_3 = 0 \\  j_3, k_3, \boldsymbol{j}' \neq 0}}\frac{1}{|j_3|}|\hat{f}_{\boldsymbol{j}}||\hat{g}_{\boldsymbol{k}}||\hat{H}_{\boldsymbol{l}}| |\boldsymbol{j}'||\boldsymbol{k}| \Big| |\boldsymbol{l}|^r e^{\tau|\boldsymbol{l}|} - |(\boldsymbol{j}',0)|^r e^{\tau|(\boldsymbol{j}',0)|} \Big|:= I_1 + I_2.
\end{eqnarray*}

We estimate $I_2$ first. By virtue of the following observation \cite{LO97}: 
For $r\geq 1$ and $\tau \geq 0$, and for all positive $\xi,\eta \in \mathbb{R}$, we have 
\begin{eqnarray}\label{lemma-inequality}
&&\hskip-.8in
|\xi^r e^{\tau \xi} - \eta^r e^{\tau \eta}| \leq C_r|\xi - \eta|\Big( |\xi - \eta|^{r-1} + \eta^{r-1} + \tau(|\xi - \eta|^{r} + \eta^r)e^{\tau|\xi-\eta|}e^{\tau \eta} \Big);
\end{eqnarray}
with $\xi = |\boldsymbol{l}|$, $\eta = |(\boldsymbol{j}',0)| = |\boldsymbol{j}'|$, and $|\xi - \eta| = \Big||\boldsymbol{l}|- |(\boldsymbol{j}',0)| \Big|\leq \Big|-\boldsymbol{l}-(\boldsymbol{j}',0)\Big| = |\boldsymbol{k}|$, inequality (\ref{lemma-inequality}) implies
\begin{eqnarray}\label{I2-type4}
I_2 \leq C_r\sum\limits_{\substack{\boldsymbol{j}'+\boldsymbol{k}'+\boldsymbol{l}'=0\\  k_3+l_3 = 0 \\  j_3, k_3, \boldsymbol{j}' \neq 0}}\frac{1}{|j_3|}|\hat{f}_{\boldsymbol{j}}||\hat{g}_{\boldsymbol{k}}||\hat{H}_{\boldsymbol{l}}| |\boldsymbol{j}'||\boldsymbol{k}|^2 \Big( |\boldsymbol{k}|^{r-1} + |\boldsymbol{j}'|^{r-1} + \tau(|\boldsymbol{k}|^{r} + |\boldsymbol{j}'|^{r})e^{\tau|\boldsymbol{k}|}e^{\tau|\boldsymbol{j}|} \Big).
\end{eqnarray}

By the definition of $H$, and since $e^x \leq 1+xe^x$ for any $x\geq 0$, we have 
\begin{eqnarray*}
|\hat{H}_{\boldsymbol{l}}| = |\boldsymbol{l}|^r e^{\tau |\boldsymbol{l}|} |\hat{h}_{\boldsymbol{l}}| \leq |\boldsymbol{l}|^r(1+\tau |\boldsymbol{l}|e^{\tau |\boldsymbol{l}|}) |\hat{h}_{\boldsymbol{l}}| \leq |\boldsymbol{l}|^r |\hat{h}_{\boldsymbol{l}}| + \tau(|\boldsymbol{j}'|+|\boldsymbol{k}|) |\hat{H}_{\boldsymbol{l}}|.
\end{eqnarray*}
Therefore, one obtains that
\begin{eqnarray*}
&&\hskip-.68in
|\hat{H}_{\boldsymbol{l}}| \Big( |\boldsymbol{k}|^{r-1} + |\boldsymbol{j}'|^{r-1} + \tau(|\boldsymbol{k}|^{r} + |\boldsymbol{j}'|^{r})e^{\tau|\boldsymbol{k}|}e^{\tau|\boldsymbol{j}|} \Big)\nonumber\\
&&\hskip-.8in
\leq \Big(|\boldsymbol{l}|^r |\hat{h}_{\boldsymbol{l}}| + \tau(|\boldsymbol{j}'|+|\boldsymbol{k}|) |\hat{H}_{\boldsymbol{l}}|\Big) \Big( |\boldsymbol{k}|^{r-1} + |\boldsymbol{j}'|^{r-1} \Big) + |\hat{H}_{\boldsymbol{l}}| \Big(\tau(|\boldsymbol{k}|^{r} + |\boldsymbol{j}'|^{r})e^{\tau|\boldsymbol{k}|}e^{\tau|\boldsymbol{j}|} \Big)\nonumber\\
&&\hskip-.8in
\leq |\hat{h}_{\boldsymbol{l}}| |\boldsymbol{l}|^r(|\boldsymbol{k}|^{r-1} + |\boldsymbol{j}'|^{r-1}) + \tau C_r |\hat{H}_{\boldsymbol{l}}|(|\boldsymbol{k}|^{r} + |\boldsymbol{j}'|^{r})e^{\tau|\boldsymbol{k}|}e^{\tau|\boldsymbol{j}|}.\label{Hl}
\end{eqnarray*}
Based on this, one has
\begin{eqnarray*}
&&\hskip-.8in
I_2 \leq C_r\sum\limits_{\substack{\boldsymbol{j}'+\boldsymbol{k}'+\boldsymbol{l}'=0\\  k_3+l_3 = 0 \\  j_3, k_3, \boldsymbol{j}' \neq 0}}\frac{1}{|j_3|}|\hat{f}_{\boldsymbol{j}}||\hat{g}_{\boldsymbol{k}}||\hat{h}_{\boldsymbol{l}}| |\boldsymbol{j}'||\boldsymbol{k}|^2 |\boldsymbol{l}|^r(|\boldsymbol{k}|^{r-1} + |\boldsymbol{j}'|^{r-1}) \nonumber\\
&&\hskip-.6in
+ \tau C_r\sum\limits_{\substack{\boldsymbol{j}'+\boldsymbol{k}'+\boldsymbol{l}'=0\\  k_3+l_3 = 0 \\  j_3, k_3, \boldsymbol{j}' \neq 0}}\frac{1}{|j_3|}|\hat{f}_{\boldsymbol{j}}||\hat{g}_{\boldsymbol{k}}||\hat{H}_{\boldsymbol{l}}| |\boldsymbol{j}'||\boldsymbol{k}|^2 (|\boldsymbol{k}|^{r} + |\boldsymbol{j}'|^{r})e^{\tau|\boldsymbol{k}|}e^{\tau|\boldsymbol{j}|}:= I_{21} + I_{22}.
\end{eqnarray*}
Here
\begin{eqnarray*}
&&\hskip-.8in
I_{21} = C_r\Big(\sum\limits_{\substack{\boldsymbol{j}'+\boldsymbol{k}'+\boldsymbol{l}'=0\\  k_3+l_3 = 0 \\  j_3, k_3, \boldsymbol{j}' \neq 0}}\frac{1}{|j_3|}|\hat{f}_{\boldsymbol{j}}||\hat{g}_{\boldsymbol{k}}||\hat{h}_{\boldsymbol{l}}| |\boldsymbol{j}'||\boldsymbol{k}|^{r+1} |\boldsymbol{l}|^r 
+ \frac{1}{|j_3|}|\hat{f}_{\boldsymbol{j}}||\hat{g}_{\boldsymbol{k}}||\hat{h}_{\boldsymbol{l}}| |\boldsymbol{j}'|^{r}|\boldsymbol{k}|^{2} |\boldsymbol{l}|^r\Big) := I_{211}+ I_{212}.
\end{eqnarray*}
Thanks to the Cauchy–Schwarz inequality, since $r> \frac{5}{2}$, we have
\begin{eqnarray*}
&&\hskip-.8in
I_{211} = C_r \sum\limits_{\substack{\boldsymbol{j}\in \mathbb{Z}^3 \\ j_3, \boldsymbol{j}'\neq 0} }  \frac{1}{|j_3|}|\boldsymbol{j}'| |\hat{f}_{\boldsymbol{j}}| \sum\limits_{\substack{\boldsymbol{k}\in \mathbb{Z}^3 \\ k_3\neq 0} } |\boldsymbol{k}|^{r+1} |(\boldsymbol{j}'+\boldsymbol{k}',k_3)|^r |\hat{g}_{\boldsymbol{k}}||\hat{h}_{-(\boldsymbol{j}'+\boldsymbol{k}',k_3)}|\nonumber\\
&&\hskip-.58in 
\leq C_r \Big(\sum\limits_{\substack{\boldsymbol{j}\in \mathbb{Z}^3 \\ j_3, \boldsymbol{j}'\neq 0} }  \frac{1}{|j_3|^2}|\boldsymbol{j}'|^{2-2r} \Big)^{\frac{1}{2}} \Big(\sum\limits_{\substack{\boldsymbol{j}\in \mathbb{Z}^3 \\ j_3, \boldsymbol{j}'\neq 0} }   |\boldsymbol{j}|^{2r} |\hat{f}_{\boldsymbol{j}}|^2  \Big)^{\frac{1}{2}} 
\Big( \sum\limits_{\substack{\boldsymbol{k}\in \mathbb{Z}^3 \\ k_3\neq 0} } |\boldsymbol{k}|^{2r+2} |\hat{g}_{\boldsymbol{k}}|^2\Big)^{\frac{1}{2}}  \nonumber \\
&&\hskip-.38in
\times \sup\limits_{\boldsymbol{j}\in \mathbb{Z}^3} \Big( \sum\limits_{\substack{\boldsymbol{k}\in \mathbb{Z}^3 \\ k_3\neq 0}}  |(\boldsymbol{j}'+\boldsymbol{k}',k_3)|^{2r}|\hat{h}_{-(\boldsymbol{j}'+\boldsymbol{k}',k_3)}|^2\Big)^{\frac{1}{2}}  
\leq C_r \|A^r f\| \|A^{r+1} g\| \|A^r h \|, \text{ and }
\end{eqnarray*}
\begin{eqnarray*}
&&\hskip-.8in
I_{212} = C_r \sum\limits_{\substack{\boldsymbol{k}\in \mathbb{Z}^3 \\ k_3\neq 0} } |\boldsymbol{k}|^{2} |\hat{g}_{\boldsymbol{k}}| \sum\limits_{\substack{\boldsymbol{j}\in \mathbb{Z}^3 \\ j_3, \boldsymbol{j}' \neq 0} } \frac{1}{|j_3|}  |\boldsymbol{j}'|^{r} |\hat{f}_{\boldsymbol{j}}|  |(\boldsymbol{j}'+\boldsymbol{k}',k_3)|^{r}|\hat{h}_{-(\boldsymbol{j}'+\boldsymbol{k}',k_3)}| \nonumber\\
&&\hskip-.58in 
\leq C_r \sum\limits_{\boldsymbol{k}'\in \mathbb{Z}^2} |(\boldsymbol{k}',\pm 1)|^{1-r}  \sum\limits_{k_3\neq 0}  |\boldsymbol{k}|^{r+1} |\hat{g}_{\boldsymbol{k}}| \Big( \sum\limits_{\substack{\boldsymbol{j}\in \mathbb{Z}^3 \\ \boldsymbol{j} \neq 0} } |\boldsymbol{j}|^{2r}  |\hat{f}_{\boldsymbol{j}}|^2\Big)^{\frac{1}{2}}  \nonumber \\
&&\hskip-.38in
\times \Big( \sum\limits_{j_3\neq 0} \frac{1}{|j_3|^2} \sum\limits_{\boldsymbol{j}'\in \mathbb{Z}^2 }  |(\boldsymbol{j}'+\boldsymbol{k}',k_3)|^{2r}|\hat{h}_{-(\boldsymbol{j}'+\boldsymbol{k}',k_3)}|^2\Big)^{\frac{1}{2}}  \nonumber\\
&&\hskip-.58in \leq C_r \|A^{r} f\| \sum\limits_{\boldsymbol{k}'\in \mathbb{Z}^2} |(\boldsymbol{k}',\pm 1)|^{1-r} \Big(\sum\limits_{k_3\neq 0}  |\boldsymbol{k}|^{2r+2} |\hat{g}_{\boldsymbol{k}}|^2  \Big)^{\frac{1}{2}}  \nonumber \\
&&\hskip-.38in
\times \Big( \sum\limits_{k_3\neq 0} \sum\limits_{\boldsymbol{j}'\in \mathbb{Z}^2 }  |(\boldsymbol{j}'+\boldsymbol{k}',k_3)|^{2r}|\hat{h}_{-(\boldsymbol{j}'+\boldsymbol{k}',k_3)}|^2\Big)^{\frac{1}{2}}  \nonumber\\
&&\hskip-.58in 
\leq C_r  \|A^r f\| \|A^r h \| \Big(\sum\limits_{\boldsymbol{k}'\in \mathbb{Z}^2} |(\boldsymbol{k}',\pm 1)|^{2-2r}\Big)^{\frac{1}{2}} \Big(\sum\limits_{\boldsymbol{k}'\in \mathbb{Z}^2} \sum\limits_{k_3\neq 0}  |\boldsymbol{k}|^{2r+2} |\hat{g}_{\boldsymbol{k}}|^2   \Big)^{\frac{1}{2}}\nonumber\\
&&\hskip-.58in 
\leq C_r \|A^r f\| \|A^{r+1} g\| \|A^r h \|.
\end{eqnarray*}

Next, for $I_{22}$, we have
\begin{eqnarray*}
&&\hskip-.8in
I_{22} = \tau C_r\sum\limits_{\substack{\boldsymbol{j}'+\boldsymbol{k}'+\boldsymbol{l}'=0\\  k_3+l_3 = 0 \\  j_3, k_3, \boldsymbol{j}' \neq 0}}\frac{1}{|j_3|}|\hat{f}_{\boldsymbol{j}}||\hat{g}_{\boldsymbol{k}}||\hat{H}_{\boldsymbol{l}}| |\boldsymbol{j}'||\boldsymbol{k}|^{r+2} e^{\tau|\boldsymbol{k}|}e^{\tau|\boldsymbol{j}|} \nonumber\\
&&\hskip-.58in  
+ \tau C_r\sum\limits_{\substack{\boldsymbol{j}'+\boldsymbol{k}'+\boldsymbol{l}'=0\\  k_3+l_3 = 0 \\  j_3, k_3, \boldsymbol{j}' \neq 0}}\frac{1}{|j_3|}|\hat{f}_{\boldsymbol{j}}||\hat{g}_{\boldsymbol{k}}||\hat{H}_{\boldsymbol{l}}| |\boldsymbol{j}'|^{r+1} |\boldsymbol{k}|^{2} e^{\tau|\boldsymbol{k}|}e^{\tau|\boldsymbol{j}|} := I_{221} + I_{222}.
\end{eqnarray*}
Noticing that $|\boldsymbol{k}|^{\frac{1}{2}}\leq C(|\boldsymbol{j}'|^{\frac{1}{2}} + |\boldsymbol{l}|^{\frac{1}{2}})$ and $|\boldsymbol{j}'|^{\frac{1}{2}} \leq C(|\boldsymbol{k}|^{\frac{1}{2}} + |\boldsymbol{l}|^{\frac{1}{2}})$, thanks to the Cauchy–Schwarz inequality, since $r> \frac{5}{2}$, we have 
\begin{eqnarray*}
&&\hskip-.8in
I_{221} =\tau C_r\sum\limits_{\substack{\boldsymbol{j}'+\boldsymbol{k}'+\boldsymbol{l}'=0\\  k_3+l_3 = 0 \\  j_3, k_3, \boldsymbol{j}' \neq 0}}\frac{1}{|j_3|}|\hat{f}_{\boldsymbol{j}}||\hat{g}_{\boldsymbol{k}}||\hat{H}_{\boldsymbol{l}}| |\boldsymbol{j}'||\boldsymbol{k}|^{r+2} e^{\tau|\boldsymbol{k}|}e^{\tau|\boldsymbol{j}|} \nonumber\\
&&\hskip-.55in
\leq \tau C_r\sum\limits_{\substack{\boldsymbol{j}'+\boldsymbol{k}'+\boldsymbol{l}'=0\\  k_3+l_3 = 0 \\  j_3, k_3, \boldsymbol{j}' \neq 0}}\frac{1}{|j_3|}|\hat{f}_{\boldsymbol{j}}||\hat{g}_{\boldsymbol{k}}||\hat{h}_{\boldsymbol{l}}| |\boldsymbol{j}'| |\boldsymbol{l}|^r |\boldsymbol{k}|^{r+\frac{3}{2}}(|\boldsymbol{j}'|^{\frac{1}{2}}+ |\boldsymbol{l}|^{\frac{1}{2}}) e^{\tau|\boldsymbol{k}|}e^{\tau|\boldsymbol{j}|}e^{\tau|\boldsymbol{l}|}\nonumber\\
&&\hskip-.55in
\leq \tau C_r\sum\limits_{\substack{\boldsymbol{j}'+\boldsymbol{k}'+\boldsymbol{l}'=0\\  k_3+l_3 = 0 \\  j_3, k_3, \boldsymbol{j}' \neq 0}}\frac{1}{|j_3|}|\hat{f}_{\boldsymbol{j}}||\hat{g}_{\boldsymbol{k}}||\hat{h}_{\boldsymbol{l}}| |\boldsymbol{j}'|^{\frac{3}{2}} |\boldsymbol{l}|^{r+\frac{1}{2}} |\boldsymbol{k}|^{r+\frac{3}{2}} e^{\tau|\boldsymbol{k}|}e^{\tau|\boldsymbol{j}|}e^{\tau|\boldsymbol{l}|}\nonumber\\
&&\hskip-.55in 
\leq \tau C_r \Big(\sum\limits_{\substack{\boldsymbol{j}\in \mathbb{Z}^3 \\ j_3, \boldsymbol{j}'\neq 0} }  \frac{1}{|j_3|^2}|\boldsymbol{j}'|^{2-2r} \Big)^{\frac{1}{2}} \Big(\sum\limits_{\substack{\boldsymbol{j}\in \mathbb{Z}^3 \\ j_3, \boldsymbol{j}'\neq 0} }   |\boldsymbol{j}|^{2r+1} e^{2\tau |\boldsymbol{j}|} |\hat{f}_{\boldsymbol{j}}|^2  \Big)^{\frac{1}{2}} 
\Big( \sum\limits_{\substack{\boldsymbol{k}\in \mathbb{Z}^3 \\ k_3\neq 0} } |\boldsymbol{k}|^{2r+3} | e^{2\tau |\boldsymbol{k}|} \hat{g}_{\boldsymbol{k}}|^2\Big)^{\frac{1}{2}}  \nonumber \\
&&\hskip-.18in
\times \sup\limits_{\boldsymbol{j}\in \mathbb{Z}^3} \Big( \sum\limits_{\substack{\boldsymbol{k}\in \mathbb{Z}^3 \\ k_3\neq 0}}  |(\boldsymbol{j}'+\boldsymbol{k}',k_3)|^{2r+1} e^{2\tau |(\boldsymbol{j}'+\boldsymbol{k}',k_3)|} |\hat{h}_{-(\boldsymbol{j}'+\boldsymbol{k}',k_3)}|^2\Big)^{\frac{1}{2}}  \nonumber\\
&&\hskip-.55in
\leq \tau C_r \|A^{r+\frac{1}{2}} e^{\tau A} f\| \|A^{r+\frac{3}{2}} e^{\tau A} g\| \|A^{r+\frac{1}{2}} e^{\tau A} h\|, \text{ and }
\end{eqnarray*}
\begin{eqnarray*}
&&\hskip-.8in
I_{222} =\tau C_r\sum\limits_{\substack{\boldsymbol{j}'+\boldsymbol{k}'+\boldsymbol{l}'=0\\  k_3+l_3 = 0 \\  j_3, k_3, \boldsymbol{j}' \neq 0}}\frac{1}{|j_3|}|\hat{f}_{\boldsymbol{j}}||\hat{g}_{\boldsymbol{k}}||\hat{H}_{\boldsymbol{l}}| |\boldsymbol{j}'|^{r+1}|\boldsymbol{k}|^{2} e^{\tau|\boldsymbol{k}|}e^{\tau|\boldsymbol{j}|} \nonumber\\
&&\hskip-.55in
\leq \tau C_r\sum\limits_{\substack{\boldsymbol{j}'+\boldsymbol{k}'+\boldsymbol{l}'=0\\  k_3+l_3 = 0 \\  j_3, k_3, \boldsymbol{j}' \neq 0}}\frac{1}{|j_3|}|\hat{f}_{\boldsymbol{j}}||\hat{g}_{\boldsymbol{k}}||\hat{h}_{\boldsymbol{l}}| |\boldsymbol{j}'|^{r+\frac{1}{2}}|\boldsymbol{k}|^{2}|{\boldsymbol{l}}|^r(|\boldsymbol{k}|^{\frac{1}{2}}+ |\boldsymbol{l}|^{\frac{1}{2}}) e^{\tau|\boldsymbol{k}|}e^{\tau|\boldsymbol{j}|}e^{\tau|\boldsymbol{l}|}\nonumber\\
&&\hskip-.55in
\leq \tau C_r\sum\limits_{\substack{\boldsymbol{j}'+\boldsymbol{k}'+\boldsymbol{l}'=0\\  k_3+l_3 = 0 \\  j_3, k_3, \boldsymbol{j}' \neq 0}}\frac{1}{|j_3|}|\hat{f}_{\boldsymbol{j}}||\hat{g}_{\boldsymbol{k}}||\hat{h}_{\boldsymbol{l}}| |\boldsymbol{j}'|^{r+\frac{1}{2}}|\boldsymbol{k}|^{\frac{5}{2}}|{\boldsymbol{l}}|^{r+\frac{1}{2}} e^{\tau|\boldsymbol{k}|}e^{\tau|\boldsymbol{j}|}e^{\tau|\boldsymbol{l}|}\nonumber\\
&&\hskip-.55in 
\leq \tau C_r \sum\limits_{\boldsymbol{k}'\in \mathbb{Z}^2} |(\boldsymbol{k}',\pm 1)|^{1-r}  \sum\limits_{k_3\neq 0}  |\boldsymbol{k}|^{r+\frac{3}{2}} e^{\tau |\boldsymbol{k}|} |\hat{g}_{\boldsymbol{k}}| \Big( \sum\limits_{\substack{\boldsymbol{j}\in \mathbb{Z}^3 \\ \boldsymbol{j} \neq 0} } |\boldsymbol{j}|^{2r+1} e^{2\tau |\boldsymbol{j}|} |\hat{f}_{\boldsymbol{j}}|^2\Big)^{\frac{1}{2}}  \nonumber \\
&&\hskip-.18in
\times \Big( \sum\limits_{j_3\neq 0} \frac{1}{|j_3|^2} \sum\limits_{\boldsymbol{j}'\in \mathbb{Z}^2 }  |(\boldsymbol{j}'+\boldsymbol{k}',k_3)|^{2r+1} e^{2\tau |(\boldsymbol{j}'+\boldsymbol{k}',k_3)|}|\hat{h}_{-(\boldsymbol{j}'+\boldsymbol{k}',k_3)}|^2\Big)^{\frac{1}{2}}  \nonumber\\
&&\hskip-.55in 
\leq \tau C_r \|A^{r+\frac{1}{2}} e^{\tau A} f\| \sum\limits_{\boldsymbol{k}'\in \mathbb{Z}^2} |(\boldsymbol{k}',\pm 1)|^{1-r} \Big(\sum\limits_{k_3\neq 0}  |\boldsymbol{k}|^{2r+3} e^{2\tau |\boldsymbol{k}|} |\hat{g}_{\boldsymbol{k}}|^2 \Big)^{\frac{1}{2}} \nonumber\\
&&\hskip-.18in
\times \Big( \sum\limits_{k_3\neq 0} \sum\limits_{\boldsymbol{j}'\in \mathbb{Z}^2 }  |(\boldsymbol{j}'+\boldsymbol{k}',k_3)|^{2r+1} e^{2\tau |(\boldsymbol{j}'+\boldsymbol{k}',k_3)|}|\hat{h}_{-(\boldsymbol{j}'+\boldsymbol{k}',k_3)}|^2\Big)^{\frac{1}{2}}  \nonumber\\
&&\hskip-.55in
\leq \tau C_r \|A^{r+\frac{1}{2}} e^{\tau A} f\| \|A^{r+\frac{1}{2}} e^{\tau A} h\| \Big(\sum\limits_{\boldsymbol{k}'\in \mathbb{Z}^2} |(\boldsymbol{k}',\pm 1)|^{2-2r}\Big)^{\frac{1}{2}} \Big(\sum\limits_{\boldsymbol{k}'\in \mathbb{Z}^2} \sum\limits_{k_3\neq 0}  |\boldsymbol{k}|^{2r+3} e^{2\tau |\boldsymbol{k}|} |\hat{g}_{\boldsymbol{k}}|^2 \Big)^{\frac{1}{2}} \nonumber\\
&&\hskip-.55in
\leq \tau C_r \|A^{r+\frac{1}{2}} e^{\tau A} f\| \|A^{r+\frac{3}{2}} e^{\tau A} g\| \|A^{r+\frac{1}{2}} e^{\tau A} h\|.
\end{eqnarray*}
Therefore, $I_2$ satisfies the desired estimates. 

To estimate $I_1$, we use \eqref{lemma-inequality} with $\xi = |\boldsymbol{l}|$, $\eta = |\boldsymbol{j}|$, and with $|\xi - \eta| = \Big||\boldsymbol{l}|- |\boldsymbol{j}| \Big|\leq |-\boldsymbol{l}-\boldsymbol{j}| = |\boldsymbol{k}|$, to obtain 
\begin{equation}\label{I1-type4}
    I_1 \leq C_r \sum\limits_{\substack{\boldsymbol{j}+\boldsymbol{k}+\boldsymbol{l}=0\\    j_3, k_3, \boldsymbol{j}' \neq 0}} \frac{1}{|j_3|} |\hat{f}_{\boldsymbol{j}}||\hat{g}_{\boldsymbol{k}}||\hat{H}_{\boldsymbol{l}}| |\boldsymbol{j}'||\boldsymbol{k}|^2  \Big( |\boldsymbol{k}|^{r-1} + |\boldsymbol{j}|^{r-1} + \tau(|\boldsymbol{k}|^{r} + |\boldsymbol{j}|^{r})e^{\tau|\boldsymbol{k}|}e^{\tau|\boldsymbol{j}|} \Big).
\end{equation}
Thanks to (\ref{Hl}), one obtains that
\begin{eqnarray*}
&&\hskip-.8in
I_1 \leq C_r \sum\limits_{\substack{\boldsymbol{j}+\boldsymbol{k}+\boldsymbol{l}=0\\    j_3, k_3, \boldsymbol{j}' \neq 0}}\frac{1}{|j_3|}|\hat{f}_{\boldsymbol{j}}||\hat{g}_{\boldsymbol{k}}||\hat{h}_{\boldsymbol{l}}| |\boldsymbol{j}'||\boldsymbol{k}|^2 |\boldsymbol{l}|^r(|\boldsymbol{k}|^{r-1} + |\boldsymbol{j}|^{r-1}) \nonumber\\
&&\hskip-.6in
+ \tau C_r \sum\limits_{\substack{\boldsymbol{j}+\boldsymbol{k}+\boldsymbol{l}=0\\    j_3, k_3, \boldsymbol{j}' \neq 0}}\frac{1}{|j_3|}|\hat{f}_{\boldsymbol{j}}||\hat{g}_{\boldsymbol{k}}||\hat{H}_{\boldsymbol{l}}| |\boldsymbol{j}'||\boldsymbol{k}|^2 (|\boldsymbol{k}|^{r} + |\boldsymbol{j}|^{r})e^{\tau|\boldsymbol{k}|}e^{\tau|\boldsymbol{j}|}:= I_{11} + I_{12}.
\end{eqnarray*}
Here
\begin{eqnarray*}
&&\hskip-.8in
I_{11} \leq C_r\Big(\sum\limits_{\substack{\boldsymbol{j}+\boldsymbol{k}+\boldsymbol{l}=0\\    j_3, k_3, \boldsymbol{j}' \neq 0}}\frac{1}{|j_3|}|\hat{f}_{\boldsymbol{j}}||\hat{g}_{\boldsymbol{k}}||\hat{h}_{\boldsymbol{l}}| |\boldsymbol{j}||\boldsymbol{k}|^{r+1} |\boldsymbol{l}|^r 
+ \frac{1}{|j_3|}|\hat{f}_{\boldsymbol{j}}||\hat{g}_{\boldsymbol{k}}||\hat{h}_{\boldsymbol{l}}| |\boldsymbol{j}|^{r}|\boldsymbol{k}|^{2} |\boldsymbol{l}|^r\Big) := I_{111}+ I_{112}.
\end{eqnarray*}
Thanks to the Cauchy–Schwarz inequality, since $r> \frac{5}{2}$, we have
\begin{eqnarray*}
&&\hskip-.8in
I_{111} = C_r \sum\limits_{\substack{\boldsymbol{j}\in \mathbb{Z}^3 \\ j_3, \boldsymbol{j}'\neq 0} }  \frac{1}{|j_3|}|\boldsymbol{j}| |\hat{f}_{\boldsymbol{j}}| \sum\limits_{\substack{\boldsymbol{k}\in \mathbb{Z}^3 \\ k_3\neq 0} } |\boldsymbol{k}|^{r+1} |\boldsymbol{j}+\boldsymbol{k}|^r |\hat{g}_{\boldsymbol{k}}||\hat{h}_{-(\boldsymbol{j}+\boldsymbol{k})}|\nonumber\\
&&\hskip-.58in 
\leq C_r \Big(\sum\limits_{\substack{\boldsymbol{j}\in \mathbb{Z}^3 \\ j_3, \boldsymbol{j}'\neq 0} }  |\boldsymbol{j}|^{2-2r} \Big)^{\frac{1}{2}} \Big(\sum\limits_{\substack{\boldsymbol{j}\in \mathbb{Z}^3 \\ j_3, \boldsymbol{j}'\neq 0} }   |\boldsymbol{j}|^{2r} |\hat{f}_{\boldsymbol{j}}|^2  \Big)^{\frac{1}{2}} 
\Big( \sum\limits_{\substack{\boldsymbol{k}\in \mathbb{Z}^3 \\ k_3\neq 0} } |\boldsymbol{k}|^{2r+2} |\hat{g}_{\boldsymbol{k}}|^2\Big)^{\frac{1}{2}}  \nonumber \\
&&\hskip-.38in
\times \sup\limits_{\boldsymbol{j}\in \mathbb{Z}^3} \Big( \sum\limits_{\substack{\boldsymbol{k}\in \mathbb{Z}^3 \\ k_3\neq 0}}  |\boldsymbol{j}+\boldsymbol{k}|^{2r}|\hat{h}_{-(\boldsymbol{j}+\boldsymbol{k})}|^2\Big)^{\frac{1}{2}}  
\leq C_r \|A^r f\| \|A^{r+1} g\| \|A^r h \|, \text{ and }
\end{eqnarray*}
\begin{eqnarray*}
&&\hskip-.8in
I_{112} = C_r \sum\limits_{\substack{\boldsymbol{k}\in \mathbb{Z}^3 \\ k_3\neq 0} } |\boldsymbol{k}|^{2} |\hat{g}_{\boldsymbol{k}}| \sum\limits_{\substack{\boldsymbol{j}\in \mathbb{Z}^3 \\ j_3, \boldsymbol{j}' \neq 0} } \frac{1}{|j_3|}  |\boldsymbol{j}|^{r} |\hat{f}_{\boldsymbol{j}}|  |\boldsymbol{j}+\boldsymbol{k}|^{r}|\hat{h}_{-(\boldsymbol{j}+\boldsymbol{k})}| \nonumber\\
&&\hskip-.58in 
\leq C_r \Big(\sum\limits_{\substack{\boldsymbol{k}\in \mathbb{Z}^3 \\ \boldsymbol{k} \neq 0} } |\boldsymbol{k}|^{2-2r} \Big)^{\frac{1}{2}} \Big(\sum\limits_{\substack{\boldsymbol{k}\in \mathbb{Z}^3 \\ \boldsymbol{k} \neq 0} }  |\boldsymbol{k}|^{2r+2} |\hat{g}_{\boldsymbol{k}}|^2  \Big)^{\frac{1}{2}} \Big( \sum\limits_{\substack{\boldsymbol{j}\in \mathbb{Z}^3 \\ \boldsymbol{j} \neq 0} } |\boldsymbol{j}|^{2r}  |\hat{f}_{\boldsymbol{j}}|^2\Big)^{\frac{1}{2}}  \nonumber \\
&&\hskip-.38in
\times \sup\limits_{\boldsymbol{k}\in\mathbb{Z}^3} \Big( \sum\limits_{\boldsymbol{j}\in \mathbb{Z}^3}  |\boldsymbol{j}+\boldsymbol{k}|^{2r}|\hat{h}_{-(\boldsymbol{j}+\boldsymbol{k})}|^2\Big)^{\frac{1}{2}} 
\leq C_r \|A^r f\| \|A^{r+1} g\| \|A^r h \|.
\end{eqnarray*}

Next, for $I_{12}$, we have
\begin{eqnarray*}
&&\hskip-.8in
I_{12} \leq \tau C_r\sum\limits_{\substack{\boldsymbol{j}+\boldsymbol{k}+\boldsymbol{l}=0\\    j_3, k_3, \boldsymbol{j}' \neq 0}}\frac{1}{|j_3|}|\hat{f}_{\boldsymbol{j}}||\hat{g}_{\boldsymbol{k}}||\hat{H}_{\boldsymbol{l}}| |\boldsymbol{j}||\boldsymbol{k}|^{r+2} e^{\tau|\boldsymbol{k}|}e^{\tau|\boldsymbol{j}|} \nonumber\\
&&\hskip-.58in  
+ \tau C_r\sum\limits_{\substack{\boldsymbol{j}+\boldsymbol{k}+\boldsymbol{l}=0\\    j_3, k_3, \boldsymbol{j}' \neq 0}}\frac{1}{|j_3|}|\hat{f}_{\boldsymbol{j}}||\hat{g}_{\boldsymbol{k}}||\hat{H}_{\boldsymbol{l}}| |\boldsymbol{j}|^{r+1} |\boldsymbol{k}|^{2} e^{\tau|\boldsymbol{k}|}e^{\tau|\boldsymbol{j}|} := I_{121} + I_{122}.
\end{eqnarray*}
Since $|\boldsymbol{k}|^{\frac{1}{2}}\leq C(|\boldsymbol{j}|^{\frac{1}{2}} + |\boldsymbol{l}|^{\frac{1}{2}})$ and $|\boldsymbol{j}|^{\frac{1}{2}} \leq C(|\boldsymbol{k}|^{\frac{1}{2}} + |\boldsymbol{l}|^{\frac{1}{2}})$, thanks to the Cauchy–Schwarz inequality, since $r> \frac{5}{2}$, we have 
\begin{eqnarray*}
&&\hskip-.8in
I_{121} =\tau C_r\sum\limits_{\substack{\boldsymbol{j}+\boldsymbol{k}+\boldsymbol{l}=0\\    j_3, k_3, \boldsymbol{j}' \neq 0}}\frac{1}{|j_3|}|\hat{f}_{\boldsymbol{j}}||\hat{g}_{\boldsymbol{k}}||\hat{H}_{\boldsymbol{l}}| |\boldsymbol{j}||\boldsymbol{k}|^{r+2} e^{\tau|\boldsymbol{k}|}e^{\tau|\boldsymbol{j}|} \nonumber\\
&&\hskip-.55in
\leq \tau C_r\sum\limits_{\substack{\boldsymbol{j}+\boldsymbol{k}+\boldsymbol{l}=0\\    j_3, k_3, \boldsymbol{j}', \boldsymbol{l} \neq 0}}\frac{1}{|j_3|}|\hat{f}_{\boldsymbol{j}}||\hat{g}_{\boldsymbol{k}}||\hat{h}_{\boldsymbol{l}}| |\boldsymbol{j}| |\boldsymbol{l}|^r |\boldsymbol{k}|^{r+\frac{3}{2}}(|\boldsymbol{j}|^{\frac{1}{2}}+ |\boldsymbol{l}|^{\frac{1}{2}}) e^{\tau|\boldsymbol{k}|}e^{\tau|\boldsymbol{j}|}e^{\tau|\boldsymbol{l}|}\nonumber\\
&&\hskip-.55in
\leq \tau C_r\sum\limits_{\substack{\boldsymbol{j}+\boldsymbol{k}+\boldsymbol{l}=0\\    j_3, k_3, \boldsymbol{j}' , \boldsymbol{l}\neq 0}}\frac{1}{|j_3|}|\hat{f}_{\boldsymbol{j}}||\hat{g}_{\boldsymbol{k}}||\hat{h}_{\boldsymbol{l}}| |\boldsymbol{j}|^{\frac{3}{2}} |\boldsymbol{l}|^{r+\frac{1}{2}} |\boldsymbol{k}|^{r+\frac{3}{2}} e^{\tau|\boldsymbol{k}|}e^{\tau|\boldsymbol{j}|}e^{\tau|\boldsymbol{l}|}\nonumber\\
&&\hskip-.55in 
\leq \tau C_r \Big(\sum\limits_{\substack{\boldsymbol{j}\in \mathbb{Z}^3 \\ j_3, \boldsymbol{j}'\neq 0} }  |\boldsymbol{j}|^{2-2r} \Big)^{\frac{1}{2}} \Big(\sum\limits_{\substack{\boldsymbol{j}\in \mathbb{Z}^3 \\ j_3, \boldsymbol{j}'\neq 0} }   |\boldsymbol{j}|^{2r+1} e^{2\tau |\boldsymbol{j}|} |\hat{f}_{\boldsymbol{j}}|^2  \Big)^{\frac{1}{2}} 
\Big( \sum\limits_{\substack{\boldsymbol{k}\in \mathbb{Z}^3 \\ k_3\neq 0} } |\boldsymbol{k}|^{2r+3} | e^{2\tau |\boldsymbol{k}|} \hat{g}_{\boldsymbol{k}}|^2\Big)^{\frac{1}{2}}  \nonumber \\
&&\hskip-.18in
\times \sup\limits_{\boldsymbol{j}\in \mathbb{Z}^3} \Big( \sum\limits_{\substack{\boldsymbol{k}\in \mathbb{Z}^3 \\ k_3\neq 0}}  |\boldsymbol{j}+\boldsymbol{k}|^{2r+1} e^{2\tau |\boldsymbol{j}+\boldsymbol{k}|} |\hat{h}_{-(\boldsymbol{j}+\boldsymbol{k})}|^2\Big)^{\frac{1}{2}}  \nonumber\\
&&\hskip-.55in
\leq \tau C_r \|A^{r+\frac{1}{2}} e^{\tau A} f\| \|A^{r+\frac{3}{2}} e^{\tau A} g\| \|A^{r+\frac{1}{2}} e^{\tau A} h\|, \text{ and }
\end{eqnarray*}
\begin{eqnarray*}
&&\hskip-.8in
I_{122} =\tau C_r\sum\limits_{\substack{\boldsymbol{j}+\boldsymbol{k}+\boldsymbol{l}=0\\    j_3, k_3, \boldsymbol{j}' \neq 0}}\frac{1}{|j_3|}|\hat{f}_{\boldsymbol{j}}||\hat{g}_{\boldsymbol{k}}||\hat{H}_{\boldsymbol{l}}| |\boldsymbol{j}|^{r+1}|\boldsymbol{k}|^{2} e^{\tau|\boldsymbol{k}|}e^{\tau|\boldsymbol{j}|} \nonumber\\
&&\hskip-.55in
\leq \tau C_r\sum\limits_{\substack{\boldsymbol{j}+\boldsymbol{k}+\boldsymbol{l}=0\\    j_3, k_3, \boldsymbol{j}', \boldsymbol{l} \neq 0}}\frac{1}{|j_3|}|\hat{f}_{\boldsymbol{j}}||\hat{g}_{\boldsymbol{k}}||\hat{h}_{\boldsymbol{l}}| |\boldsymbol{j}|^{r+\frac{1}{2}}|\boldsymbol{k}|^{2}|{\boldsymbol{l}}|^r(|\boldsymbol{k}|^{\frac{1}{2}}+ |\boldsymbol{l}|^{\frac{1}{2}}) e^{\tau|\boldsymbol{k}|}e^{\tau|\boldsymbol{j}|}e^{\tau|\boldsymbol{l}|}\nonumber\\
&&\hskip-.55in
\leq \tau C_r\sum\limits_{\substack{\boldsymbol{j}+\boldsymbol{k}+\boldsymbol{l}=0\\    j_3, k_3, \boldsymbol{j}', \boldsymbol{l} \neq 0}}\frac{1}{|j_3|}|\hat{f}_{\boldsymbol{j}}||\hat{g}_{\boldsymbol{k}}||\hat{h}_{\boldsymbol{l}}| |\boldsymbol{j}|^{r+\frac{1}{2}}|\boldsymbol{k}|^{\frac{5}{2}}|{\boldsymbol{l}}|^{r+\frac{1}{2}} e^{\tau|\boldsymbol{k}|}e^{\tau|\boldsymbol{j}|}e^{\tau|\boldsymbol{l}|}\nonumber\\
&&\hskip-.55in 
\leq \tau C_r \Big(\sum\limits_{\substack{\boldsymbol{k}\in \mathbb{Z}^3 \\ \boldsymbol{k} \neq 0} } |\boldsymbol{k}|^{2-2r}\Big)^{\frac{1}{2}}  \Big(\sum\limits_{\substack{\boldsymbol{k}\in \mathbb{Z}^3 \\ \boldsymbol{k} \neq 0} } |\boldsymbol{k}|^{2r+3} e^{2\tau |\boldsymbol{k}|} |\hat{g}_{\boldsymbol{k}}|^2\Big)^{\frac{1}{2}}  \Big( \sum\limits_{\substack{\boldsymbol{j}\in \mathbb{Z}^3 \\ \boldsymbol{j} \neq 0} } |\boldsymbol{j}|^{2r+1} e^{2\tau |\boldsymbol{j}|} |\hat{f}_{\boldsymbol{j}}|^2\Big)^{\frac{1}{2}}  \nonumber \\
&&\hskip-.18in
\times \sup\limits_{\boldsymbol{k}\in \mathbb{Z}^3}\Big(\sum\limits_{\substack{\boldsymbol{j}\in \mathbb{Z}^3 \\ \boldsymbol{j} \neq 0} } |\boldsymbol{j}+\boldsymbol{k}|^{2r+1} e^{2\tau |\boldsymbol{j}+\boldsymbol{k}|}|\hat{h}_{-(\boldsymbol{j}+\boldsymbol{k})}|^2\Big)^{\frac{1}{2}}  \nonumber\\
&&\hskip-.55in 
\leq \tau C_r \|A^{r+\frac{1}{2}} e^{\tau A} f\| \|A^{r+\frac{3}{2}} e^{\tau A} g\| \|A^{r+\frac{1}{2}} e^{\tau A} h\|.
\end{eqnarray*}
Therefore, $I_1$ satisfies the desired estimates. The proof is completed.
\end{proof}
Finally, we sketch the proof of Lemma \ref{lemma-difference-type3}.
\begin{proof}(proof of Lemma \ref{lemma-difference-type3})
Similar to the proof of Lemma \ref{lemma-difference-type4}, we have 
\begin{eqnarray*}
&&\hskip-.8in
I:=\Big|\Big\langle A^r e^{\tau A} \Big( (\int_0^z \nabla\cdot f(\boldsymbol{x}',s)ds) \partial_z g  \Big), A^r e^{\tau A} h  \Big\rangle -  \Big\langle  (\int_0^z \nabla\cdot f(\boldsymbol{x}',s)ds)A^r e^{\tau A} \partial_z g , A^r e^{\tau A} h  \Big\rangle\Big| \nonumber\\
&&\hskip-.62in
= \Big| \Big\langle (\int_0^z \nabla\cdot f(\boldsymbol{x}',s)ds) \partial_z g, A^r e^{\tau A} H \Big\rangle - \Big\langle  (\int_0^z \nabla\cdot f(\boldsymbol{x}',s)ds)A^r e^{\tau A} \partial_z g ,  H  \Big\rangle\Big|. \nonumber\\
&&\hskip-.62in
\leq C \sum\limits_{\substack{\boldsymbol{j}+\boldsymbol{k}+\boldsymbol{l}=0\\    j_3, k_3, \boldsymbol{j}' \neq 0}} \frac{1}{|j_3|} |\hat{f}_{\boldsymbol{j}}||\hat{g}_{\boldsymbol{k}}||\hat{H}_{\boldsymbol{l}}| |\boldsymbol{j}'||\boldsymbol{k}| \Big| |\boldsymbol{l}|^r e^{\tau|\boldsymbol{l}|} - |\boldsymbol{k}|^r e^{\tau|\boldsymbol{k}|} \Big| \nonumber\\
&&\hskip-.6in
+ C \sum\limits_{\substack{\boldsymbol{j}'+\boldsymbol{k}'+\boldsymbol{l}'=0\\  k_3+l_3 = 0 \\  j_3, k_3, \boldsymbol{j}' \neq 0}}\frac{1}{|j_3|}|\hat{f}_{\boldsymbol{j}}||\hat{g}_{\boldsymbol{k}}||\hat{H}_{\boldsymbol{l}}| |\boldsymbol{j}'||\boldsymbol{k}| \Big| |\boldsymbol{l}|^r e^{\tau|\boldsymbol{l}|} - |\boldsymbol{k}|^r e^{\tau|\boldsymbol{k}|} \Big|:= I_1 + I_2.
\end{eqnarray*}

For $I_1$, since $\boldsymbol{j}+\boldsymbol{k}+\boldsymbol{l}=0$, we use \eqref{lemma-inequality} with $\xi = |\boldsymbol{l}|$, $\eta = |\boldsymbol{k}|$ and $|\xi - \eta| = \Big||\boldsymbol{l}|- |\boldsymbol{k}| \Big|\leq \Big|-\boldsymbol{l}-\boldsymbol{k}\Big| = |\boldsymbol{j}|$, to conclude
\begin{equation}\label{I1-type3}
    I_1 \leq C_r \sum\limits_{\substack{\boldsymbol{j}+\boldsymbol{k}+\boldsymbol{l}=0\\    j_3, k_3, \boldsymbol{j}' \neq 0}} \frac{1}{|j_3|} |\hat{f}_{\boldsymbol{j}}||\hat{g}_{\boldsymbol{k}}||\hat{H}_{\boldsymbol{l}}| |\boldsymbol{j}'||\boldsymbol{j}||\boldsymbol{k}|  \Big( |\boldsymbol{k}|^{r-1} + |\boldsymbol{j}|^{r-1} + \tau(|\boldsymbol{k}|^{r} + |\boldsymbol{j}|^{r})e^{\tau|\boldsymbol{k}|}e^{\tau|\boldsymbol{j}|} \Big).
\end{equation}

For $I_2$, since $(\boldsymbol{j}',0)+\boldsymbol{k}+\boldsymbol{l}=0$, we use \eqref{lemma-inequality} with $\xi = |\boldsymbol{l}|$, $\eta = |\boldsymbol{k}|$ and $|\xi - \eta| = \Big||\boldsymbol{l}|- |\boldsymbol{k}| \Big|\leq \Big|-\boldsymbol{l}-\boldsymbol{k}\Big| = |\boldsymbol{j}'|$, to obtain
\begin{eqnarray}\label{I2-type3}
I_2 \leq C_r\sum\limits_{\substack{\boldsymbol{j}'+\boldsymbol{k}'+\boldsymbol{l}'=0\\  k_3+l_3 = 0 \\  j_3, k_3, \boldsymbol{j}' \neq 0}}\frac{1}{|j_3|}|\hat{f}_{\boldsymbol{j}}||\hat{g}_{\boldsymbol{k}}||\hat{H}_{\boldsymbol{l}}| |\boldsymbol{j}'|^2|\boldsymbol{k}| \Big( |\boldsymbol{k}|^{r-1} + |\boldsymbol{j}'|^{r-1} + \tau(|\boldsymbol{k}|^{r} + |\boldsymbol{j}'|^{r})e^{\tau|\boldsymbol{k}|}e^{\tau|\boldsymbol{j}|} \Big).
\end{eqnarray}

Observe that the difference between the sums in the right-hand sides of (\ref{I1-type3}) and (\ref{I1-type4}) is manifested in the factors $|\boldsymbol{j}'||\boldsymbol{j}||\boldsymbol{k}|$ and  $|\boldsymbol{j}'||\boldsymbol{k}|^2$, and between (\ref{I2-type3}) and (\ref{I2-type4}) is manifested in the factors $|\boldsymbol{j}'|^2|\boldsymbol{k}|$ and  $|\boldsymbol{j}'||\boldsymbol{k}|^2$. Therefore, one can follow the estimates of $I_1$ in (\ref{I1-type4}) and $I_2$ in (\ref{I2-type4}), and obtain that $I_1$ in (\ref{I1-type3}) and $I_2$ in (\ref{I2-type3}) satisfy the desired bound in Lemma \ref{lemma-difference-type3}.

\end{proof}

\noindent
\section*{Acknowledgments}
The authors would like to thank the anonymous referees for carefully reading the paper and for their useful remarks. Q. Lin would like to thank University of Victoria for the kind and warm hospitality where part of this work was completed, and would like to thank Xin Liu for interesting discussions. The work of E.S. Titi\ was supported in part by the Einstein Stiftung/Foundation - Berlin, through the Einstein Visiting Fellow Program. The work of S. Ibrahim was supported by NSERC grant (371637-2019). The work of T. Ghoul was supported by SITE (Center for Stability, Instability, Turbulence and Experiments).

\end{document}